\documentclass[reqno,11pt]{amsart}
\usepackage{amsthm,amsmath,amsfonts,amssymb}
\usepackage{color}
\usepackage{bold-extra} 

 \usepackage[pagebackref]{hyperref}
\hypersetup{
  colorlinks=true,
  citecolor = blue, 
  linkcolor= blue,
}

\usepackage{enumerate}
\usepackage{enumitem} 
\setlist[enumerate,1]  {label={\rm (\roman*)}, leftmargin=1.5em}   
\setlist{noitemsep} 

\usepackage{mathrsfs} 

\usepackage{etoolbox}
\patchcmd{\section}{\normalfont}{\bfseries}{}{}
\makeatletter
\renewcommand{\@secnumfont}{\bfseries}
\makeatother

\usepackage[all]{xy}
\usepackage{pb-diagram,pb-xy} 

\usepackage{mdframed}
\usepackage{xcolor}
{\vspace{0,5\baselineskip}
  \begin{mdframed}[backgroundcolor=lime]}
  {\end{mdframed}
  \vspace{0,5\baselineskip}
}
{\vspace{0,5\baselineskip}
  \begin{mdframed}[backgroundcolor=pink]}
  {\end{mdframed}
  \vspace{0,5\baselineskip}
}
 \usepackage{version}
\excludeversion{reserve} 
\excludeversion{reserve2e}

 \usepackage{color,soul}

\usepackage{soul}
\usepackage{tikz}
\usetikzlibrary{calc}
\usetikzlibrary{decorations.pathmorphing}

\makeatletter

\newcommand{\defhighlighter}[3][]{%
  \tikzset{every highlighter/.style={color=#2, fill opacity=#3, #1}}%
}

\defhighlighter{yellow}{.5}

\newcommand{\highlight@DoHighlight}{
  \fill [ decoration = {random steps, amplitude=1pt, segment length=15pt}
        , outer sep = -15pt, inner sep = 0pt, decorate
        , every highlighter, this highlighter ]
        ($(begin highlight)+(0,8pt)$) rectangle ($(end highlight)+(0,-3pt)$) ;
}

\newcommand{\highlight@BeginHighlight}{
  \coordinate (begin highlight) at (0,0) ;
}

\newcommand{\highlight@EndHighlight}{
  \coordinate (end highlight) at (0,0) ;
}

\newdimen\highlight@previous
\newdimen\highlight@current

\DeclareRobustCommand*\highlight[1][]{%
  \tikzset{this highlighter/.style={#1}}%
  \SOUL@setup
  \def\SOUL@preamble{%
    \begin{tikzpicture}[overlay, remember picture]
      \highlight@BeginHighlight
      \highlight@EndHighlight
    \end{tikzpicture}%
  }%
  \def\SOUL@postamble{%
    \begin{tikzpicture}[overlay, remember picture]
      \highlight@EndHighlight
      \highlight@DoHighlight
    \end{tikzpicture}%
  }%
  \def\SOUL@everyhyphen{%
    \discretionary{%
      \SOUL@setkern\SOUL@hyphkern
      \SOUL@sethyphenchar
      \tikz[overlay, remember picture] \highlight@EndHighlight ;%
    }{%
    }{%
      \SOUL@setkern\SOUL@charkern
    }%
  }%
  \def\SOUL@everyexhyphen##1{%
    \SOUL@setkern\SOUL@hyphkern
    \hbox{##1}%
    \discretionary{%
      \tikz[overlay, remember picture] \highlight@EndHighlight ;%
    }{%
    }{%
      \SOUL@setkern\SOUL@charkern
    }%
  }%
  \def\SOUL@everysyllable{%
    \begin{tikzpicture}[overlay, remember picture]
      \path let \p0 = (begin highlight), \p1 = (0,0) in \pgfextra
        \global\highlight@previous=\y0
        \global\highlight@current =\y1
      \endpgfextra (0,0) ;
      \ifdim\highlight@current < \highlight@previous
        \highlight@DoHighlight
        \highlight@BeginHighlight
      \fi
    \end{tikzpicture}%
    \the\SOUL@syllable
    \tikz[overlay, remember picture] \highlight@EndHighlight ;%
  }%
  \SOUL@
}
\makeatother

\newcommand{\R}{{\mathbb R}}
\newcommand{\N}{{\mathbb N}}

\newcommand{\eps}{\varepsilon}
\newcommand{\D}{\displaystyle}

\newcommand{\Car}{\mbox{$\mathcal{X}$}} 

\newcommand{\mydef}{\ensuremath{\stackrel{\text{def}}{:=}}}

\newtheorem{theorem}{Theorem}[section]

\newtheorem{remark}[theorem]{Remark}
\newtheorem{lemma}[theorem]{Lemma}
\newtheorem{definition}[theorem]{Definition}
\newtheorem{proposition}[theorem]{Proposition}
\newtheorem{corollary}[theorem]{Corollary}

\numberwithin{equation}{section}

\renewenvironment{proof}[1][Proof]{\noindent \textbf{#1.} }
{\  \rule{0.5em}{0.5em}\par \medskip}

\title{
Evolution equation with fractional Schr\"odinger operators:
monotonicity and exponential decay of solutions  in Morrey spaces
}

\thanks{Partially supported by Project PID2022-137074NB-I00, MINECO, Spain.
\newline{$\phantom{ai}$ Key words and phrases: fractional Schr\"odinger operator, fractional partial differential equations, linear evolution equations, positive solutions, stability}.
\newline{$\phantom{ai}$ Mathematical Subject Classification
2020:\ 35J10, 35R11,47D06,35B09,35B35}.
}

\thanks{${}^{*}$Partially supported by Severo
Ochoa Grant  CEX2023-001347-S funded by MICIU/AEI/10.13039/501100011033.}

\begin{document}

\date{\bf \today}
\maketitle

\begin{center}
{\sc Jan W. Cholewa}${}^1\ $  \ {\sc and} \ {\sc Anibal Rodriguez-Bernal}${}^\text{2}$
\end{center}

\makeatletter
\begin{center}
${}^{1}$Institute of
Mathematics\\
University of Silesia in Katowice\\ 40-007 Katowice, Poland\\ {E-mail:
jan.cholewa@us.edu.pl}
\\ \mbox{}
\\
${}^{2}$Departamento de Análisis Matemático y  Matem\'atica Aplicada\\ Universidad
  Complutense de Madrid\\ 28040 Madrid, Spain\\ and \\
  Instituto de Ciencias Matem\'aticas \\
CSIC-UAM-UC3M-UCM${}^{*}$, Spain  \\ {E-mail:
arober@ucm.es}
\end{center}
\makeatother

\setcounter{tocdepth}{3}
\makeatletter
\def\l@subsection{\@tocline{2}{0pt}{2.5pc}{5pc}{}}
\def\l@subsubsection{\@tocline{2}{0pt}{5pc}{7.5pc}{}}
\makeatother

\begin{abstract}

We consider evolution equation with fractional Schr\"odinger operators
in Morrey spaces. We prove order preserving 
properties of the associated semigroup
in Morrey scale. We prove monotonicity of the semigroup with respect
to Morrey's potentials and give some precise estimates of its
exponential growth. We show that Arendt and Batty's type condition on
the potential  is
necessary for exponential decay of Morrey's norms of the semigroup and
find a large class of dissipative potentials for which it is also
sufficient. 
\end{abstract}

\section{Introduction}

In this paper we consider some fractional Schr\"odinger evolution
problems of the form
\begin{equation}\label{eq:evol-problem-intro}
  \begin{cases}
    u_{t} +  A_0^{\mu} u = V(x) u , \quad  x\in \R^{N}, \ t>0
    \\
u(x,0)= u_{0}(x), \quad  x\in \R^{N}
  \end{cases}
\end{equation}
where $A_0 = -\sum_{|\zeta|= 2} a_\zeta \partial^\zeta$
is a second order uniformly
elliptic operator with constant real coefficients and
$0< \mu \leq 1$, so $A_0^{\mu} $ describes fractional diffusion
phenomena. By a linear change of coordinates we can and shall 
reduce ourselves to the case     $A_0 = -\Delta$ without any loss of
generality.

Here we consider that   both the initial data $u_{0}$ and the potential $V$  belong  to
suitable Morrey space $M^{p,\ell}(\R^N)$, which for $1\leq p<\infty$,
$0\leq\ell\leq N$ consists of $\phi \in L_{loc}^{p} (\R^N)$ such that
there exists $C>0$ such that for all  $x_0\in\R^N$ and $R>0$
\begin{displaymath}
R^{\frac{\ell-N}{p}} \|\phi\|_{L^p(B(x_0,R))}\leq C
\end{displaymath}
endowed with the norm
\begin{displaymath}
\| \phi\|_{M^{p,\ell}(\R^N)} \mydef  \sup_{x_0\in \R^N\! ,\
  \! R>0} R^{\frac{\ell-N}{p}} \|\phi\|_{L^p(B(x_0,R))}<\infty .
  \end{displaymath}
We remark that if  $\ell=N$ then
$M^{p,N}(\R^N) = L^{p}(\R^N)$ for $1\leq p< \infty$ (taking $R\to \infty$), whereas if $\ell
=0$ then $M^{p,0}(\R^N)=L^\infty(\R^N)$ (taking $R\to 0$ and using Lebesgue's differentiation theorem).
We also set
$M^{\infty,\ell}(\R^N):=L^\infty(\R^N)$, $0\leq \ell \leq N$.
We moreover consider the space of Morrey measures, that we
denote by $\mathcal{M}^{\ell}(\R^N)$, as  the space of Radon measures $\mu$ satisfying
\begin{displaymath}
\|\mu\|_{\mathcal{M}^{\ell}(\R^N)} \mydef \sup_{x_{0} \in \R^N, R>0} \
  R^{\ell-N} \ |\mu|(B(x_{0},R))  <\infty.
\end{displaymath}
So for  any $0<\ell\leq N$, $M^{1,\ell}(\R^N) \subset
\mathcal{M}^{\ell} (\R^N)$  isometrically, where
$\mathcal{M}^{N} (\R^N)=\mathcal{M}_{BTV} (\R^N)$ is the space of Radon
measures with bounded total variation. Hence, Morrey spaces are spaces
of locally integrable functions (or measures), with some prescribed
distribution of its mass on balls of arbitrary radius and can be
considered as some sort of intermediate spaces between $L^{p}(\R^{N})$
and $L^{\infty}(\R^{N})$.

In this  paper we want to discuss two different type of important properties of
solutions of (\ref{eq:evol-problem-intro}) that are well known to hold
in more common functional settings like Lebesgue spaces, but which
seem not available in the literature for  Morrey ones. 
First, we want to prove  order preserving properties of solutions. That is, we want to show, on
the one hand, that if two initial data are ordered then the
corresponding solutions remain ordered for all times. Since
(\ref{eq:evol-problem-intro}) is linear, this amounts to prove that if
the initial data is non negative, so is the solution for all times. On
the other hand, we also want to prove that solutions are monotone with
respect to the potentials. That is, if we take two different ordered
potentials and a nonnegative initial data, then the corresponding
solutions remain ordered for all times. Second, we want to show that under suitable conditions on the
potentials, solutions converge to zero in suitable Morrey norms, as
$t\to \infty$. More precisely we want to obtain conditions on the
potentials such that all solutions converge to zero exponentially
fast. This last property requires a contribution from diffusion and a
precise balance between the negative part of $V$, which must be strong
enough to produce exponential decay, and some smallness of the
positive part of $V$, which tends to produce exponential growth of
solutions. 

This is carried out in the following way. In Section \ref{sec:homogeneous-lin-eq} we collect some properties
of Morrey spaces and some previous results on the solutions of the
homogeneous problem, that is, $V=0$, for initial data in Morrey
spaces. In particular, we recall some crucial  smoothing estimates of the
solutions between Morrey spaces, see (\ref{eq:estimates_Mpl-Mqs}) and
(\ref{eq:estimates_Mpl-Mpl}).

In Section \ref{sec:perturbed_equation_Morrey} we also recall some
previous results from \cite{C-RB-morrey_linear_perturbation},  on
the solutions of the  Schr\"odinger evolution problem
(\ref{eq:evol-problem-intro}) for initial data in Morrey
spaces.  These results require,  on the one hand,  that the potential
$V$ is \emph{admissible}, see (\ref{eq:admisible_potential}). This condition
requires $V$ to belong to a suitable Morrey space and in turn,
determines   the Morrey spaces in which the initial data can be
taken.
On the second hand, as embeddings between Morrey spaces are subtle,
considering two (or more)  different potentials is not a completely
obvious task, so for the sake of applicability, we quote some results
for the case of two admissible potentials in different Morrey
spaces. In such a case the more restrictive potential determines the
Morrey spaces of initial data. In these cases, we also have smoothing estimates of the
solutions between Morrey spaces, see
(\ref{eq:smoothing_one_potential}) and (\ref{eq:smoothing_two_potentials}). Finally, in this section we provide
different equivalent forms for the integral representation of
solutions, see Lemma \ref{lem:fixed-point-lem}, whose proof is given
in Appendix \ref{app:technicalities}. This result will be instrumental
in proving several  of the main results in the next sections.

Notice that the results in Sections \ref{sec:homogeneous-lin-eq} and
\ref{sec:perturbed_equation_Morrey} are stated 
for uniformly  elliptic $2m$ order operator of the form    $ A_0 =
\sum_{|\zeta|= 2m} (-1)^{m}a_\zeta \partial^\zeta$ 
 with constant real coefficients  $a_{\zeta}$,   in
 (\ref{eq:evol-problem-intro}). In the rest of the paper, we restrict 
 to second order operators and, as mentioned above, by a linear change
 of variables, we consider the case $A_{0}= -\Delta$.

In Section \ref{sec:order-preservation} we develop some of the main
results in this paper about the order preserving properties of
solutions. Namely, we first  prove that  for $0 <\mu \leq 1$
solutions of the pure diffusion problem, that is, 
of (\ref{eq:evol-problem-intro}) with $V=0$,  are order  preserving,  see Theorem
  \ref{thr:Smu-order-preserving}.
Using this, we show then that for admissible potentials, the complete problem
(\ref{eq:evol-problem-intro}) is also order preserving. We do this for
one potential in Theorem \ref{thm:order-pres-fo-general-Morrey-pot}
and for two in Theorem \ref{thm:order-pres-for-2-Morrey-pot}.

In Section \ref{sec:monotonicity} we show that solutions of
(\ref{eq:evol-problem-intro}) behave monotonically with respect to
potentials, see Theorem
\ref{thm:monotonicity-with-respect-to-potentials}. This and some
results in Section \ref{sec:perturbed_equation_Morrey} allows us to
obtain some exponential bounds on the solutions of
(\ref{eq:evol-problem-intro}) in terms of the sizes of the positive
parts of the admissible potentials, which are the ones pushing for
growth of solutions, see Theorem
\ref{thm:exponential-estimate-of-S{mu,V0,V1}}.

In Section \ref{sec:decay-indivual-solutions} we analyse the decay to
zero of individual solutions of (\ref{eq:evol-problem-intro}). First,
for the purely diffusive case, that is with $V=0$, the smoothing
estimates in Section \ref{sec:perturbed_equation_Morrey} imply that individual solutions will converge to zero, as
$t\to \infty$ in \emph{better} Morrey spaces than the one for the
initial data.  Convergence to zero in the same space of the initial
data requires some smallness conditions in the Morrey norm  on the initial data as  $|x|
\to \infty$, see Theorem
\ref{thm:more-pure-frac-dif-decay-of-individual-sol}. In particular,
this is always the case in Lebesgue spaces $L^{p}(\R^{N}) =
M^{p,N}(\R^{N})$. This smallness conditions are actually needed as
otherwise we can construct selfsimilar solutions with constant Morrey
norm, see Proposition \ref{prop:constant_norm_solutions}. 
From the monotonicity results in Section \ref{sec:monotonicity}, these
results translate to the case of dissipative potentials in 
(\ref{eq:evol-problem-intro}), that is,   $V\leq  0$, see Corollary
\ref{cor:from-{thm:more-pure-frac-dif-decay-of-individual-sol}}. In
this latter case we also show that in the subspace of initial data for
which the solutions decay to zero, either solutions decay
exponentially  or there exist arbitrarily slowly decaying
solutions, see Corollary \ref{cor:dicothomy_4_decaying_solutions}.

In Section \ref{sec:exponential-decay},  we address the stronger
property of exponential decay in time of solutions for the case of
dissipative potentials. First, we consider  the case of a bounded dissipative
potential, in Section \ref{sec:bounded-potential}. In this case
(\ref{eq:evol-problem-intro}) can be also considered with initial data
in Lebesgue spaces and in such a case, the condition on the potential for
exponential decay is much better understood. Actually exponential
decay in Lebesgue spaces is equivalent to the potential satisfy the  so called Arendt and Batty condition, see
(\ref{eq:Arendt-Batty-condition}) and Remark
\ref{rem:equivalent-2-AB-condition}. This condition is known to be
optimal for Schrödinger evolutions problems (that is, $\mu=1$) in
Lebesgue spaces, see \cite{Arendt-Batty}. For fractional
Schr\"odinger evolutions problems, that is, $0<\mu<1$,  it has been
recently considered in \cite{C-RB-exponential_decay} in Lebesgue
spaces for potentials in the larger class of  uniform spaces, see
Appendix \ref{sec:expon-decay-uniform}. Here we prove first that this
condition is necessary for exponential decay in Morrey spaces, see
Theorem \ref{thm:necessity-of-Arendt-Batty-condition}. Then we prove a
general relationship between the exponential bounds on solutions in
Morrey spaces and the exponential bounds in $L^{\infty}(\R^{N})$,
which in particular, imply again the the Arendt and Batty condition is
also sufficient for exponential decay in all Morrey spaces, see Theorem
\ref{thm:eponential-type-in-Morrey}. Moreover, as for the results in
\cite{C-RB-exponential_decay} in Lebesgue spaces, these results imply
that in case of exponential decay, one can find uniform exponential
rates of decay independent of the Morrey space. Also, if $\mu=1$ all
exponential bounds are the same for all Morrey spaces, while if
$0<\mu<1$ the exponential bound in $L^{\infty}(\R^{N})$ controls from
below and above the ones in Morrey spaces. The same was obtained in
\cite{C-RB-exponential_decay}  for Lebesgue spaces. See Theorem
\ref{thm:eponential-type-in-Morrey} and Remark
\ref{rem:exponential_type_in_Lebesgue}. Finally in Proposition
\ref{prop:exponential_decay_Linfty} we make a detailed analysis of the
decay rate in $L^{\infty}(\R^{N})$.

Then in Section \ref{sec:dissipative-admissible-morrey} we consider
the case of a general  dissipative admissible potential. Then we show again how
the Arendt and Batty condition provides exponential decay in Morrey
spaces, see Proposition
\ref{prop:decay_4_dissipative_morrey_potential}. Moreover, the
exponential decay is characterised by the alternative condition 
  \begin{displaymath}
 \inf\{\int_{\R^{N}} |(-\Delta)^{\frac{\mu}{2}} \phi|^2
    + \int_{\R^{N}} |V| |\phi|^2 \colon \phi\in C^\infty_c(\R^N), \
    \|\phi\|_{L^2(\R^N)}=1 \}  > 0
  \end{displaymath}
see Remark \ref{rem:exponential_type_in_Lebesgue}. 

Finally, in Section \ref{sec:expon-bounds-2-potentials}, for general
Morrey potentials,  we obtain an exponential bound on all solutions
that reflects the relative strengths of the negative  (hence
dissipative) and positive parts of the potential. This also allows us to obtain
some threshold value for the size of the positive part of the
potential for which we still have exponential decay of solutions,  see
Corollaries 
\ref{cor:growth_4_morrey_potential} and
\ref{cor:some-implication-of-Arendt-Batty-condition}. 

We have also included three appendixes where we state and prove some
auxiliary results needed in previous sections. As mentioned above, 
Appendix \ref{app:technicalities} gives a proof of the instrumental
result Lemma \ref{lem:fixed-point-lem}. Appendix
\ref{sec:result-exponential-type} proves a general result on the
exponential bound on solutions in two nested spaces. 
Finally, in Appendix \ref{sec:expon-decay-uniform} we prove that if we
consider (\ref{eq:evol-problem-intro}) with a bounded potential and
initial data in uniform spaces, then the exponential bounds on
solutions is again independent of the uniform space. In
particular, if $V$ satisfies the Arendt and Batty condition then
solutions decay exponentially to zero in all uniform spaces and at the
same exponential rate. We also prove that for dissipative bounded
potential, solutions of (\ref{eq:evol-problem-intro}) in uniform
spaces can be represented using an integral kernel. 

In this paper we denote by $c$ or $C$ generic constants that may
change from line to line, whose value is not important for the
results.

Also we will denote $A \sim B$ to denote quantities (norms or
functions, for example) such that there exist positive  constants $c_{1}, c_{2}$
such that $c_{1} A \leq B \leq c_{2}A$.

\section{Homogeneous fractional diffusion equation}
\label{sec:homogeneous-lin-eq}

It is known from \cite[Section 4]{C-RB-morrey_linear_perturbation} that given
a  uniformly  elliptic $2m$ order operator of the form
\begin{displaymath}
    A_0 =  \sum_{|\zeta|= 2m} (-1)^{m} a_\zeta \partial ^\zeta
    \quad
\text{ with constant real coefficients } \ a_{\zeta} ,
\end{displaymath}
the initial problem of the form
\begin{equation}\label{eq:linear-e-eq}
  \begin{cases}
    u_{t} + A_0^{\mu} u =0 ,  & x \in \R^{N}, \quad t>0, \\
    u(0,x)=u_0(x), & x\in \R^N,
  \end{cases}
\qquad \text{ where $0< \mu \leq 1$},
\end{equation}
defines a semigroup of solutions $\{S_{\mu}(t)\}_{t\geq0}$ in Morrey spaces $M^{p,\ell}(\R^N)$, $1\leq p \leq \infty$,  $0< \ell\leq N$, and in the space of measures $\mathcal{M}^{\ell}(\R^N)$. The semigroup $\{S_{\mu}(t)\}_{t\geq0}$ has a selfsimilar kernel, so for $u_0 \in M^{p,\ell}(\R^N)$, $1\leq p \leq \infty$,  $0< \ell\leq N$
\begin{equation}\label{eq:Smu-kernel}
    S_{\mu}(t) u_{0}(x) = \int_{\R^{N}} k_{\mu}(t,x,y) u_{0}(y) \, dy ,
    \quad x\in \R^{N}, \ t>0
\end{equation}
where
$k_{\mu}(t,x,y) = \frac{1}{t^\frac{N}{2\mu}}
K_{\mu}\Big(\frac{x-y}{t^{\frac{1}{2\mu}}}\Big)$.
Also, given
$1\leq p, q \leq \infty$ and $0\leq s\leq \ell\leq N$ such that
$\frac{s}q\leq \frac{\ell}p$
\begin{displaymath}
(0,\infty)\times M^{p,\ell}(\R^N) \ni (t,u_0) \to S_{\mu}(t) u_0 \in M^{q,s}(\R^{N})
\ \text{ is continuous},
 \end{displaymath}
and satisfies the smoothing estimates
  \begin{equation} \label{eq:estimates_Mpl-Mqs}
    \|S_{\mu}(t)\|_{{\mathcal L}(M^{p,\ell}(\R^N), M^{q,s} (\R^N))} =
    \frac{c}{ t^{\frac{1}{2m\mu}(\frac\ell{p}-\frac{s}q)}}, \quad t>0
  \end{equation}
    for some constant $c=c_{\mu,p,\ell,q,s}$. All these results remain
    true if for $p=1$ we replace $M^{1,\ell}(\R^N)$ by
    $\mathcal{M}^{\ell}(\R^N)$ and (\ref{eq:Smu-kernel}) reads now
    \begin{displaymath}
          S_{\mu}(t) u_{0}(x) = \int_{\R^{N}} k_{\mu}(t,x,y) \, du_{0}(y),
    \quad x\in \R^{N}, \ t>0 . 
    \end{displaymath}

    For any $ 1\leq p \leq \infty$, $0< \ell\leq N$ and $u_0\in M^{p,\ell}(\R^N)$ we moreover have
\begin{equation}\label{eq:behavior-at-0-of-S{mu}(t)}
  S_\mu(t)u_0 \to u_0 \ \text{ as } \ t\to 0^+, \quad \mbox{in
    $L^p_{loc}(\R^N)$}
\end{equation}
and $\{S_{\mu}(t)\}_{t\geq0}$ is analytic in $M^{p,\ell}(\R^N)$ with
sectorial generator provided that $0<\mu<1$, or $\mu=1$ and
$1< p \leq \infty$. For all the results above see \cite[Section
4]{C-RB-morrey_linear_perturbation} and \cite{C-RB-linear_morrey}.
As for initial data in the Morrey space of measures we have that if
$\mu\in \mathcal{M}^{\ell}(\R^N)$, $\ell\in[0,N]$, then
$ S_\mu (t)\mu \to \mu$ as $t\to 0^{+}$  in the sense of measures, that is,
\begin{displaymath}
\int_{\R^N} \phi  S_\mu (t)\mu \, dy \to \int_{\R^N} \phi \, d\mu \ \text{ as } \ t\to0^+
\ \text{ for each } \ \phi \in C_c(\R^N) ,
\end{displaymath}
see \cite[Proposition 3.4(ii)]{C-RB-morrey_linear_perturbation}.

The dotted Morrey space $\dot M^{p,\ell}(\R^N)$ is defined
as the subspace of $M^{p,\ell}(\R^N)$ in which translations
$\tau_y\phi(\cdot)=\phi(\cdot-y)$ in $\R^N$ are continuous in the
sense that
\begin{displaymath}
\tau_y \phi -\phi \to 0 \ \text{ as } \ y\to 0, \quad \text{ in } \ M^{p,\ell}(\R^N),
\end{displaymath}
and in particular  for $p=\infty$ we have   $\dot{M}^{\infty,\ell} (\R^N)= BUC(\R^N)$.

Then we have that for $1\leq p\leq \infty$ and $0<\ell \leq N$ and  $u_0\in \dot
M^{p,\ell}(\R^N)$
\begin{equation}
  \label{eq:behavior-at-0-of-S{mu}(t)-for-dotted}
S_\mu(t)u_0 \to u_0 \ \text{ as } \ t\to 0^+ \quad \text{ in } \
    M^{p,\ell}(\R^N)  ,
\end{equation}
see \cite[Proposition 4.5]{C-RB-morrey_linear_perturbation}  and
Remark \ref{rem:Smu-is-C0} below.
We also have for $1\leq  p\leq \infty$ and $0<\ell\leq N$ that
\begin{displaymath}
  S_\mu(t) ( \dot M^{p,\ell}(\R^N) ) \subset \dot M^{p,\ell}(\R^N),
  \quad t>0 ,
\end{displaymath}
see \cite[Proposition 4.1(iv)]{C-RB-morrey_linear_perturbation}. These
imply that  the semigroup $\{S_\mu(t)\}_{t\geq0}$  is strongly continuous in the dotted Morrey spaces.

Both undotted and dotted Morrey spaces possess suitable imbedding properties; namely if $1\leq p\leq q \leq \infty$,   $0\leq \ell \leq s \leq N$ satisfy $\frac{\ell}p=\frac{s}q$ then
  \begin{equation}
    \label{eq:summary-undotted-Morrey-2}
    M^{q,s}(\R^N)\subset M^{p,\ell}(\R^N)
  \end{equation}
where the spaces on both sides can be simultaneously replaced by the dotted ones
  (see \cite[Corollary 2.1]{C-RB-linear_morrey}).

Following the approach in \cite[Section 11 of Chapter IX]{1978_Yosida}
for a general bounded $C^0$ semigroup, as in  the case of
the dotted Morrey spaces, see also \cite[Appendix]{C-RB-scaling1}, and following \cite[Appendix
B]{C-RB-linear_morrey} in the case of the undotted Morrey spaces,
the semigroup $\{S_\mu(t)\}_{t\geq0}$ can be computed using
$\{S_1(t)\}_{t\geq0}$ as
\begin{equation}\label{eq:S0alpha(t)-in-MU}
S_{\mu}(t)  = \int_{0}^{\infty} f_{t,\mu}(s) S_1(s)  \, ds, \quad t> 0,
\end{equation}
where
\begin{displaymath}
0\leq   f_{t,\mu} (\lambda) =
  \begin{cases}
    \frac{1}{2\pi i} \D \int_{\sigma -i \infty}^{\sigma + i \infty}
    e^{z\lambda -tz^{\mu}} \, dz & \lambda \geq0,
    \\
    0 & \lambda <0
  \end{cases}
\end{displaymath}
(with $\sigma >0$) is continuous (see the expression
\cite[(17), p. 263]{1978_Yosida}), and the branch for $z^{\mu}$ is chosen such that
$Re(z^{\mu})>0$ if $Re(z)>0$. Also $\int_0^\infty f_{t,\mu}(s) \, ds
=1$, see \cite[(14), p. 262]{1978_Yosida}.

\begin{remark}\label{rem:Smu-is-C0}

From \cite[Section 11 of Chapter IX]{1978_Yosida}, if the semigroup
$\{S_{1}(t)\}_{t\geq0}$ is  strongly continuous so is
$\{S_\mu(t)\}_{t\geq0}$ given by  (\ref{eq:S0alpha(t)-in-MU}).

This is the case in the dotted Morrey spaces $\dot M^{p,\ell}(\R^N)$,
$1\leq  p\leq \infty$ and $0<\ell\leq N$ and  in particular in $BUC(\R^N)$. 
\end{remark}

We now show that the semigroup $\{S_{\mu}(t)\}_{t\geq0}$ smooths
solutions into  dotted Morrey spaces.

\begin{lemma}\label{lem:smoothing-to-dotted-spaces}
Assume $0<\mu \leq 1$ and $u_0\in M^{p,\ell}(\R^N)$ for some $1\leq p
\leq \infty$ and $0<\ell\leq N$ or $u_{0}\in
\mathcal{M}^{\ell}(\R^N)$.

If $1< q \leq \infty$ and $0<s\leq \ell$ satisfy $\frac{s}{q}\leq \frac{\ell}{p}$ then
\begin{equation}\label{eq:Smu(t)-in-dot{M}qs}
S_\mu(t) u_0\in \dot{M}^{q,s}(\R^N), \quad t>0 .
\end{equation}

\end{lemma}

\begin{proof}
We consider two  cases.

\smallskip

\noindent
\fbox{Case $1< q< \infty$}
From (\ref{eq:estimates_Mpl-Mqs})  we have
$S_\mu\big(\frac{t}{2}\big) u_0,
S_\mu(t)
u_0\in L^\infty(\R^N)\cap M^{q,s}(\R^N)$.
Using  \cite[Theorem 5.1(i), Theorem 1.1(vii)]{C-RB-linear_morrey} we
get that taking first order derivatives  $D^{1}
S_\mu\big(\frac{t}{2}\big) S_\mu\big(\frac{t}{2}\big) u_0=
D^{1}S_\mu(t)u_0\in M^{q,s}(\R^N)$. Applying then
\cite[Proposition 2.2]{C-RB-linear_morrey} we get (\ref{eq:Smu(t)-in-dot{M}qs}).

\medskip

\noindent
\fbox{Case  $q = \infty$}
We first collect some properties of
$S_{1}(t)$.  First, from  (\ref{eq:estimates_Mpl-Mqs}), for any $t>0$,
$S_{1}(t)$ maps Morrey spaces into $L^\infty(\R^N)$. Then from
\cite[Proposition 4.9]{C-RB} $S_1 (t) $ takes
$L^\infty(\R^N)$ into $BUC(\R^N)$ for $t>0$.   This, in particular, proves
(\ref{eq:Smu(t)-in-dot{M}qs}) for $\mu =1$.
Thus, from estimates (\ref{eq:estimates_Mpl-Mqs}),  $\{S_1(t)\}_{t\geq0}$ is a
bounded semigroup in $BUC(\R^N)$ and strongly continuous  by \cite[(4.6)]{C-RB}.
Also,
this and the semigroup property implies that for $v_{0} \in
L^\infty(\R^N)$ the function $(0,\infty) \ni t \mapsto S_{1}(t) v_{0} \in
BUC(\R^N)$ is continuous.

For $0<\mu <1$ these results imply in particular that for $v_{0}\in
L^\infty(\R^N)$ and $t>0$, $S_{\mu}(t) v_{0} $ is given by
(\ref{eq:S0alpha(t)-in-MU}) and $S_{\mu}(t) v_{0} \in BUC(\R^N)$.

Now, if $u_0$ is as in the statement and $t>0$,
(\ref{eq:estimates_Mpl-Mqs}) implies
$v_0=S_\mu\big(\frac{t}{2}\big) u_0\in L^\infty(\R^N)$ and then
$S_\mu\big(\frac{t}{2}\big) v_0 =S_\mu(t) u_0 \in BUC(\R^N)$.
\end{proof}

As mentioned in the introduction, for second order elliptic operator,
by a linear change of variables we can always assume, without loss of
generality, that
\begin{displaymath}
A_0=-\Delta
\end{displaymath}
and then,  for $\mu=1$,  the kernel in (\ref{eq:Smu-kernel}) is
\begin{equation}\label{eq:kernel-heat-smgp}
0\leq k_1(t,x,y)= \frac{1}{t^{\frac{N}{2}}} K_{1}
    \left(\frac{x-y}{t^{\frac{1}{2}}} \right)= \frac{1}{(4\pi
      t)^\frac{N}{2}} e^{-\frac{|x-y|^2}{4t}} ,
\end{equation}
whereas for $0<\mu<1$
\begin{equation}\label{eq:kernel-fractional-heat-smgp}
  0\leq k_{\mu}(t,x,y)= \frac{1}{t^{\frac{N}{2\mu}}} k_{0,\mu}
    \left(\frac{x-y}{t^{\frac{1}{2\mu}}} \right)  \sim
\frac{1}{t^\frac{N}{2\mu}}
H_{\mu}\left(\frac{x-y}{t^{\frac{1}{2\mu}}}\right),
\end{equation}
where $H_{\mu}(\cdot)=  \min\left\{1, \frac{1}{|\cdot|^{N+2\mu}}\right\}
\sim I_{\mu}(\cdot) = \frac{1}{(1+|\cdot|^{2})^{\frac{N+2\mu}{2}}}$, see \cite[Section 6]{C-RB-scaling3-4},
\cite{2019Bogdan_fract_laplac_hardy,2017BonforteSireVazquez:_optimal-existence}.
Then $\{S_{\mu}(t)\}_{t\geq 0}$ in $M^{p,\ell}(\R^N)$ for $1\leq p
\leq \infty$, $0<\ell\leq N$   satisfies
(\ref{eq:estimates_Mpl-Mpl}) below.

\begin{lemma}\label{lem:about-contractions}

  Assume $0<\mu \leq 1$.
Then for any $1\leq p \leq \infty$ and $0<\ell\leq N$
  \begin{equation} \label{eq:estimates_Mpl-Mpl}
    \|S_{\mu}(t)\|_{{\mathcal L}(M^{p,\ell}(\R^N))}  =
    1, \quad t\geq 0 
  \end{equation}
and the same in $\mathcal{M}^{\ell}(\R^{N})$. 
  
\end{lemma}

\begin{proof}
If $\mu=1$ we get via \cite[(2.2) in Lemma 2.1]{K} that
(\ref{eq:estimates_Mpl-Mqs}) holds with $c=1$ when $p=q$ and
$\ell=s$ and for the spaces of measures  $\mathcal{M}^{\ell}(\R^{N})$. 

For $0<\mu <1$, for $u_0\in M^{p,\ell}(\R^N)$, $1\leq p
\leq \infty$, $0<\ell\leq N$ or $u_{0}\in \mathcal{M}^{\ell}(\R^{N})$,      $S_{\mu}(t) u_{0} $ is given by
(\ref{eq:S0alpha(t)-in-MU}) and then
\begin{displaymath}
\|S_\mu(t)u_0\|_{M^{p,\ell}(\R^N)} \leq \int_0^\infty f_{t,\mu} (s)
\|S_1(s)u_0\|_{M^{p,\ell}(\R^N)} \leq \|u_0\|_{M^{p,\ell}(\R^N)} ,
\quad t>0 .
\end{displaymath}
Therefore, using (\ref{eq:estimates_Mpl-Mqs}), we get
$\|S_{\mu}(t)\|_{{\mathcal L}(M^{p,\ell}(\R^N))} =c \leq 1$, but this
implies that  $c=1$ since
otherwise we will get  exponential decay of $
\|S_{\mu}(t)\|_{{\mathcal L}(M^{p,\ell}(\R^N))} $.
\end{proof}

\section{Perturbed equation with  Morrey potentials}
\label{sec:perturbed_equation_Morrey}

In this section, exploiting the results for the \emph{unperturbed}
problem (\ref{eq:linear-e-eq}) in Section
\ref{sec:homogeneous-lin-eq}, we solve the \emph{perturbed} problem
involving a Morrey potential, or a  superposition of such potentials,
with initial data in suitable Morrey spaces. Regarding this, all
potentials we consider below are \emph{admissible potentials} in the
sense that they belong to Morrey spaces
$M^{\tilde p,\tilde \ell}(\R^N)$ such that
  \begin{equation} \label{eq:admisible_potential}
    1\leq \tilde p\leq \infty, \quad
   0<\tilde \ell \leq N \ \text{ with }
 \tilde \kappa=\frac{\tilde \ell}{2m\mu \tilde p}<1 .
  \end{equation}

We consider first the case of one potential, that is
\begin{equation}\label{eq:case-of-one-potential}
  \begin{cases}
    u_{t} + A_{0}^{\mu} u = V^{1}(x) u , \quad  x\in \R^{N}, \ t>0
    \\
u(x,0)= u_{0}(x), \quad  x\in \R^{N} .
  \end{cases}
\end{equation}
Then
 from \cite[Theorem 7.2]{C-RB-morrey_linear_perturbation} we have that
 if we consider an admissible potential
  \begin{displaymath}
V^{1}\in M^{p_1,\ell_1}(\R^N) \text{ with  } 1\leq p_1\leq \infty,
\quad    0<\ell_{1} \leq N \ \text{ and  }
 \kappa_1=\frac{\ell_{1}}{2m\mu p_{1}}<1,
\end{displaymath}
then for   $1\leq p \leq \infty$ and $0<\ell \leq \ell_1$,
  (\ref{eq:case-of-one-potential}) defines a semigroup of solutions 
  $\{S_{\mu,V^{1}}(t)\}_{t\geq0}$ in $M^{p,\ell}(\R^N)$ such that for
  $u_0\in M^{p,\ell}(\R^N)$, $u(t) \mydef   S_{\mu,V^{1}}(t)u_0 $
  satisfies
  \begin{equation} \label{eq:VCF_one_potential}
    u(t) = S_{\mu}(t) u_{0} + \int_{0}^{t}  S_{\mu}(t-s) V^{1} u(s) \, ds,
    \quad  t>0,
  \end{equation}
 \begin{equation}  \label{eq:SV_approach_initial_data}
    \lim_{t\to 0^+} \| u(t)- S_{\mu}(t)
    u_0\|_{M^{p,\ell}(\R^N)}= 0 .
  \end{equation} 
  Also,
  \begin{equation}\label{eq:Mp,ell-estimate-of-S{mu,V}(t)}
    \|S_{\mu,V^{1}}(t) \|_{\mathcal{L}(M^{p,\ell}(\R^N))} \leq C e^{\omega
      t}, \quad t\geq 0
  \end{equation}
  for some constants $C$, $\omega$. For  $1\leq q\leq \infty$ and
$0< s\leq \ell$ satisfying  $\frac{s}{q}\leq \frac{\ell}{p}$ we have,
for any $a>\omega$ and some constant $C$,
\begin{equation}\label{eq:smoothing_one_potential}
  \| S_{\mu,V^{1}}(t)\|_{\mathcal{L}(M^{p,\ell}(\R^N), M^{q,s}(\R^N))}
  \leq
  \frac{C e^{at}}{t^{\frac{1}{2m\mu}(\frac{\ell}{p}-\frac{s}{q})}},
  \quad t>0
\end{equation}
and
\begin{displaymath}
  (0,\infty)\times M^{p,\ell}(\R^N) \ni (t,u_0) \to S_{\mu,V^{1}}(t) u_0 \in M^{q,s}(\R^{N})
  \ \text{ is continuous}.
\end{displaymath}
For $p=1$ all these  remain  true if we replace $M^{1,\ell}(\R^N)$ by
$\mathcal{M}^{\ell}(\R^N)$.

For $1 \leq  p \leq \infty$ and $0<\ell \leq \ell_1$, by
(\ref{eq:behavior-at-0-of-S{mu}(t)-for-dotted}),     if $u_{0}\in \dot
{M}^{p,\ell}(\R^N)$ then $u(t) = S_{\mu, V^{1}}(t)u_0\to u_{0}$ in
$M^{p,\ell}(\R^N)$ as $t\to 0^{+}$.

  Moreover, when $p\geq p_1'$ the exponent in
  (\ref{eq:Mp,ell-estimate-of-S{mu,V}(t)}) can be taken as
  \begin{equation}\label{eq:growth_bound}
    \omega =c \| V^{1}\|_{M^{p_1,\ell_1}(\R^N)}^{\frac{1}{1-\kappa_{1}}},
  \end{equation}
  for some positive constant $c=c(p,\ell)$. 
    Moreover, if $p\geq p_1'$, or $q\geq p_1'$, then
    (\ref{eq:smoothing_one_potential}) holds with any $a>\omega$ as in
    (\ref{eq:growth_bound}).

Similarly to Lemma \ref{lem:smoothing-to-dotted-spaces} we now show
that the perturbed semigroup $\{S_{\mu,V^{1}}(t)\}_{t\geq0}$ also
smoothes solutions into dotted Morrey spaces.

\begin{lemma}\label{lem:perturbed-smoothing-to-dotted-spaces}
  Assume $0<\mu \leq 1$ and $V^{1} \in M^{p_1,\ell_1}(\R^N)$ is an
  admissible potential. Assume also that $u_0\in M^{p,\ell}(\R^N)$ for
  some $1\leq p \leq \infty$ and $0<\ell\leq \ell_1$ or $u_{0}\in
  \mathcal{M}^{\ell}(\R^{N})$ if $p=1$.

If $1<q \leq \infty$ and $0<s\leq \ell$ satisfy
$\frac{s}{q}\leq \frac{\ell}{p}$ then
\begin{displaymath}
S_{\mu,V^{1}}(t) u_0\in \dot{M}^{q,s}(\R^N), \quad t>0.
\end{displaymath}

In particular, if $1<p\leq \infty$,  $\{S_{\mu,V^{1}}(t)\}_{t\geq0}$ is a strongly continuous
semigroup in $\dot{M}^{p,\ell}(\R^N)$.

\end{lemma}
\begin{proof}
For $u_{0}$ as in the statement and $t>0$ fixed, from
(\ref{eq:smoothing_one_potential}) we have  that for all $1\leq \tilde{p}\leq \infty$ and
$0< \tilde{\ell}\leq \ell$ such that
$\frac{\tilde{\ell}}{\tilde{p}}\leq \frac{\ell}{p}$ we have
\begin{displaymath}
v_0=S_{\mu,V^{1}}(t)u_0\in M^{\tilde{p} ,\tilde{\ell}}(\R^N) \cap
M^{p,\ell}(\R^N),
\end{displaymath}
and  from  (\ref{eq:VCF_one_potential}), $v(t) = S_{\mu,V^{1}}(t)v_0 = S_{\mu,V^{1}}(2t)u_0$
satisfies
\begin{displaymath}
   v(t) = S_{\mu}(t) v_{0} + \int_{0}^{t}  S_{\mu}(t-\tau) V^{1}
   v(\tau) \, d\tau, 
\end{displaymath}
and
\begin{displaymath}
\|v(\tau) \|_{M^{\tilde{p} ,\tilde{\ell}} (\R^N)}
\leq  M(t) \|v_0\|_{ M^{\tilde{p} ,\tilde{\ell}}(\R^N)}, \quad 0\leq
\tau \leq t 
\end{displaymath}
with $M(t) = Ce^{|a|t}$. 

Assume $\tilde{p}\geq p_{1}'$ and define then  $1\leq z\leq \infty$ and $0\leq \nu \leq \ell_{1}$ by
\begin{displaymath}
\frac{1}{z}= \frac{1}{\tilde{p}}+\frac{1}{p_1}, \ \qquad \ \frac{\nu}{z}= \frac{\tilde{\ell}}{\tilde{p}}+\frac{\ell_1}{p_1}.
\end{displaymath}

Then,  multiplication properties in Morrey spaces, see \cite[Lemma
5.1]{C-RB-morrey_linear_perturbation},
 give
\begin{displaymath}
\begin{split}
\|V^{1} v(\tau) \|_{M^{z,\nu}(\R^N)}
&\leq \|V^{1}
\|_{M^{p_1,\ell_1}(\R^N)}\|v(\tau) \|_{M^{\tilde{p} ,\tilde{\ell}}(\R^N)}
\\
&
\leq C M(t) \|V^{1} \|_{M^{p_1,\ell_1}(\R^N)}  \|v_0\|_{
  M^{\tilde{p} ,\tilde{\ell}} (\R^N)}, \quad 0 <  \tau \leq t.
\end{split}
\end{displaymath}
On the other hand, when $s\leq \nu$ and $\frac{s}{q}\leq\frac{\nu}{z}$ from  (\ref{eq:estimates_Mpl-Mqs}) we get
\begin{displaymath}
    \|S_{\mu}(t-\tau)\|_{{\mathcal L}(M^{z,\nu}(\R^N), M^{q,s}
      (\R^N))} = \frac{c}{(t-
      \tau)^{\frac{1}{2m\mu}(\frac{\nu}{z}-\frac{s}{q})}}
=  \frac{c}{(t-
  \tau)^{\frac{1}{2m\mu}(\frac{\tilde{\ell}}{\tilde{p}}+\frac{\ell_1}{p_1}-\frac{s}{q})}}
, \ 0 < \tau<t.
\end{displaymath}

Then if we can take $\tilde{p}$, $\tilde{\ell}$ such that  $p_{1}' \leq \tilde{p}\leq \infty$ and
$0< \tilde{\ell}\leq \ell$,  $\frac{\tilde{\ell}}{\tilde{p}}\leq
\frac{\ell}{p}$ such that  $s\leq \nu$ and
$\frac{s}{q}\leq\frac{\nu}{z}$ and
additionally
\begin{displaymath}
  {\frac{1}{2m\mu}(\frac{\tilde{\ell}}{\tilde{p}} +\frac{\ell_1}{p_1}-\frac{s}{q})}
  <1
\end{displaymath}
then we would get
\begin{displaymath}
    \int_0^t  \|S_{\mu}(t-\tau) V^{1} v(\tau)\|_{M^{q,s}(\R^N)} \, d\tau
 <\infty.
\end{displaymath}

Since, from Lemma \ref{lem:smoothing-to-dotted-spaces}
\begin{displaymath}
  S_{\mu}(t) v_{0} \in \dot{M}^{q,s}(\R^N),   \qquad
  S_{\mu}(t-\tau) V^{1} v(\tau) \in \dot{M}^{q,s}(\R^N) , \quad
  \tau\in (0,t)
\end{displaymath}
which is a closed subspace of $M^{q,s}(\R^N)$,  then we would get
$v(t)= S_{\mu,V^{1}}(t)v_0 = S_{\mu,V^{1}}(2t)u_0 \in
\dot{M}^{q,s}(\R^N)$  and we get the result.

Now we show that we can take $\tilde{p}$, $\tilde{\ell}$ as above.

\noindent
\fbox{Case 1: $p_1' \leq q \leq \infty$}
In this case we take $\tilde{p}=q$ and $\tilde{\ell}=s$ and then
\begin{displaymath}
  \frac{1}{z}= \frac{1}{q}+\frac{1}{p_1}, \ \qquad \ \frac{\nu}{z}=
  \frac{s}{q}+\frac{\ell_1}{p_1} \geq  \frac{s}{q}
\end{displaymath}
and $ \frac{s}{z}= \frac{s}{q}+\frac{s}{p_1}\leq
\frac{s}{q}+\frac{\ell_{1}}{p_1} =  \frac{\nu}{z}$, so $s\leq \nu$
and  all conditions are satisfied since $\frac{\ell_1}{2m\mu p_1}  <1$
because $V^{1}$ is an admissible potential.

\noindent
\fbox{Case 2: $1<q<p_1'$}
In this case we take $\tilde{p}=p_{1}'$ and $\tilde{\ell}=s$ and then
\begin{displaymath}
  \frac{1}{z}= \frac{1}{p_1'}+\frac{1}{p_1}=1,  \qquad
  \nu=\frac{\nu}{z}= \frac{s}{p_1'}+\frac{\ell_1}{p_1}
\end{displaymath}
so $\frac{s}{q}<s=\frac{s}{p_1'}+\frac{s}{p_1}\leq
\frac{s}{p_1'}+\frac{\ell_{1}}{p_1} =\frac{\nu}{z}=\nu$,
\begin{displaymath}
\frac{\tilde{\ell}}{\tilde{p}}=   \frac{s}{p_1'}< \frac{s}{q}\leq \frac{\ell}{p},
\end{displaymath}
and
\begin{displaymath}
  {\frac{1}{2m\mu}(\frac{\tilde{\ell}}{\tilde{p}} +\frac{\ell_1}{p_1}-\frac{s}{q})}
 =    {\frac{1}{2m\mu}(\frac{s}{p_1'} +\frac{\ell_1}{p_1}-\frac{s}{q})}
 \leq \frac{\ell_1}{2m\mu p_1} <1
\end{displaymath}
since $V^{1}$ is an admissible potential.
\end{proof}

We now consider
\begin{equation}\label{eq:perturbed-problem}
  \begin{cases}
    u_{t} + A_{0}^{\mu} u = V^{0}(x)u + V^{1}(x)u , &  x\in \R^N , t>0,
    \\
    u(0,x)=u_0(x), &x\in \R^N
  \end{cases}
\end{equation}
with two admissible potentials
  \begin{equation}\label{eq:for-thm-with-2potentials}
V^{i}\in M^{p_i,\ell_i}(\R^N) \text{ with  } 1\leq p_i\leq \infty,
\quad    0<\ell_{0} \leq \ell_{1}\leq N \ \text{ and  }
 \kappa_i=\frac{\ell_{i}}{2m\mu p_{i}}<1 , \quad i=0,1 .
\end{equation}

From  \cite[Theorem 7.5]{C-RB-morrey_linear_perturbation} we have that
for   $1\leq p \leq \infty$ and $0<\ell \leq \ell_0$,  problem
(\ref{eq:perturbed-problem}) defines a semigroup of solutions 
$\{S_{\mu,\{V^{0}, V^{1}\}}(t)\}_{t\geq0}$  in  $M^{p,\ell}(\R^N)$ such
that for $u_0\in M^{p,\ell}(\R^N)$,  $u(t) \mydef S_{\mu,
  \{V^{0},V^{1}\}}(t)u_0$ satisfies
  \begin{equation} \label{eq:VCF_two_potentials}
u(t) = S_{\mu}(t) u_{0} + \int_{0}^{t}  S_{\mu}(t-s) V^{0} u(s) \, ds + \int_{0}^{t}  S_{\mu}(t-s) V^{1} u(s) \, ds,  \quad  t>0,
  \end{equation}
 \begin{displaymath}
    \lim_{t\to 0^+} \| u(t)- S_{\mu}(t) u_0\|_{M^{p,\ell}(\R^N)}= 0 .
  \end{displaymath}
  Also, there are constants $C$, $\omega$ such that
  \begin{equation}\label{eq:Mp,ell-estimate-of-S{mu,V0,V1}(t)}
\|S_{\mu, \{V^{0},V^{1}\}}(t)\|_{\mathcal{L}(M^{p,\ell}(\R^N))}\leq C e^{\omega t}, \quad t\geq 0.
  \end{equation}
For     $1\leq q\leq \infty$ and $0< s\leq \ell$ satisfy
$\frac{s}{q}\leq \frac{\ell}{p}$ we have, for any $a> \omega$ and some constant $C$, 
\begin{equation}\label{eq:smoothing_two_potentials}
  \| S_{\mu, \{V^{0}, V^{1}\}}(t)\|_{\mathcal{L}(M^{p,\ell}(\R^N), M^{q,s}(\R^N))}
  \leq
  \frac{Ce^{at}}{t^{\frac{1}{2m\mu}(\frac{\ell}{p}-\frac{s}{q})}},
  \quad t>0,
\end{equation}
and
\begin{displaymath}
  (0,\infty)\times M^{p,\ell}(\R^N) \ni (t,u_0) \to S_{\mu, \{V^{0}, V^{1}\}}(t) u_0 \in M^{q,s}(\R^{N})
  \ \text{ is continuous}.
\end{displaymath}
For $p=1$ this remains true if we replace $M^{1,\ell}(\R^N)$ by the
space of Morrey  measures $\mathcal{M}^{\ell}(\R^N)$.

For $1\leq p\leq \infty$ and $0<\ell\leq \ell_{0}$, by
(\ref{eq:behavior-at-0-of-S{mu}(t)-for-dotted}), if
$u_{0}\in \dot {M}^{p,\ell}(\R^N)$ then
$u(t) = S_{\mu, \{V^{0},V^{1}\}}(t)u_0\to u_{0}$ in $M^{p,\ell}(\R^N)$
as $t\to 0^{+}$.

Moreover, if $p\geq \max\{p_0',p_1'\}$ then (\ref{eq:Mp,ell-estimate-of-S{mu,V0,V1}(t)}) holds with $\omega$ satisfying
  \begin{equation} \label{eq:general_growth_bound}
    \omega =
    c\big(
    \| V^{0}\|_{M^{p_0,\ell_0}(\R^N)}^{\frac{1}{1-\kappa_{0}}}
    +
    \| V^{1}\|_{M^{p_1,\ell_1}(\R^N)}^{\frac{1}{1-\kappa_{1}}}
    \big)
  \end{equation}
for some positive constant $c=c(p,\ell)$.
  Moreover, if $p\geq \max\{p_0',p_1'\}$, or
  $q\geq \max\{p_0',p_1'\}$, then (\ref{eq:smoothing_two_potentials})
  holds with $a> \omega$ as in (\ref{eq:general_growth_bound}).

\begin{remark}
  \label{rem:adding_potentials}

  \begin{enumerate}
 \item If one potential is bounded, $V \in L^\infty(\R^N)$,
    because $M^{\infty,\ell}(\R^N)=L^\infty(\R^N)$ for $\ell\in[0,N]$,
    we have that $V$ is an admissible potential as in
    (\ref{eq:admisible_potential}) with $\tilde p = \infty$ and we
    can take any $\tilde \ell \in[0,N]$.

    In particular, if we have an admissible potential $V^{1}$ such
    that $\kappa_1=\frac{\ell_1}{2m\mu p_1}<1$ and
    $V^{0} \in L^\infty(\R^N)$ then
    (\ref{eq:for-thm-with-2potentials}) holds with $p_0=\infty$ and
    $\ell_0=\ell_1$.

  \item
    If the potentials are admissible and $p_{0}=p_{1}$ and
    $\ell_{0} = \ell_{1}$ then for $u_0\in M^{p,\ell}(\R^N)$ with
    $1\leq p \leq\infty$, $0<\ell\leq\ell_0$ we have
    \begin{displaymath}
      S_{\mu,\{V^{0}, V^{1}\}}(t)u_0= S_{\mu,  V^{0} + V^{1}}(t)u_0, \qquad t>0 .
    \end{displaymath}
    This follows from (\ref{eq:VCF_two_potentials}).
  \end{enumerate}
\end{remark}

As shown in \cite{C-RB-morrey_linear_perturbation}, the semigroup
$\{S_{\mu,\{V^{0}, V^{1}\}}(t)\}_{t\geq0}$  can be constructed as
different successive perturbations,  first of $\{S_{\mu}(t)\}_{t\geq0}$
and then of $\{S_{\mu,V^{1} }(t)\}_{t\geq0}$.  It was actually proved
in \cite[Theorem 6.8]{C-RB-morrey_linear_perturbation} that the
resulting solution family $\{S_{\mu,\{V^{0}, V^{1}\}}(t)\}_{t\geq0}$
is independent of the order the successive perturbations are applied.
In this direction, the  following technical results
imply  that we can obtain different integral representations of
$\{S_{\mu,\{V^{0}, V^{1}\}}(t)\}_{t\geq0}$  besides
(\ref{eq:VCF_two_potentials}). The proof is given in Appendix
\ref{app:technicalities}.

\begin{lemma} \label{lem:fixed-point-lem}
Assume $0<\mu \leq 1$ and admissible potentials as in
(\ref{eq:for-thm-with-2potentials}) and initial data $u_0\in
M^{p,\ell}(\R^N)$, $1\leq p \leq\infty$, $0<\ell\leq\ell_0$ or
$p=1$ and $u_{0}\in \mathcal{M}^{\ell}(\R^{N})$.

Then  for each  $T>0$, all of the following results hold.

\begin{enumerate}
\item
For $c\in \R$, the  function $e^{ct} S_{\mu,\{V^{0}, V^{1}\}}(t)u_0$
on $(0,T]$, coincides with the fixed point of the map
\begin{displaymath}
\begin{split}
\varphi\mapsto  {\mathcal F} (\varphi,u_0) (t)
= S_{\mu}(t) u_0 +  \int_0^t S_{\mu}(t-\tau)  V^{0} \varphi(\tau) \,  d\tau &+  \int_0^t S_{\mu}(t-\tau)  V^{1} \varphi(\tau)
\\
&
+ \int_0^t S_{\mu}(t-\tau)  c \varphi(\tau), \quad t\in(0,T],
\end{split}
\end{displaymath}
which is  a strict contraction in the set in
$\mathcal{K}_{T,K_0,\theta}^{p,\ell,w,\kappa}$ defined as
\begin{equation}\label{eq:space-for-fixed-point}
\mathcal{K}_{T,K_0,\theta}^{p,\ell,w,\kappa}=\{\psi\in L^{\infty}_{loc}((0,T], M^{w,\kappa}(\R^N)):
\sup_{t\in(0,T]} t^{\frac{1}{2m\mu}(\frac{\ell}{p}-\frac{\kappa}{w})}
e^{-\theta t} \|\psi(t) \|_{M^{w,\kappa}(\R^N)}\leq K_0  \}
\end{equation}
 with some $1\leq w \leq\infty$,
$0<\kappa\leq\ell$ satisfying $\frac{\kappa}{w}\leq \frac{\ell}{p}$
and all large enough $K_0,\theta>0$.
Furthermore, $w,\kappa$ can be chosen as
\begin{equation}
  \label{eq:omega-and-kappa}
  (w,\kappa)=
  \begin{cases} (p,\ell) & \text{ if }  \frac{1}{p_1}\leq \theta \\
    (\frac{1}{\theta}, \ell \theta) & \text{ if } \frac{1}{p_1}>
    \theta
    \end{cases} \quad \text{ for } \ \theta:=\min\big\{\frac{1}{p_0'},
    \frac{1}{p_1'}\big\},
\end{equation}
and so $w,\kappa$ can be chosen depending only on $p,\ell,p_0,p_1$.

\item
The  function $S_{\mu,\{V^{0}, V^{1}\}}(t)u_0$
on $(0,T]$,
coincides with the fixed point of the map
\begin{displaymath}
\varphi\mapsto  {\mathcal F} (\varphi,u_0) (t) = S_{\mu,V^{1}}(t)
u_0 +  \int_0^t S_{\mu,V^{1}}(t-\tau)  V^{0} \varphi(\tau) \,  d\tau,
\quad t\in(0,T],
\end{displaymath}
which is a strict contraction in the set
$\mathcal{K}_{T,K_0,\theta}^{p,\ell,w,\kappa}$  in part (i).

\item
If $V^{0}\in L^\infty(\R^N)$  then we take $\ell_{0}=\ell_{1}$ and the function
$S_{\mu,\{V^{0}, V^{1}\}}(t)u_0$ on $(0,T]$ coincides with the fixed point of the map
\begin{displaymath}
\varphi\mapsto  {\mathcal F} (\varphi,u_0) (t)= S_{\mu,V^{0}}(t) u_0 +  \int_0^t
S_{\mu,V^{0}}(t-\tau)  V^{1} \varphi(\tau) \,  d\tau, \quad t\in(0,T],
\end{displaymath}
which is a strict contraction in the set
$\mathcal{K}_{T,K_0,\theta}^{p,\ell,w,\kappa}$  in part (i).

\item
If $\tilde{V}^0\in M^{p_0,\ell_0}(\R^N)$ then the function
$S_{\mu,\{V^{0} + \tilde{V}^{0}, V^{1}\}}(t)u_0$ on $(0,T]$ coincides with the fixed point of the map
\begin{displaymath}
\varphi\mapsto  {\mathcal F} (\varphi,u_0) (t) = S_{\mu,\{V^{0},
  V^{1}\}}(t) u_0 +  \int_0^t S_{\mu,\{V^{0}, V^{1}\}}(t-\tau)
\tilde{V}^{0} \varphi(\tau) \,  d\tau, \ t\in(0,T],
\end{displaymath}
which is a strict contraction in the set
$\mathcal{K}_{T,K_0,\theta}^{p,\ell,w,\kappa}$  in part (i).

\item
    If $\tilde{V}^1\in M^{p_1,\ell_1}(\R^N)$ then the function
$S_{\mu,\{V^{0}, V^{1} + \tilde{V}^{1}\}}(t)u_0$ on $(0,T]$ coincides with the fixed point of the map
\begin{displaymath}
\varphi\mapsto  {\mathcal F} (\varphi,u_0) (t) = S_{\mu,\{V^{0},
  V^{1}\}}(t) u_0 +  \int_0^t S_{\mu,\{V^{0}, V^{1}\}}(t-\tau)
\tilde{V}^{1} \varphi(\tau) \,  d\tau, \ t\in(0,T],
\end{displaymath}
which is a strict contraction in the set
$\mathcal{K}_{T,K_0,\theta}^{p,\ell,w,\kappa}$  in part (i).

\end{enumerate}

\end{lemma}

We now use Lemma \ref{lem:fixed-point-lem} to prove the following useful fact about the perturbed semigroup.

\begin{proposition}
  \label{add_exponential_to_semigroup}
  Let $0<\mu \leq 1$ and $V^{1}\in M^{p_1,\ell_1}(\R^N)$ is an
  admissible potential as in (\ref{eq:admisible_potential}).

If $V^{0}\in L^\infty(\R^N)$ and $c\in\R$ then for any $u_0\in
M^{p,\ell}(\R^N)$ with $1\leq p \leq\infty$, $0<\ell\leq\ell_1$ or
$p=1$ and $u_{0}\in \mathcal{M}^{\ell}(\R^{N})$ 
\begin{displaymath}
S_{\mu,\{V^{0} + c, V^{1}\}}(t)u_0 = e^{ct} S_{\mu,\{V^{0}, V^{1}\}}(t)u_0 , \qquad t>0 .
\end{displaymath}

\end{proposition}
\begin{proof}
Given  $T>0$,  using part (i) in Remark \ref{rem:adding_potentials}  we take $p_{0}= \infty$ and  $\ell_0=\ell_1$  and using  part (iv) in Lemma
\ref{lem:fixed-point-lem}  we  see that $S_{\mu,\{V^{0} + c, V^{1}\}}(t)u_0$ in $(0,T]$ coincides
with the fixed point of $\varphi\mapsto  S_{\mu,\{V^{0}, V^{1}\}}(t)
u_0 +  \int_0^t S_{\mu,\{V^{0}, V^{1}\}}(t-\tau)  c \varphi(\tau) \,
d\tau$ in $\mathcal{K}_{T,K_0,\theta}^{p,\ell,w,\kappa}$.

On the other hand,  by the  smoothing estimate
(\ref{eq:smoothing_two_potentials}), the function $v(t)= e^{ct}
S_{\mu,\{V^{0}, V^{1}\}}(t)u_0$ in $(0,T]$ belongs to
$\mathcal{K}_{T,K_0,\theta}^{p,\ell,w,\kappa}$ and
\begin{displaymath}
S_{\mu,\{V^{0}, V^{1}\}}(t) u_0 +  \int_0^t S_{\mu,\{V^{0},
  V^{1}\}}(t-\tau)  c v(\tau) \,
d\tau  = S_{\mu,\{V^{0}, V^{1}\}}(t) u_0 +  \big( \int_0^t c e^{c\tau}
\, d\tau  \big) S_{\mu,\{V^{0}, V^{1}\}}(t )u_0 = v(t)
\end{displaymath}
and  we get the result.
\end{proof}

We next show that,  as in the case of a  single potential  in Lemma
\ref{lem:perturbed-smoothing-to-dotted-spaces}, the semigroup
$\{S_{\mu,\{V^{0}, V^{1}\}}(t)\}_{t\geq0}$ also smooths solutions into
 dotted Morrey spaces.

\begin{lemma}

Assume $0<\mu \leq 1$, and admissible potentials as in (\ref{eq:for-thm-with-2potentials}).
Assume also that $u_0\in M^{p,\ell}(\R^N)$ for some $1\leq p \leq
\infty$ and $0<\ell\leq \ell_0$ or
$p=1$ and $u_{0}\in \mathcal{M}^{\ell}(\R^{N})$.

If $1<q \leq \infty$ and $0<s\leq \ell$ satisfy
$\frac{s}{q}\leq \frac{\ell}{p}$ then
\begin{displaymath}
S_{\mu,\{V^{0},V^{1} \}}(t) u_0\in \dot{M}^{q,s}(\R^N), \quad t>0.
\end{displaymath}

In particular, if $1<p\leq \infty$,  $\{S_{\mu,,\{V^{0},V^{1} \}}(t)\}_{t\geq0}$ is a strongly continuous
semigroup in $\dot{M}^{p,\ell}(\R^N)$.

\end{lemma}

\begin{proof}
Due to part (ii) of
Lemma \ref{lem:fixed-point-lem}
for each $t>0$, setting
$v_0=S _{\mu,\{V^{0},V^{1} \}}(t)u_0$, then $v(t) = S
_{\mu,\{V^{0},V^{1} \}}(t)v_0 = S _{\mu,\{V^{0},V^{1} \}}(2t)u_0$
satisfies
\begin{displaymath}
v(t)  = S_{\mu,V^{1}}(t)
v_0 +  \int_0^t S_{\mu,V^{1}}(t-\tau)  V^{0} v(\tau) \,  d\tau .
\end{displaymath}

Now we can repeat the proof in Lemma
\ref{lem:perturbed-smoothing-to-dotted-spaces} with the obvious
changes and using now that
\begin{displaymath}
  S_{\mu,V^{1}}(t) v_{0} \in \dot{M}^{q,s}(\R^N),   \qquad
  S_{\mu,V^{1}}(t-\tau) V^{0} v(\tau) \in \dot{M}^{q,s}(\R^N) , \quad \tau\in (0,t).
\end{displaymath}
\end{proof}

\section{Order preserving properties of the perturbed semigroup in
  Morrey spaces}
\label{sec:order-preservation}

In this section we study order preserving properties of the perturbed
semigroup. By this we mean that if the initial data satisfies
$u_{0}\geq 0$ then $S_{\mu,\{V^{0},  V^{1}\}}(t) u_{0} \geq 0$ for all
$t>0$. This is of course equivalent to the property of comparison of different
solutions with ordered initial data. Also, recall that any order
preserving semigroup satisfies 
  \begin{displaymath}
    |S(t) u_{0}| \leq S(t) |u_{0}| . 
  \end{displaymath}

So from now on we restrict ourselves to the case of second order
operators and, as observed in Section \ref{sec:homogeneous-lin-eq}, we
can assume without loss of generality that $A_{0}=-\Delta$. So all
results in Section \ref{sec:perturbed_equation_Morrey} will be used
below with $m=1$. 

First, we start with the unperturbed semigroup.

\begin{theorem}
\label{thr:Smu-order-preserving}

For $0< \mu \leq 1$, the semigroup $\{S_{\mu}(t)\}_{t\geq 0}$ in $M^{p,\ell}(\R^N)$ for $1\leq p \leq
\infty$, $0<\ell\leq N$, or in $\mathcal{M}^{\ell}(\R^{N})$, is order
  preserving.

  \end{theorem}
  \begin{proof}
This is immediate since the kernels in
    (\ref{eq:kernel-heat-smgp}),
    (\ref{eq:kernel-fractional-heat-smgp}) are nonnegative.
      \end{proof}

  \begin{remark}
Observe that from (\ref{eq:S0alpha(t)-in-MU}) the kernels  in
    (\ref{eq:kernel-heat-smgp}),    (\ref{eq:kernel-fractional-heat-smgp}) are related
by
    \begin{displaymath}
      k_{\mu}(t,x,y)=   \int_{0}^{\infty} f_{t,\mu}(s)  k_{1}(s,x,y) \, ds
    \end{displaymath}
    see \cite[Corollary 4.6]{C-RB-scaling3-4}. With this we have that
    since $k_{1}\geq 0$ then $k_{\mu}\geq 0$ as well. 
  \end{remark}

For the perturbed semigroup, our  first result is about order preserving properties for
  perturbations with bounded potentials.

  \begin{proposition}
\label{prop:order_preserving_bounded_perturbation}

Consider  $0<\mu\leq 1$ and the semigroup  $\{S_{\mu}(t)\}_{t\geq 0}$ in $M^{p,\ell}(\R^N)$ for $1\leq p \leq
 \infty$, $0<\ell\leq N$ or in $\mathcal{M}^{\ell}(\R^{N})$.

 \begin{enumerate}
 \item
Assume  $V^{1} \in L^\infty(\R^N)\cap M^{p_1,\ell_1}(\R^N)$ is an admissible potential as in
(\ref{eq:admisible_potential}) with $m=1$. 

Then for $1\leq p \leq \infty$, $0<\ell\leq \ell_1$ the semigroup
$\{S_{\mu, V^{1}}(t)\}_{t\geq 0}$ in $M^{p,\ell}(\R^N)$  and  in $\mathcal{M}^{\ell}(\R^{N})$
is order preserving.

 \item
Assume  $V^{1} \in M^{p_1,\ell_1}(\R^N)$ is an admissible potential as in
(\ref{eq:admisible_potential}) with $m=1$, such that for $1\leq p \leq \infty$,
$0<\ell\leq \ell_1$ the semigroup $\{S_{\mu, V^{1}}(t)\}_{t\geq 0}$ in
$M^{p,\ell}(\R^N)$ or in  $\mathcal{M}^{\ell}(\R^{N})$
is order preserving.

Assume we have another admissible potential
  $V^{0} \in L^\infty(\R^N)\cap  M^{p_0,\ell_0}(\R^N)$ with $\ell_{0} \leq \ell_{1}$, so
  (\ref{eq:for-thm-with-2potentials}) is satisfied with $m=1$.

  Then for $1\leq p \leq \infty$, $0<\ell\leq \ell_0$ the semigroup
$\{S_{\mu, \{V^{0}, V^{1}\}} (t)\}_{t\geq 0}$ in $M^{p,\ell}(\R^N)$
or  in  $\mathcal{M}^{\ell}(\R^{N})$ is order preserving.

 \end{enumerate}
  \end{proposition}
\begin{proof}
(i)   We consider all along the proof  initial data $u_{0}\in M^{p,\ell}(\R^N)$ for $1\leq p \leq \infty$, $0<\ell\leq
  \ell_1$ or $p=1$ and  $u_{0} \in \mathcal{M}^{\ell}(\R^{N})$. 

  \noindent
{\bf Step 1.} Assume first $V^{1}\geq 0$ and  we show that $\{S_{\mu, V^{1}}(t)\}_{t\geq 0}$ is
order preserving in $M^{p,\ell}(\R^N)$ or in
$\mathcal{M}^{\ell}(\R^{N})$. 

Indeed by part (i) in Lemma \ref{lem:fixed-point-lem} we have that for
$u_{0}\geq 0$ and $T>0$,  the sequence of Picard's
iterations
\begin{displaymath}
U_{n+1}(t)=   S_{\mu}(t)
u_0 +  \int_0^t S_{\mu}(t-\tau)  V^{1} U_{n}(\tau) \,  d\tau,
\qquad U_{1}(t) =  S_{\mu}(t) u_{0}, \qquad t\in(0,T], \ n\in \N,
\end{displaymath}
converge to  $u(t) = S_{\mu, V^{1}}(t) u_{0}$. Since
$u_{0}\geq 0$ then for all $n\in \N$ we get $U_{n}(t) \geq 0$ in
$[0,T]$ and then $u(t) = S_{\mu, V^{1}}(t) u_{0}\geq
0$.

  \noindent
  {\bf Step 2.}
Without any assumption on the sign of $V^{1}$, Step 1 above implies
$\{S_{\mu,  (V^{1})^{+}}(t)\}_{t\geq 0}$ is order preserving in
$M^{p,\ell}(\R^N)$ and in  $\mathcal{M}^{\ell}(\R^{N})$. 

  Assume now $W\! \in L^{\infty}(\R^{N})$. Then we show that
$\{S_{\mu,  \{W, (V^{1})^{+}\} }(t)\}_{t\geq 0}$ is order preserving
in $M^{p,\ell}(\R^N)$ and in  $\mathcal{M}^{\ell}(\R^{N})$. Using this, and in particular taking  $W=-(V^{1})^{-}\in L^\infty(\R^N)\cap
M^{p_1,\ell_1}(\R^N)$ we would conclude part (i) since  by Remark \ref{rem:adding_potentials}
we have
$\{S_{\mu,  \{W, (V^{1})^{+}\} }(t)\}_{t\geq 0} = \{S_{\mu,
  (V^{1})^{+} + W}(t)\}_{t\geq 0} = \{S_{\mu,  V^{1}}(t)\}_{t\geq 0} $  is order
preserving.

Assume first $W\geq 0$. Then by part (ii) in Lemma \ref{lem:fixed-point-lem} we have that for
$u_{0}\geq 0$ and $T>0$,  the sequence of Picard's
iterations
\begin{displaymath}
U_{n+1}(t)=   S_{\mu, (V^{1})^{+} }(t)
u_0 +  \int_0^t S_{\mu , (V^{1})^{+}}(t-\tau)  W U_{n}(\tau) \,  d\tau,
\quad U_{1}(t) =  S_{\mu, (V^{1})^{+} }(t) u_{0}, \quad t\in(0,T], \ n\in \N,
\end{displaymath}
converge to  $u(t) = S_{\mu,  \{W, (V^{1})^{+}\} }(t) u_{0}$. Since
$u_{0}\geq 0$ then for all $n\in \N$ we get $U_{n}(t) \geq 0$ in
$[0,T]$ and then $u(t) = S_{\mu,  \{W, (V^{1})^{+}\} }(t) u_{0}\geq
0$.

Without assumption on the sign of $W$, take $c\in \R$ such that $W+c
\geq 0$. Then by Proposition \ref{add_exponential_to_semigroup} we get
\begin{displaymath}
0\leq   S_{\mu,\{W + c, (V^{1})^{+} \}}(t)u_0 = e^{ct} S_{\mu,\{W, (V^{1})^{+} \}}(t)u_0 , \qquad t>0
\end{displaymath}
and then $u(t) = S_{\mu,  \{W, (V^{1})^{+}\} }(t) u_{0}\geq
0$.

\noindent (ii)
Now  we consider all along the proof  initial data $u_{0}\in
M^{p,\ell}(\R^N)$ for $1\leq p \leq \infty$, $0<\ell\leq   \ell_0$ or
$p=1$ and  $u_{0}\in \mathcal{M}^{\ell}(\R^{N})$, and
we follow closely the arguments in Step 2 above.

Assume first $V^{0}\geq 0$. Then by part (ii) in Lemma \ref{lem:fixed-point-lem} we have that for
$u_{0}\geq 0$ and $T>0$,  the sequence of Picard's
iterations
\begin{displaymath}
U_{n+1}(t)=   S_{\mu, V^{1} }(t)
u_0 +  \int_0^t S_{\mu , V^{1}}(t-\tau)  V^{0} U_{n}(\tau) \,  d\tau,
\quad U_{1}(t) =  S_{\mu, V^{1} }(t) u_{0}, \quad t\in(0,T], \ n\in \N,
\end{displaymath}
converge to  $u(t) = S_{\mu,  \{V^{0}, V^{1}\} }(t) u_{0}$. Since
$u_{0}\geq 0$ then for all $n\in \N$ we get $U_{n}(t) \geq 0$ in
$[0,T]$ and then $u(t) = S_{\mu,  \{V^{0}, V^{1}\} }(t) u_{0}\geq
0$.

Without assumption on the sign of $V^{0}$, take $c\in \R$ such that $V^{0}+c
\geq 0$. Then by Proposition \ref{add_exponential_to_semigroup} we get
\begin{displaymath}
  S_{\mu,\{V^{0} + c, V^{1}\}}(t)u_0 = e^{ct} S_{\mu,\{V^{0}, V^{1} \}}(t)u_0 , \qquad t>0
\end{displaymath}
and then $u(t) = S_{\mu,  \{V^{0}, V^{1}\} }(t) u_{0}\geq
0$.
\end{proof}

Now we prove our first main result in this section. Namely that a
perturbed semigroup with any admissible potential is order preserving.

\begin{theorem}
  \label{thm:order-pres-fo-general-Morrey-pot}

Consider  $0<\mu \leq 1$ and the semigroup  $\{S_{\mu}(t)\}_{t\geq 0}$ in $M^{p,\ell}(\R^N)$ for $1\leq p \leq
\infty$, $0<\ell\leq N$ or in  $\mathcal{M}^{\ell}(\R^{N})$ and $p=1$. 
Assume also  $V^{1} \in M^{p_1,\ell_1}(\R^N)$ is an admissible
potential as in (\ref{eq:admisible_potential}) with $m=1$. 

Then for $1\leq p \leq \infty$, $0<\ell\leq \ell_1$ the semigroup
$\{S_{\mu, V^{1}}(t)\}_{t\geq 0}$ in $M^{p,\ell}(\R^N)$ or in  $\mathcal{M}^{\ell}(\R^{N})$
is order preserving.

\end{theorem}
\begin{proof}
The result is a consequence of what we prove below: there  exists a
sequence of potentials $\{V^{1}_n\} \subset L^\infty(\R^N) \cap
M^{p_1,\ell_1}(\R^N)$
such that  $\{S_{\mu, V^{1}_n}(t)\}_{t}$ is order
preserving in $M^{p,\ell} (\R^N)$ for  $1\leq p \leq \infty$,
$0<\ell\leq \ell_1$ or in  $\mathcal{M}^{\ell}(\R^{N})$ 
and there exist $1< q \leq \infty$ and
$0<s\leq \ell$ such that for $u_{0}\in M^{p,\ell}(\R^N)$ we have
\begin{equation}
  \label{eq:convergence_of_solutions}
  \lim_{n\to\infty}\|S_{\mu, V^{1}_n}(t) u_0 - S_{\mu, V^{1}}(t) u_0\|_{M^{q,s}(\R^N)} = 0 \quad
  \text{ for  } \ t>0.
\end{equation}
Moreover, $\{V^{1}_n\}$ can be chosen such that
$V^{1}_n \to V^{1}$  as $n\to \infty$ in $L^{p_1}_{loc}(\R^N)$.
Clearly, since  $\{S_{\mu, V^{1}_{n}}(t)\}_{t\geq 0}$
is order preserving, then  (\ref{eq:convergence_of_solutions})
implies $S_{\mu, V^{1}}(t) u_0\geq 0$ if $u_{0} \geq 0$.

To prove (\ref{eq:convergence_of_solutions}), we \emph{smooth out} the
potential and define
\begin{displaymath}
V^{1}_n= S_\mu\big(\frac{1}{n}\big) V^{1}, \quad n\in\N.
\end{displaymath}
Then the  smoothing estimates (\ref{eq:estimates_Mpl-Mqs}) imply that
$\{V^{1}_n\} \subset L^\infty(\R^N) \cap M^{p_1,\ell_1}(\R^N)$,
and from  (\ref{eq:behavior-at-0-of-S{mu}(t)})  we have that $V^{1}_n \to
V^{1}$  as $n\to \infty$ in $L^{p_1}_{loc}(\R^N)$. In particular,
since $V^{1}_n \in L^\infty(\R^N)$ then from Proposition
\ref{prop:order_preserving_bounded_perturbation}  we have that
$\{S_{\mu, V^{1}_{n}}(t)\}_{t\geq 0}$
is order preserving in $M^{p,\ell} (\R^N)$ or in
$\mathcal{M}^{\ell}(\R^{N})$.

Therefore, in what follows we prove
(\ref{eq:convergence_of_solutions}) for $V^{1}_n$ as above.
Taking $1\leq p \leq \infty$, $0<\ell\leq \ell_1$ and $u_{0}\in M^{p,\ell}(\R^N)$ we denote
\begin{displaymath}
u_n (t)= S_{\mu, V^{1}_n}(t) u_0, \ \qquad \  u(t)=S_{\mu, V^{1}}(t) u_0, \quad n\in \N, \ t>0
\end{displaymath}
and using (\ref{eq:VCF_one_potential}) for both semigroups and the
definition of $V^{1}_n $ we get
\begin{equation}\label{eq:formula-for-un-u}
\begin{split}
u_n (t) - u (t)
&= \int_{0}^{t}  S_{\mu}(t- \tau) V^{1}_n u_n( \tau) \,  d\tau - \int_{0}^{t}  S_{\mu}(t- \tau) V^{1} u( \tau) \,  d\tau
\\
&
= \int_{0}^{t}  \left( S_\mu\big(\frac{1}{n}\big) S_{\mu}(t- \tau)
  V^{1} u_n( \tau) - S_{\mu}(t- \tau) V^{1} u( \tau) \right)  d\tau
\\
& =
\int_{0}^{t} \left(S_\mu\big(\frac{1}{n}\big) S_{\mu}(t- \tau) V^{1} u( \tau) - S_{\mu}(t- \tau) V^{1} u( \tau)\right)  d\tau
\\
&
\phantom{aaaaaaaaa} +
\int_{0}^{t}  S_\mu\big(\frac{1}{n}\big) S_{\mu}(t- \tau) V^{1} \left(u_n( \tau) -u(\tau) \right) d\tau, \quad t>0.
\end{split}
\end{equation}

Below we show that there exist $1< q \leq \infty$ and $0<s\leq \ell_1$ such that for any $t>0$ we have
\begin{equation}
  \label{eq:required-convergence}
\lim_{n\to \infty}\|u_n( t) -u(t) \|_{M^{q,s}(\R^N)}=0
\end{equation}
which will prove (\ref{eq:convergence_of_solutions}).
The proof of (\ref{eq:required-convergence})
follows in four  steps.

\smallskip

\noindent
{\bf Step 1.} In this step, taking $t>0$ fixed and denoting
\begin{displaymath}
U_{n} ( \tau)=S_\mu\big(\frac{1}{n}\big) S_{\mu}(t- \tau) V^{1} u( \tau) - S_{\mu}(t- \tau) V^{1} u( \tau), \quad 0<\tau< t,
\end{displaymath}
we show that if $p_1'\leq q\leq \infty$ and $0<s\leq \ell$ satisfy $\frac{s}{q}\leq \frac{\ell}{p}$ then there is a constant $C>0$ such that we have both
\begin{equation}\label{eq:estimate-in-step-1}
\sup_{n\in \N} \| U_{n}( \tau) \|_{M^{q,s}(\R^N)} \leq
C\frac{\| V^{1}\|_{M^{p_1,\ell_1}(\R^N)} \big( \|u_0\|_{M^{p,\ell}(\R^N)} + \| V^{1}\|_{M^{p_1,\ell_1}(\R^N)}\big) }{\tau^{\frac{1}{2\mu}(\frac{\ell}{p}-\frac{s}{q})}(t- \tau)^{\frac{\ell_1}{2\mu p_1}}}, \quad 0<\tau < t
\end{equation}
and
\begin{equation}
  \label{eq:convergence-in-step-2}
  \lim_{n\to \infty} \| U_{n}( \tau) \|_{M^{q,s}(\R^N)}=0, \quad 0<\tau < t.
\end{equation}
Notice that if $p=1$ and   $u_{0}\in \mathcal{M}^{\ell}(\R^{N})$ we replace
$\|u_0\|_{M^{1,\ell}(\R^N)}$ by $\|u_0\|_{\mathcal{M}^{\ell}(\R^N)}$. 

For this, the estimate from \cite[Theorem 7.4]{C-RB-morrey_linear_perturbation}
ensures that for $1\leq q\leq \infty$, $0<s\leq \ell$ satisfying
$\frac{s}{q}\leq \frac{\ell}{p}$ we have 
\begin{equation}\label{eq:integrability-of-u-in-(0,t)}
\|u ( \tau) \|_{M^{q,s}(\R^N)} \leq
\frac{C_0}{\tau^{\frac{1}{2\mu}(\frac{\ell}{p}-\frac{s}{q})}} \big(
\|u_0\|_{M^{p,\ell}(\R^N)} + \|
V^{1}\|_{M^{p_1,\ell_1}(\R^N)}\big)=:H( \tau), \quad 0 <\tau
\leq t.
\end{equation}
Again,  notice that if $p=1$ and   $u_{0}\in \mathcal{M}^{\ell}(\R^{N})$ we replace
$\|u_0\|_{M^{1,\ell}(\R^N)}$ by $\|u_0\|_{\mathcal{M}^{\ell}(\R^N)}$.

Next we use the multiplication properties in Morrey spaces,
\cite[Lemma 5.1]{C-RB-morrey_linear_perturbation} to see that if $q,s$
are as above and $q\geq p_1'$ then for $z$, $\nu$ such that
\begin{equation}\label{eq:z-and-nu}
\frac{1}{z}= \frac{1}{q}+\frac{1}{p_1}, \ \qquad \ \frac{\nu}{z}=
\frac{s}{q}+\frac{\ell_1}{p_1},
\end{equation}
we have
\begin{displaymath}
\|V^{1}u ( \tau) \|_{M^{z,\nu}(\R^N)}\leq \|V^{1}
\|_{M^{p_1,\ell_1}(\R^N)}\|u ( \tau) \|_{M^{q,s}(\R^N)},
\end{displaymath}
and from  (\ref{eq:integrability-of-u-in-(0,t)}) we get
\begin{displaymath}
\|V^{1}u ( \tau) \|_{M^{z,\nu}(\R^N)}\leq
\| V^{1}\|_{M^{p_1,\ell_1}(\R^N)}H( \tau),
\quad
0<\tau \leq t.
\end{displaymath}

Now from the  smoothing properties of the semigroup
$\{S_{\mu}(t)\}_{t\geq 0}$ in (\ref{eq:estimates_Mpl-Mqs})  we  obtain
\begin{displaymath}
  \| S_{\mu}(t- \tau) V^{1} u( \tau) \|_{M^{q,s}(\R^N)} \leq
K( \tau) \|V^{1} u( \tau) \|_{M^{z,\nu}(\R^N)}  \leq
\| V^{1}\|_{M^{p_1,\ell_1}(\R^N)}H( \tau) K(
\tau),
\end{displaymath}
with
\begin{equation}\label{eq:mathscr{K}}
  K( \tau) = \frac{c}{(t- \tau)^{\frac{1}{2\mu}(\frac{\nu}{z}-\frac{s}{q})}}
=  \frac{c}{(t- \tau)^{\frac{\ell_1}{2\mu p_1}}}.
\end{equation}
Using this and the  boundedness of the semigroup
$\{S_\mu\big(t)\}_{t\geq0}$ in $M^{q,s}(\R^N)$, see again (\ref{eq:estimates_Mpl-Mqs}),
we get (\ref{eq:estimate-in-step-1}).

Since for $z,\nu$ as in (\ref{eq:z-and-nu}) we have
$V^{1}u ( \tau)\in M^{z,\nu}(\R^N)$,
from Lemma \ref{lem:smoothing-to-dotted-spaces} we get
\begin{equation}\label{eq:here-we-exclude{q=1}}
S_{\mu}(t- \tau) V^{1} u( \tau) \in \dot M^{q,s}(\R^N)
\end{equation}
provided  $q >1$.
Then using  (\ref{eq:behavior-at-0-of-S{mu}(t)-for-dotted}) in
$M^{q,s}(\R^N)$ and  Remark \ref{rem:Smu-is-C0} if $q=\infty$, we get
(\ref{eq:convergence-in-step-2}).

\smallskip

\noindent
{\bf Step 2.} In this step taking $T>0$ we show that given $p_1'\leq
q\leq \infty$ and $0<s\leq \ell$ satisfying $\frac{s}{q}\leq
\frac{\ell}{p}$ there exists  a constant $c>0$ such that for each $0<\tau <
t < T$ and $n\in\N$
\begin{equation}\label{eq:estimate-in-Step-2}
   \| S_\mu\big(\frac{1}{n}\big) S_{\mu}(t- \tau) V^{1} \left(u_n( \tau) -u(\tau) \right) \|_{M^{q,s}(\R^N)} \leq
c\frac{\| V^{1}\|_{M^{p_1,\ell_1}(\R^N)} }{(t- \tau)^{\frac{\ell_1}{2\mu p_1}}} \|u_n( \tau) -u(\tau) \|_{M^{q,s}(\R^N)}.
\end{equation}

For this using again  \cite[Lemma
5.1]{C-RB-morrey_linear_perturbation} we get that  if $q,s$ are as
above and $z$, $\nu$ are as in (\ref{eq:z-and-nu})
then
\begin{displaymath}
\|V^{1}\left(u_n( \tau) -u(\tau) \right) \|_{M^{z,\nu}(\R^N)}\leq
\| V^{1}\|_{M^{p_1,\ell_1}(\R^N)} \|u_n( \tau) -u(\tau) \|_{M^{q,s}(\R^N)},
\quad
0<\tau \leq t.
\end{displaymath}
Again from  the  smoothing properties of the semigroup
$\{S_{\mu}(t)\}_{t\geq 0}$ in  (\ref{eq:estimates_Mpl-Mqs})  we have
\begin{displaymath}
\| S_{\mu}(t- \tau) V^{1} \left( u_n( \tau) -u(\tau) \right) \|_{M^{q,s}(\R^N)} \leq
K( \tau) \|V^{1} \left( u_n( \tau) -u(\tau) \right) \|_{M^{z,\nu}(\R^N)}
\end{displaymath}
with $K( \tau)$ as in (\ref{eq:mathscr{K}}), and we  obtain
\begin{displaymath}
\| S_{\mu}(t- \tau) V^{1} \left( u_n( \tau) -u(\tau) \right) \|_{M^{q,s}(\R^N)} \leq
K( \tau) \| V^{1}\|_{M^{p_1,\ell_1}(\R^N)}\|u_n( \tau) -u(\tau) \|_{M^{q,s}(\R^N)} ,
\end{displaymath}
and the  boundedness of the semigroup $\{S_\mu\big(t)\}_{t\geq0}$ in
$M^{q,s}(\R^N)$, see again (\ref{eq:estimates_Mpl-Mqs}),
gives (\ref{eq:estimate-in-Step-2}).

\smallskip

\noindent
{\bf Step 3.} Now we show that for some $p_1'\leq q\leq \infty$, $0<s\leq \ell$ satisfying $\frac{s}{q}\leq \frac{\ell}{p}$ we have (\ref{eq:required-convergence}).

For this, from (\ref{eq:formula-for-un-u}) and Steps 1 and 2 above,
for $p_1'\leq q\leq \infty$, $0<s\leq \ell$ satisfying
$\frac{s}{q}\leq \frac{\ell}{p}$, denote
\begin{displaymath}
a_n(t)= \int_0^t  \|U_{n}(\tau)\|_{M^{q,s}(\R^N)} \, d\tau,
\end{displaymath}
and then for $0\leq t <T$ we have
\begin{equation}\label{eq:inequality-for-norm-of-un-u}
\|u_n(t)-u(t)\|_{M^{q,s}(\R^N)} \leq a_n(t) + \int_0^t \frac{ c_1
}{(t- \tau)^{\frac{\ell_1}{2\mu p_1}}}
\|u_n(\tau)-u(\tau)\|_{M^{q,s}(\R^N)} \, d\tau,
\end{equation}
with  $c_1=c\| V^{1}\|_{M^{p_1,\ell_1}(\R^N)}$.

Therefore, if we chose $q,s$ as above and satisfying additionally
\begin{equation}\label{eq:restriction-for-q,s}
  \frac{1}{2\mu}\big(\frac{\ell}{p}-\frac{s}{q}\big)<1
\end{equation}
then  (\ref{eq:estimate-in-step-1}) yields  that for some constant
independent of $n\in \N$ we have
\begin{displaymath}
  a_{n}(t) \leq c \int_0^t  \tau^{-\frac{1}{2\mu}(\frac{\ell}{p}-\frac{s}{q})}(t-
\tau)^{-\frac{\ell_1}{2\mu p_1}} \, d\tau=
t^{1-\frac{1}{2\mu}(\frac{\ell}{p}-\frac{s}{q})-\frac{\ell_1}{2\mu
    p_1}}\int_0^1  r^{-\frac{1}{2\mu}(\frac{\ell}{p}-\frac{s}{q})}(1-
r)^{-\frac{\ell_1}{2\mu p_1}} \, d r=: W(t) .
\end{displaymath}
Since $\frac{\ell_1}{2\mu p_1}<1$ because  $V^1$ is an admissible
potential and (\ref{eq:restriction-for-q,s}),
$W(t)$ is well defined  and therefore $a_{n}$ is  integrable in $(0,T)$.

On the other hand,  observe that since, using again
(\ref{eq:estimates_Mpl-Mqs}),   $\| V^{1}_n\|_{M^{p_1,\ell_1}(\R^N)} =
\|S_\mu(\frac{1}{n})  V^{1} \|_{M^{p_1,\ell_1}(\R^N)} \leq c
\|V^{1}\|_{M^{p_1,\ell_1}(\R^N)}$,  then, arguing as above,
(\ref{eq:integrability-of-u-in-(0,t)}) also holds for $u_n$. As a
consequence, (\ref{eq:restriction-for-q,s}) implies that
$\|u_n(t)-u(t)\|_{M^{q,s}(\R^N)}$ is integrable in $(0,T)$ for any
$T>0$.

Then we show now that, given $p$ and $\ell$ as in the statement,
we can find  $q$ and $s$ satisfying $p_1'\leq q\leq \infty$, $0<s\leq \ell$ and
$\frac{s}{q}\leq \frac{\ell}{p}$, $q>1$ and
(\ref{eq:restriction-for-q,s}).

\smallskip

\noindent
\fbox{Case I: $ p_1'\leq p \leq \infty$}
If $p >1$ then we take
\begin{displaymath}
q=p \quad \text{ and } \quad s=\ell
\end{displaymath}
which clearly satisfy all conditions including (\ref{eq:restriction-for-q,s}).

On the other hand, if   $p=1$ then $p_1'=1$ and we can take
$q=1+\varepsilon$ and $s=\ell$ with
some  small  $\varepsilon>0$.

\smallskip

\noindent
\fbox{Case II: $1\leq p< p_1'$}
In this case we can take
\begin{displaymath}
q=p_1' >1 \quad \text{ and } \quad s=\ell.
\end{displaymath}
For which only  (\ref{eq:restriction-for-q,s}) remains to be
checked. Actually, the left hand side in
(\ref{eq:restriction-for-q,s}) is  $\frac{\ell}{2\mu p_1} +
\frac{\ell}{2\mu}(\frac{1}{p}- 1)$ were  first term  is
strictly less than $1$, since  $\frac{\ell}{2\mu p_1} \leq
\frac{\ell_1}{2\mu p_1} <1$ because  $V^{1}$ is an admissible
potential, whereas the second term  is
$\frac{\ell}{2\mu}(\frac{1}{p}- 1)\leq 0$.

\noindent
{\bf Step 4.}
With $q,s$ as above, using the singular Gronwall lemma, \cite[Lemma
7.1.1]{1981_Henry},  we get from
(\ref{eq:inequality-for-norm-of-un-u}) that for $t\in[0,T)$
\begin{equation}\label{eq:another-inequality-for-norm-of-un-u}
\|u_n(t)-u(t)\|_{M^{q,s}(\R^N)} \leq a_n(t) + \theta_1 \int_0^t
E'\big(\theta_1 (t-\tau)\big) a_n(\tau)\, d\tau,
\end{equation}
where $\theta_1=\left(c_1 \Gamma(b_1)\right)^\frac{1}{b_1}$ with $b_1=1-\frac{\ell_1}{2\mu p_1}$ and $E(z)=\sum_{n=0}^\infty \frac{z^{nb_1}}{\Gamma(nb_1 +1)}$.

From  \cite[(1.35)-(1.36), p. 27]{2010_Yagi}, $\mathcal{E}(z)=\sum_{n=0}^\infty
\frac{z^{n}}{\Gamma(nb_1 +1)}$  is an entire function and
$E'(\theta_1 (t-\tau)) = b_1 \theta_1^{b_1-1} (t-\tau)^{b_1-1}
\mathcal{E}'((\theta_1 (t-\tau))^{b_1})$ hence  in particular $E'(\theta_1
(t-\tau))$ is continuous for  $\tau \in [0,t)$. Using
this and recalling from \cite[Lemma 7.1.1]{1981_Henry} that
$E'(\theta_1 (t-\tau))$ behaves like a multiple of $(t-\tau)^{b_1-1}$
as $\tau\nearrow t$ we see that $E'(\theta_1 (t-\tau))$  is integrable
in $\tau\in (0,t)$.

On the other hand, from (\ref{eq:estimate-in-step-1}),
(\ref{eq:convergence-in-step-2}) and Lebesgue's theorem we get
\begin{displaymath}
 \lim_{n\to \infty}a_n(t)=0, \ t>0
\end{displaymath}
and again Lebesgue's theorem in
(\ref{eq:another-inequality-for-norm-of-un-u}) yields
(\ref{eq:required-convergence}), which completes the proof.
\end{proof}

Our next main result is for  the case of two admissible  potentials for which we
get the following.

  \begin{theorem}
    \label{thm:order-pres-for-2-Morrey-pot}

    Consider  $0<\mu \leq 1$ and the
semigroup  $\{S_{\mu}(t)\}_{t\geq 0}$ in $M^{p,\ell}(\R^N)$ for $1\leq
p \leq\infty$, $0<\ell\leq N$ or in  $\mathcal{M}^{\ell}(\R^{N})$ and
$p=1$.  Also consider   admissible potentials
as in (\ref{eq:for-thm-with-2potentials}) with $m=1$ and initial data $u_0\in
M^{p,\ell}(\R^N)$, $1\leq p \leq\infty$, $0<\ell\leq\ell_0$ or in
$\mathcal{M}^{\ell}(\R^{N})$.

 Then $\{S_{\mu,\{V^{0}, V^{1}\}}(t)\}_{t \geq 0}$ is
  order preserving  in   $M^{p,\ell}(\R^N)$ or in
  $\mathcal{M}^{\ell}(\R^{N})$. 

\end{theorem}
\begin{proof}
From Theorem \ref{thm:order-pres-fo-general-Morrey-pot} we have that
$\{S_{\mu, V^{1}}(t)\}_{t\geq 0}$ is order preserving in
$M^{p,\ell}(\R^N)$ or in $\mathcal{M}^{\ell}(\R^{N})$. 

  Now we perturb this semigroup with the  admissible potential
  $V^{0} \in M^{p_0,\ell_0}(\R^N)$ with $\ell_{0} \leq \ell_{1}$, so
  (\ref{eq:for-thm-with-2potentials}) is satisfied.
  From part (ii) in  Lemma \ref{lem:fixed-point-lem} for $u_{0}\in
  M^{p,\ell}(\R^N)$,  or $u_{0}\in \mathcal{M}^{\ell}(\R^{N})$, 
  and $T>0$
  $u(t) = S_{\mu,\{V^{0}, V^{1}\}}(t)u_0$ satisfies
  \begin{displaymath}
    u(t) = S_{\mu,V^{1}}(t)
    u_0 +  \int_0^t S_{\mu,V^{1}}(t-\tau)  V^{0} u(\tau) \,  d\tau,
    \quad t\in(0,T].
  \end{displaymath}

Then   we carry out the same arguments as in the proof of Theorem
  \ref{thm:order-pres-fo-general-Morrey-pot} with $S_{\mu,V^{1}}(t)$
  playing the role of $S_{\mu}(t)$. In particular we take as
  approximating potentials
  \begin{displaymath}
    V^{0}_n= S_{\mu,V^{1}}\big(\frac{1}{n}\big) V^{0}, \quad n\in\N.
  \end{displaymath}
  Instead of (\ref{eq:estimates_Mpl-Mqs}) we use in this case
  \begin{displaymath}
    \|S_{\mu, V^{1}}(t)\|_{{\mathcal L}(M^{p,\ell}(\R^N), M^{q,s}
      (\R^N))} \leq
    \frac{c}{ t^{\frac{1}{2\mu}(\frac\ell{p}-\frac{s}q)}}, \quad t \in (0,T)
  \end{displaymath}
  which follows from (\ref{eq:smoothing_one_potential}). Therefore
  $\{V^{0}_n\} \subset L^\infty(\R^N) \cap M^{p_0,\ell_0}(\R^N)$,
and from  (\ref{eq:behavior-at-0-of-S{mu}(t)}) and
(\ref{eq:SV_approach_initial_data}) we have that
  $V^{0}_n \to
V^{0}$  as $n\to \infty$ in $L^{p_0}_{loc}(\R^N)$. In particular,
since $V^{0}_n \in L^\infty(\R^N)$ then from Proposition
\ref{prop:order_preserving_bounded_perturbation}  we have that $\{S_{\mu, \{V^{1},
  V^{0}_{n}\}}(t)\}_{t\geq 0}$
is order preserving in $M^{p,\ell} (\R^N)$.

Therefore we consider here
\begin{displaymath}
u_n (t)= S_{\mu, \{V^{0}_{n},V^{1}
  \} }(t) u_0, \ \qquad \  u(t)=S_{\mu, \{V^{0},V^{1}\} }(t) u_0, \quad n\in \N, \ t>0
\end{displaymath}
and then in place of (\ref{eq:formula-for-un-u})  we get
\begin{displaymath}
\begin{split}
u_n (t) - u (t)
& =
\int_{0}^{t} \left(S_{\mu, V^{1}} \big(\frac{1}{n}\big) S_{\mu,
    V^{1}}(t- \tau) V^{0} u( \tau) - S_{\mu, V^{1}}(t- \tau) V^{0} u( \tau)\right)  d\tau
\\
&
\phantom{aaaaaaaaa} +
\int_{0}^{t}  S_{\mu, V^{1}} \big(\frac{1}{n}\big) S_{\mu, V^{1}}(t- \tau) V^{0} \left(u_n( \tau) -u(\tau) \right) d\tau, \quad t>0.
\end{split}
\end{displaymath}

Hence, we can proceed as in Step 1 of the proof of Theorem
\ref{thm:order-pres-fo-general-Morrey-pot} and instead of Lemma
  \ref{lem:smoothing-to-dotted-spaces} in
  (\ref{eq:here-we-exclude{q=1}}) we use Lemma
  \ref{lem:perturbed-smoothing-to-dotted-spaces}   to get
  \begin{displaymath}
    S_{\mu, V^{1}}(t- \tau) V^{0} u( \tau) \in \dot M^{q,s}(\R^N)
  \end{displaymath}
  and then by (\ref{eq:SV_approach_initial_data}) we get that, as
  $n \to \infty$,
  \begin{displaymath}
    S_{\mu,V^{1}}\big(\frac{1}{n}\big)   S_{\mu, V^{1}}(t- \tau) V^{0} u(
    \tau) \to  S_{\mu, V^{1}}(t- \tau) V^{0} u(\tau)
  \end{displaymath}
  in $M^{q,s}(\R^N)$ and we get the equivalent to
  (\ref{eq:convergence-in-step-2}).

  The rest of the proof follows with the obvious changes.
  \end{proof}

\section{Monotonicity and growth bound of the perturbed semigroups}
\label{sec:monotonicity}

With the notations in Section \ref{sec:order-preservation}, we  now
show monotonicity properties of the order preserving perturbed 
semigroup with respect to Morrey potentials.

  The main result about monotonicity with respect to potentials
  is the following.

  \begin{theorem}
    \label{thm:monotonicity-with-respect-to-potentials}

     Consider  $0<\mu\leq 1$ and the
semigroup  $\{S_{\mu}(t)\}_{t\geq 0}$ in $M^{p,\ell}(\R^N)$ for $1\leq
p \leq\infty$, $0<\ell\leq N$ or in $\mathcal{M}^{\ell}(\R^{N})$ and
$p=1$. 
Also consider  admissible potentials as in
(\ref{eq:for-thm-with-2potentials}) with $m=1$.

Assume also
\begin{displaymath}
V^{0}\leq \tilde{V}^{0}, \qquad   V^{1}\leq \tilde{V}^{1} \quad \text{
  where} \quad  V^{0}, \tilde{V}^{0}\in M^{p_{0},\ell_{0}}(\R^N), \
V^{1}, \tilde{V}^{1}\in M^{p_{1},\ell_{1}}(\R^N) .
\end{displaymath}

Then for $0\leq u_{0} \in M^{p,\ell}(\R^N)$ with $1\leq p \leq \infty$,
$0<\ell\leq \ell_0$ or $p=1$ and $u_{0}\in \mathcal{M}^{\ell}(\R^{N})$ we have
\begin{equation}\label{eq:bound_with_u0geq0}
   0 \leq S_{\mu,\{V^{0},V^{1} \}}(t) u_0 \leq S_{\mu,\{\tilde{V}^{0},\tilde{V}^{1} \}}(t) u_0, \quad t>0
\end{equation}
and for $u_{0} \in M^{p,\ell}(\R^N)$ or $u_{0}\in \mathcal{M}^{\ell}(\R^{N})$ we have
\begin{equation}\label{eq:bound_with_|u0|}
  |S_{\mu,\{V^{0},V^{1}\}}(t) u_{0}| \leq S_{\mu,\{\tilde{V}^{0},\tilde{V}^{1}\}}(t)|u_{0}|, \quad t\geq 0.
\end{equation}

\end{theorem}
\begin{proof}
We take $0\leq u_{0} \in M^{p,\ell}(\R^N)$ with $1\leq p \leq \infty$,
$0<\ell\leq \ell_0$ or in $\mathcal{M}^{\ell}(\R^{N})$ if $p=1$. Also observe that by Theorems
\ref{thm:order-pres-fo-general-Morrey-pot} and
\ref{thm:order-pres-for-2-Morrey-pot} all semigroups appearing below are order preserving.

By part (iv) in Lemma \ref{lem:fixed-point-lem} we have that for
$T>0$, $u(t)= S_{\mu,\{\tilde{V}^{0},V^{1} \}}(t) u_0 \geq 0$
satisfies
\begin{displaymath}
u(t) =    S_{\mu,\{V^{0},
  V^{1}\}}(t) u_0 +  \int_0^t S_{\mu,\{V^{0}, V^{1}\}}(t-\tau)
\big(\tilde{V}^{0}- V^{0}\big)  u(\tau) \,  d\tau, \ t\in(0,T],
\end{displaymath}
and then $u(t) \geq  S_{\mu,\{V^{0},  V^{1}\}}(t) u_0\geq 0$ as the
integral term is nonnegative. 

Analogously, by part (v) in Lemma \ref{lem:fixed-point-lem} we have that for
$T>0$, $v(t)= S_{\mu,\{\tilde{V}^{0},\tilde{V}^{1} \}}(t) u_0 \geq 0$
satisfies
\begin{displaymath}
v(t) =    S_{\mu,\{\tilde{V}^{0},
  V^{1}\}}(t) u_0 +  \int_0^t S_{\mu,\{\tilde{V}^{0}, V^{1}\}}(t-\tau)
\big(\tilde{V}^{1}- V^{1}\big)  v(\tau) \,  d\tau, \ t\in(0,T],
\end{displaymath}
and then $v(t) \geq  S_{\mu,\{\tilde{V}^{0},  V^{1}\}}(t) u_0
=u(t)$, as the
integral term is nonnegative, and we already have   $u(t)\geq  S_{\mu,\{V^{0},
  V^{1}\}}(t) u_0\geq 0$. Hence  we get
(\ref{eq:bound_with_u0geq0}).

Then (\ref{eq:bound_with_|u0|}) follows from
(\ref{eq:bound_with_u0geq0}) and $-|u_{0}|\leq u_{0}\leq |u_{0}|$.
\end{proof}

Now we obtain growth bound for order preserving perturbed semigroup in
terms of the norms of the  positive parts of the potentials. The
estimate obtained below improves (\ref{eq:smoothing_two_potentials}),
(\ref{eq:general_growth_bound}).

\begin{theorem}\label{thm:exponential-estimate-of-S{mu,V0,V1}}

  Consider  $0<\mu \leq 1$ and the
semigroup  $\{S_{\mu}(t)\}_{t\geq 0}$ in $M^{p,\ell}(\R^N)$ for $1\leq
p \leq\infty$, $0<\ell\leq N$ or in $\mathcal{M}^{\ell}(\R^{N})$ if $p=1$. 
Also consider   admissible potentials
as in (\ref{eq:for-thm-with-2potentials}) with $m=1$. 

Then for $1\leq p,q\leq \infty$, $0<s\leq\ell\leq\ell_0$ satisfying
$\frac{s}{q}\leq \frac{\ell}{p}$ and $p\geq \max\{p_0',p_1'\}$ or $q\geq \max\{p_0',p_1'\}$ we have
\begin{displaymath}
  \| S_{\mu, \{V^{0}, V^{1}\}}(t)\|_{\mathcal{L}(M^{p,\ell}(\R^N), M^{q,s}(\R^N))}
  \leq
  \frac{Ce^{at}}{t^{\frac{1}{2\mu}(\frac{\ell}{p}-\frac{s}{q})}}
  \quad t>0,
\end{displaymath}
where for $p=1$  we can replace $M^{1,\ell}(\R^{N})$ with
$\mathcal{M}^{\ell}(\R^{N})$,  
with 
\begin{equation}\label{eq:letting-a}
  a=
  c
  \big(
  \|(V^{0})^+\|_{M^{p_0,\ell_0}(\R^N)}^{\frac{1}{1-\kappa_{0}}}
  +
  \|(V^{1})^+\|_{M^{p_1,\ell_1}(\R^N)}^{\frac{1}{1-\kappa_{1}}}
  \big)
\end{equation}
and $c$, $C$ are some positive constants.

\end{theorem}
\begin{proof}
  Observe that since $V^{0} \leq (V^{0})^+$ and $V^{1} \leq (V^{1})^+$
  then Theorem \ref{thm:monotonicity-with-respect-to-potentials}
  gives,  for $u_{0} \in M^{p,\ell}(\R^N)$ with $1\leq p \leq \infty$,
$0<\ell\leq \ell_0$ or in  $\mathcal{M}^{\ell}(\R^{N})$, 
\begin{displaymath}
    |S_{\mu,\{V^{0},V^{1}\}}(t) u_{0}| \leq S_{\mu,\{(V^{0})^+,(V^{1})^+\}}(t)|u_{0}|, \quad t\geq 0 ,
  \end{displaymath}
  while from (\ref{eq:smoothing_two_potentials}) we get
  $\| S_{\mu, \{(V^{0})^+,(V^{1})^+ \}}(t)u_0\|_{M^{q,s}(\R^N)} \leq
  \frac{Ce^{at}}{t^{\frac{1}{2\mu}(\frac{\ell}{p}-\frac{s}{q})}} \|
  u_0 \|_{M^{p,\ell}(\R^N)}$ for $t>0$ with $a$ as in
  (\ref{eq:letting-a}). Thus we get the result.
\end{proof}

\section{Decay of individual solution in Morrey spaces}
\label{sec:decay-indivual-solutions}

In this section we obtain decay properties of individual solution in Morrey scale.
We start from the results concerning pure fractional
diffusion.

Observe that from  from (\ref{eq:estimates_Mpl-Mqs}) we have for $u_0
\in M^{p,\ell}(\R^N)$ or $p=1$ and $u_{0}\in
\mathcal{M}^{\ell}(\R^N)$,  and  $1\leq p, q \leq \infty$ and $0\leq s\leq
\ell\leq N$ with $\frac{s}q < \frac{\ell}p$, 
\begin{equation}\label{eq:pure-dif-decaying-of-individual-sol}
\|S_{\mu}(t)u_0   \|_{M^{q,s}(\R^N)} \to 0 \quad \text{ as } t\to\infty.
\end{equation}

With additional assumption on the  initial data we now prove that
the convergence in (\ref{eq:pure-dif-decaying-of-individual-sol}) extends to  cases when
$\frac{s}{q}= \frac{\ell}{p}$. Notice that  (ii) or (iii) below
impose some regularity or smallness of the \emph{tails} $\big(1-\chi_{_{
    B(x,r)}}\big)u_0$ of the initial data,  respectively, where $\chi_{_{ B(x,r)}}$
denotes the characteristic function of a ball $B(x,r)\subset \R^N$. 

\begin{theorem}
\label{thm:more-pure-frac-dif-decay-of-individual-sol}
Assume $0<\mu \leq 1$ and $u_0 \in M^{p,\ell}(\R^N)$ for some $1\leq p<
\infty$, $0<\ell\leq N$.

Assume  $1<p <\infty$ if $\ell=N$  or  the initial data satisfies either one of the following conditions  
  
\begin{enumerate}

\item
   $u_0$ is the limit in $M^{p,\ell}(\R^N)$ of a sequence
   $\{u_{0n}\}\subset M^{z,\nu}(\R^N)\cap M^{p,\ell}(\R^N)$ 
  with $1\leq z\leq \infty$, $0<\ell\leq\nu \leq N$
  satisfying $\frac{\ell}{p} < \frac{\nu}{z}$.

\item
$\big(1-\chi_{_{ B(x,r)}}\big)u_0\in M^{z,\nu}(\R^N)$
for some $x\in \R^N$, $r>0$ and some $1\leq z\leq \infty$, $0<\ell\leq\nu\leq N$
  satisfying $\frac{\ell}{p} < \frac{\nu}{z}$.

  \item
   $\|\big(1-\chi_{_{ B(x,r)}}\big)u_0\|_{M^{p,\ell}(\R^N)} \to 0$ as
$r\to \infty$ for some $x\in \R^N$.

\end{enumerate}

Then
\begin{equation}\label{eq:more-pure-dif-decaying-of-individual-sol}
\|S_{\mu}(t)u_0   \|_{M^{p,\ell}(\R^N)} \to 0 \quad \text{ as } t\to\infty.
\end{equation}
Consequently, (\ref{eq:pure-dif-decaying-of-individual-sol}) holds for $1\leq q\leq \infty$ and $0<s\leq\ell$ satisfying
$\frac{s}{q}= \frac{\ell}{p}$.

\end{theorem}
\begin{proof}
Assuming (i), $u_0$ is the limit in $M^{p,\ell}(\R^N)$ of a sequence $\{u_{0n}\}\subset M^{z,\nu}(\R^N)\cap M^{p,\ell}(\R^N)$. Then, due to (\ref{eq:estimates_Mpl-Mqs}), for positive times
$\|S_{\mu} (t) \|_{\mathcal{L}(M^{z,\nu}(\R^N),M^{p,\ell}(\R^N))} \leq \frac{c}{t^{\frac{1}{2\mu}(\frac{\nu}{z}-\frac{\ell}{p})}} $
and $\|S_{\mu}(t)\|_{\mathcal{L}(M^{p,\ell}(\R^N))}
\leq     c $.
Using this and writing $u_0$ as $u_0-u_{0n} + u_{0n}$ we obtain
\begin{displaymath}
  \| S_{\mu} (t)u_0\|_{M^{p,\ell}(\R^N)}
    \leq
  c \| u_0-u_{0n}\|_{M^{p,\ell}(\R^N)} + \frac{c}{t^{\frac{1}{2\mu}(\frac{\nu}{z}-\frac{\ell}{p})}} \| u_{0n}\|_{M^{z,\nu}(\R^N)}, \quad t>0, \ n\in\N,
\end{displaymath}
which yields $\limsup_{t\to\infty} \| S_{\mu} (t)u_0\|_{M^{p,\ell}(\R^N)} \leq \| u_0-u_{0n}\|_{M^{p,\ell}(\R^N)}$ for every $n\in\N$.
Since this implies that $\limsup_{t\to\infty} \| S_{\mu} (t)u_0\|_{M^{p,\ell}(\R^N)} =0$, we get (\ref{eq:more-pure-dif-decaying-of-individual-sol}).

If (ii) holds then
actually $\big(1-\chi_{_{ B(x,r)}}\big)u_0\in M^{z,\nu}(\R^N) \cap M^{p,\ell}(\R^N)$, so
applying (i)  with $u_0$ replaced by $\big(1-\chi_{_{ B(x,r)}}\big)u_0$ we get \begin{displaymath}
\lim_{t\to\infty} \|S_{\mu} (t)\big(\big(1-\chi_{_{ B(x,r)}}\big)u_0\big)\|_{M^{p,\ell}(\R^N)}=0.
\end{displaymath}

On the other hand, since $\chi_{_{ B(x,r)}}u_0 \in M^{p,\ell}(\R^N)$
and has compact support, then it belongs to
$L^1(\R^N)=M^{1,N}(\R^N)$, and then  $\chi_{_{ B(x,r)}}u_0$ is a
limit in $M^{1,N}(\R^N)$ of a sequence from
$C^\infty_c(\R^N)\subset M^{p,\ell}(\R^N) \cap M^{1,N}(\R^N)$. Hence,
(i) applies  with $u_0$ replaced by $\chi_{_{ B(x,r)}}u_0$ with  $z=1$ and
$\nu=N$ (also for  $1<p<\infty$ if $\ell=N$). Therefore 
we get
\begin{displaymath}
\lim_{t\to\infty} \|S_{\mu} (t)\big(\chi_{_{ B(x,r)}}u_0\big)\|_{M^{p,\ell}(\R^N)}=0
\end{displaymath}
and thus (\ref{eq:more-pure-dif-decaying-of-individual-sol}).

If (iii) holds then, since we get as above that $\lim_{t\to\infty}
\|S_{\mu} (t)\big(\chi_{_{ B(x,r)}}u_0\big)\|_{M^{p,\ell}(\R^N)}=0$
(also  $1<p<\infty$ if $\ell=N$),
and
\begin{displaymath}
\|S_{\mu} (t)\big(\big(1-\chi_{_{
    B(x,r)}}\big)u_0\big)\|_{M^{p,\ell}(\R^N)} \leq
c\|\big(1-\chi_{_{B(x,r)}}\big)u_0 \|_{M^{p,\ell}(\R^N)} ,
\end{displaymath}
we see that
\begin{displaymath}
\limsup_{t\to \infty} \| S_{\mu} (t)u_0\|_{M^{p,\ell}(\R^N)} \leq  c\|
\big(1-\chi_{_{ B(x,r)}}\big)u_0\|_{M^{p,\ell}(\R^N)} \to
0 \ \text{ as } r\to\infty, 
\end{displaymath}
which yields (\ref{eq:more-pure-dif-decaying-of-individual-sol}). This
holds in particular if $\ell=N$ and $1<p<\infty$.  

We remark that once
(\ref{eq:more-pure-dif-decaying-of-individual-sol}) is proved, the
convergence (\ref{eq:pure-dif-decaying-of-individual-sol}) holds for
$1\leq q\leq \infty$, $0<s\leq\ell$ satisfying
$\frac{s}{q}= \frac{\ell}{p}$ due to Morrey's embedding
(\ref{eq:summary-undotted-Morrey-2}).
\end{proof}

Now we prove that in general,  without additional assumption on the
initial data,  (\ref{eq:pure-dif-decaying-of-individual-sol}) does not hold  for $(q,s) =
(p, \ell)$ since we can construct solutions of the pure fractional diffusion problem with
constant norm.

\begin{proposition}\label{prop:constant_norm_solutions}
For  $0<\mu \leq 1$, $1\leq  p \leq  \infty$ and  $0<\ell <  N$ and
for $p=1$ and $\ell=N$, there exists 
$u_{0} \in \dot{M}^{p,\ell}(\R^N)$ such that
\begin{displaymath}
\|S_{\mu}(t) u_{0}\|_{M^{p,\ell}(\R^N)} = c , \quad t>0. 
\end{displaymath}
\end{proposition}
\begin{proof}
  For $p=\infty$ we take  $u_{0}=1$ and then  $S_{\mu}(t)u_{0}=1$ and has constant
$L^{\infty}(\R^{N})$  norm.

Now for  $1\leq p  <\infty$ and $0<\ell < N$ the function 
  \begin{displaymath}
\phi(x)=\frac{1}{|x|^\frac{\ell}{p}}, \quad x\in \R^N \setminus \{0\} ,
\end{displaymath}
satisfies $\phi  \in M^{p,\ell}(\R^{N})$  and is   homogenous of
degree $-\frac{\ell}{p}$, see \cite[Lemma A.4]{C-RB-scaling3-4},  so the solution $S_{\mu}(t)
\phi$ is selfsimilar, see \cite[Theorem 5.2]{C-RB-scaling3-4} and
$\|S_{\mu}(t) \phi\|_{M^{p,\ell}(\R^N)}$ has constant norm, see
\cite[Corollary 5.4 and Corollary 5.5]{C-RB-scaling3-4}. Then  from  Lemma
\ref{lem:smoothing-to-dotted-spaces}
$u_{0}= S_{\mu}(1) \phi\in \dot{M}^{p,\ell}(\R^{N})$ and with this initial
data we also get a solution with constant norm. 

Finally, for  $p=1$ and $\ell =N$, that is, in $L^{1}(\R^{N})$, 
  the kernel $k_{\mu}(t,x)$ (which is the solution with initial data
  the $\delta$)   has constant $L^{1}$ norm, see Propositions 6.1 and
  6.2 in \cite{C-RB-scaling3-4}. Hence we can take $u_{0}=
  k_{\mu}(1,\cdot)$ to get the result. 
\end{proof}

\begin{remark}
The existence of solutions that do not decay to zero in
$M^{p,\ell}(\R^{N})$ is in fact related to the fact that the range of
the fractional diffusion operator  $(-\Delta)^\mu$ is not dense, see \cite[Remark
1.5]{ArendtBattyBenilan92} and \cite[Remark 4.2(iii)]{C-RB-exponential_decay}
for the case of Lebesgue spaces 
$L^{p}(\R^{N})$, $1< p< \infty$.
Actually, the latter reference contains the result in Theorem
\ref{thm:more-pure-frac-dif-decay-of-individual-sol} for the case
$1<p<\infty$ and $\ell=N$. 
\end{remark}

We now consider dissipative  Morrey potentials, for  which  we have the following result.

\begin{corollary}\label{cor:from-{thm:more-pure-frac-dif-decay-of-individual-sol}}
Assume $0<\mu \leq 1$, $V^0\leq 0$ and $V^1\leq 0$ are admissible potentials as in
(\ref{eq:for-thm-with-2potentials}) with $m=1$, and $u_0 \in M^{p,\ell}(\R^N)$ for some $1\leq p< \infty$, $0<\ell\leq \ell_0$. Assume also
$1\leq q\leq \infty$, $0<s\leq\ell$ are such that $\frac{s}{q}\leq \frac{\ell}{p}$ and if
$\frac{s}{q}= \frac{\ell}{p}$ assume moreover that either,  $\ell<N$
and we have (i), or (ii), or (iii) in Theorem
\ref{thm:more-pure-frac-dif-decay-of-individual-sol},  or $\ell=N$ and $p\not=1$.

Then
\begin{displaymath}
\|S_{\mu, \{V^{0},V^{1}\}}(t)u_0   \|_{M^{q,s}(\R^N)} \to 0 \quad \text{ as } t\to\infty.
\end{displaymath}

\end{corollary}

\begin{proof}
Theorem \ref{thm:monotonicity-with-respect-to-potentials} gives
\begin{equation}\label{eq:from-{thm:monotonicity-with-respect-to-potentials}}
|S_{\mu, \{V^{0},V^{1}\}}(t)u_0 | \leq S_{\mu}(t)|u_{0}| ,
\quad  t>0.
\end{equation}
Using this, (\ref{eq:pure-dif-decaying-of-individual-sol}) and Theorem \ref{thm:more-pure-frac-dif-decay-of-individual-sol}
we get the result.
\end{proof}

We also have the following  dichotomy.

\begin{corollary} \label{cor:dicothomy_4_decaying_solutions}

  Assume $0<\mu \leq 1$, $V^0\leq 0$ and $V^1\leq 0$ are admissible potentials as in
(\ref{eq:for-thm-with-2potentials}) with $m=1$, and $\mathcal{Z}^{p,\ell}$ is given by
\begin{displaymath}
  \mathcal{Z}^{p,\ell}=
\{\phi\in M^{p,\ell}(\R^N) \colon \lim_{t\to\infty} \|S_{\mu, \{V^{0},V^{1}\}}(t) \phi   \|_{M^{p,\ell}(\R^N)} = 0 \}
\end{displaymath}
for some $1\leq p< \infty$ and $0<\ell\leq \ell_0$ with the restriction that $1<p<\infty$ if
$\ell=\ell_0=N$.

Then $\mathcal{Z}^{p,\ell}$ with the norm of  $M^{p,\ell}(\R^N)$ is a
nontrivial Banach space, positively invariant under
$\{S_{\mu,\{V^{0},V^{1}\}}(t)\}_{t\geq 0}$ and either 
\begin{displaymath}
  \|S_{\mu,\{V^{0},V^{1}\}}(t)\|_{\mathcal{L}(\mathcal{Z}^{p,\ell})} =1 \quad \text{ for all } t\geq0 ,
\end{displaymath}
or
$\{S_{\mu,\{V^{0},V^{1}\}}(t)\}_{t\geq 0}$ is exponentially decaying in $\mathcal{Z}^{p,\ell}$.

Moreover, in the former case
there exist arbitrarily slowly decaying solutions, that is,
given a continuous, positive and decreasing function $g(t)$ of variable $t\in[0,\infty)$ satisfying $\lim_{t\to\infty}g(t)=0$ there is $u_0\in \mathcal{Z}^{p,\ell}$ such that
\begin{displaymath}
      \limsup_{t\to\infty} \frac{\|S_{\mu, \{V^{0},V^{1}\}}(t)u_0\|_{\mathcal{Z}^{p,\ell}}}{g(t)}=\infty .
    \end{displaymath}

If $p=1$ we can replace $M^{1,\ell}(\R^{N})$ with
$\mathcal{M}^{\ell}(\R^{N})$. 
    
\end{corollary}
\begin{proof}
It is immediate that  $\mathcal{Z}^{p,\ell}$ is a linear  space
and we now prove it is closed in $M^{p,\ell}(\R^N)$. For this we take
$\{\phi_n\}\subset \mathcal{Z}^{p,\ell}$ convergent to $\phi$ in
$M^{p,\ell}(\R^N)$. Using (\ref{eq:estimates_Mpl-Mqs}) and
(\ref{eq:from-{thm:monotonicity-with-respect-to-potentials}}) we see
that 
\begin{displaymath}
\begin{split}
\|S_{\mu, \{V^{0},V^{1}\}}(t) \phi   \|_{M^{p,\ell}(\R^N)}  &
\leq
\|S_{\mu, \{V^{0},V^{1}\}}(t) (\phi -\phi_n)  \|_{M^{p,\ell}(\R^N)}
+
\|S_{\mu, \{V^{0},V^{1}\}}(t) \phi_n   \|_{M^{p,\ell}(\R^N)}
\\
&
\leq
c \|\phi -\phi_n  \|_{M^{p,\ell}(\R^N)}
+
\|S_{\mu, \{V^{0},V^{1}\}}(t) \phi_n   \|_{M^{p,\ell}(\R^N)} .
\end{split}
\end{displaymath}
Passing to the limit as $t\to\infty$ we get
$\limsup_{t\to\infty} \|S_{\mu, \{V^{0},V^{1}\}}(t) \phi
\|_{M^{p,\ell}(\R^N)} \leq c \|\phi -\phi_n \|_{M^{p,\ell}(\R^N)} $
for every $n\in\N$, which yields
$\lim_{t\to\infty} \|S_{\mu, \{V^{0},V^{1}\}}(t) \phi
\|_{M^{p,\ell}(\R^N)}=0$. Hence $\phi\in \mathcal{Z}^{p,\ell}$, and so
$\mathcal{Z}^{p,\ell}$ is a Banach space, which is nontrivial by
Theorem \ref{thm:more-pure-frac-dif-decay-of-individual-sol}. 

Since for each positive time $S_{\mu, \{V^{0},V^{1}\}}(t)$ takes
$\mathcal{Z}^{p,\ell}$ into itself, then 
$\{S_{\mu, \{V^{0},V^{1}\}}(t)\}_{t\geq0}$ is a semigroup  in
$\mathcal{Z}^{p,\ell}$. 

If $\|S_{\mu,\{V^{0},V^{1}\}}(t_0)\|_{\mathcal{L}(\mathcal{Z}^{p,\ell})}
\not=1$ for some $t_0>0$ then,  from 
(\ref{eq:from-{thm:monotonicity-with-respect-to-potentials}}) and
Lemma \ref{lem:about-contractions}, 
$\|S_{\mu,\{V^{0},V^{1}\}}(t_0)\|_{\mathcal{L}(\mathcal{Z}^{p,\ell})}
<1$, so $\{S_{\mu,\{V^{0},V^{1}\}}(t)\}_{t\geq 0}$ is exponentially
decaying in $\mathcal{Z}^{p,\ell}$, see Lemma \ref{lem:about-decay}
proved below. 

If this is not the case, take $g$ as in the statement and  suppose
that $\limsup_{t\to\infty} \frac{\|S_{\mu,
    \{V^{0},V^{1}\}}(t)u_0\|_{\mathcal{Z}^{p,\ell}}}{g(t)}<\infty $
for each $u_0\in \mathcal{Z}^{p,\ell}$. 
Then for each
    $u_0\in \mathcal{Z}^{p,\ell}$
    there is a constant $C=C(u_{0})$ such that
    \begin{displaymath}
      \frac{\|S_{\mu, \{V^{0},V^{1}\}}(t)u_0\|_{\mathcal{Z}^{p,\ell}}}{g(t)} \leq C(u_{0}), \quad t\geq
      0 ,
    \end{displaymath}
    so by the uniform boundedness principle we have
    \begin{displaymath}
    \sup_{t\geq0}\frac{\|S_{\mu, \{V^{0},V^{1}\}}(t)\|_{\mathcal{L}(\mathcal{Z}^{p,\ell})}}{g(t)} =:M <\infty ,
    \end{displaymath}
    and so $\|S_{\mu,
      \{V^{0},V^{1}\}}(t)\|_{\mathcal{L}(\mathcal{Z}^{p,\ell})} \leq M
    g(t)<1$ for large enough $t$ which is a contradiction. 

With the same proof, if $p=1$ we can replace 
$M^{1,\ell}(\R^{N})$ with $\mathcal{M}^{\ell}(\R^{N})$. 
  \end{proof}

\section{Exponential bounds  for  the  perturbed semigroups. The case
  of a dissipative potential}
\label{sec:exponential-decay}

 We consider first the case of a dissipative admissible  potential, that is $V\leq 0$, and our goal is to obtain
 sharp exponential bounds on the perturbed semigroup $\{S_{\mu, V}(t)\}_{t\geq 0}$ in Morrey  spaces.
Hence, if $V\leq 0$,  from
Theorem \ref{thm:monotonicity-with-respect-to-potentials}
for $u_{0} \in
M^{p,\ell}(\R^N)$ with $1\leq p \leq
\infty$, $0<\ell\leq N$ we have
\begin{equation} \label{eq:domination}
  |S_{\mu, V}(t) u_{0}| \leq S_{\mu}(t)|u_{0}|, \quad t\geq 0
\end{equation}
and from Lemma \ref{lem:about-contractions} we get
\begin{equation}\label{eq:contracting-property-in-Mpl}
     \|S_{\mu, V}(t)\|_{{\mathcal L}(M^{p,\ell}(\R^N))}  \leq
    1, \quad t\geq 0 . 
  \end{equation}
    
For  $0<\mu\leq 1$ the exponential type of the semigroup $\{S_{\mu,
  V}(t)\}_{t\geq 0}$ in $M^{p,\ell}(\R^N)$ is defined by
\begin{equation}\label{eq:morrey_exponential_type}
 \omega_{p,\ell} (V) \mydef   \sup \Omega_{p,\ell} (V)
\end{equation}
where
\begin{displaymath}
\Omega_{p,\ell} (V)  \mydef \{\omega \in \R \colon  \|S_{\mu, V} (t)\|_{\mathcal{L}(M^{p,\ell}(\R^N))} \leq
 M e^{-\omega t} \ \mbox{for all $t\geq 0$ and some $M\geq 1$} \} . 
\end{displaymath}
A completely analogous definition holds when we replace
$M^{1,\ell}(\R^{N})$ by $\mathcal{M}^{\ell}(\R^{N})$. The exponential
type is denoted $\omega_{\ell}(V)$ in this case.

Since is dissipative,  $V\leq 0$,  from (\ref{eq:contracting-property-in-Mpl}),  $0\in
\Omega_{p,\ell} (V)$ and then $\omega_{p,\ell} (V)\geq 0$. Therefore, either
$\omega_{p,\ell} (V) >0$ and the semigroup decays exponentially or $\omega_{p,\ell} (V) =0$ and
then $ \|S_{\mu, V}(t)\|_{{\mathcal L}(M^{p,\ell}(\R^N))} =1$ for
all $t\geq 0$ and Corollary \ref{cor:dicothomy_4_decaying_solutions}
applies so there exist slowly decaying solutions. Also,  to prove exponential decay it is enough to prove that for
some $t_{0}>0$ we have
\begin{equation}\label{eq:enough_4_decay}
 \|S_{\mu, V}(t_{0})\|_{{\mathcal  L}(M^{p,\ell}(\R^N))} <1 
\end{equation}
see Lemma \ref{lem:about-decay}.

In what follows, unless some confusion may arise, we will drop the
dependence on $V$ in the exponential type.

\subsection{A  bounded potential}
\label{sec:bounded-potential}

We first consider the case of a bounded dissipative potential $0\geq V \in L^{\infty}(\R^{N})$. In such a case,
following \cite{Arendt-Batty}, we consider the following
definition.

\begin{definition}

  \begin{enumerate}
  \item
    We say that a set contains arbitrarily large balls iff for any
    $r > 0$ there exists $x_0 \in \R^N$ such that the ball of radius
    $r$ around $x_0$ is included in this set.

    The class $\mathcal{G}$ consists of all open subsets of $\R^N$
    containing arbitrarily large balls.

\item 
    We say that the bounded dissipative potential $V\leq 0$ satisfies the
    Arendt and Batty's condition iff
    \begin{equation}\label{eq:Arendt-Batty-condition}
      \int_{G} V\, dx=-\infty \quad \text{ for all } \ G\in
      \mathcal{G} . 
    \end{equation}
      \end{enumerate}
\end{definition}

\begin{remark} \label{rem:equivalent-2-AB-condition}

    From \cite[Proposition 1.4]{Arendt-Batty},
    (\ref{eq:Arendt-Batty-condition}) is equivalent to the existence
    of $c>0$ and $r>0$ such that
    \begin{displaymath}
      \sup_{x\in \R^N} \int_{B(x,r)}  V \leq -c . 
    \end{displaymath}

\end{remark}

Because $V\in L^{\infty}(\R^{N})$ the semigroup $S_{\mu,V}(t)$ can
also be considered in  Lebesgue spaces $L^{p}(\R^{N})$ with $1\leq
p\leq \infty$ (recall that  when  $\ell =N$, $M^{p,\ell}(\R^N) =
L^{p}(\R^N)$).
Then, in a completely analogous way as in
(\ref{eq:morrey_exponential_type})  above we define the exponential
type $\omega_{p}\geq 0$ of the semigroup in $L^{p}(\R^{N})$ for
$1\leq p \leq \infty$. In particular, the results in
\cite[Theorem 4.1]{C-RB-exponential_decay} imply that for $0<\mu \leq 1$, either
all $\omega_{p}>0$ or all are zero, see Remark
\ref{rem:exponential_type_in_Lebesgue}. 
It was actually  proved in \cite{C-RB-exponential_decay} that for $0<\mu \leq
1$,  exponential decay  happens
for some $1\leq p \leq \infty$ if and only if
(\ref{eq:Arendt-Batty-condition})  holds true, see \cite[Theorem 4.6  and
Theorem 4.9]{C-RB-exponential_decay}. Recall that the case $\mu=1$ was
first considered  in  \cite{Arendt-Batty}.
In particular,  condition
(\ref{eq:Arendt-Batty-condition}) is equivalent to $\omega_{\infty}>0$
for $0<\mu \leq 1$.

Hence, (\ref{eq:Arendt-Batty-condition}) characterises the exponential decay
in Lebesgue spaces for a bounded dissipative  potential and in case it
holds there exists $C_0\geq1$, $\omega_{0}>0$ such that for all $1\leq p\leq
\infty$,
\begin{displaymath} 
      \|S_{\mu,V}(t) \|_{\mathcal{L}(L^p(\R^N))}\leq C_0 e^{-\omega_0 t}, \quad t\geq0.
\end{displaymath}
Notice that $\omega_{0}$ can be taken independent of $1\leq p\leq
\infty$, see \cite[Theorem 4.1, Theorem 4.6]{C-RB-exponential_decay}.

We show below similar results  for Morrey spaces. Recall that if
$V$ is bounded then the semigroup $S_{\mu, V}(t)$ is well
defined in $M^{p,\ell}(\R^{N})$ for $1\leq p\leq \infty$ and $0<\ell
\leq N$ and in $\mathcal{M}^{\ell}(\R^{N})$ for $0<\ell \leq N$, see Remark
\ref{rem:adding_potentials}.

First we prove condition (\ref{eq:Arendt-Batty-condition}) is
necessary for exponential decay.

\begin{theorem}\label{thm:necessity-of-Arendt-Batty-condition}

  Assume $0<\mu \leq 1$ and $0 \geq V\in L^{\infty}(\R^N)$.

If for some $1\leq p,q\leq \infty$, $0<\ell\leq N$ satisfying
$\frac{s}{q}\leq \frac{\ell}{p}$ we have
\begin{equation}\label{eq:the-estimate-considered-here}
\|S_{\mu, V }(t)\|_{\mathcal{L}(M^{p,\ell}(\R^N),M^{q,s}(\R^N))}\leq \frac{Ce^{-\omega t}}{t^{\frac{1}{2\mu}(\frac{\ell}{p}-\frac{s}{q})}},
  \quad t>0,
\end{equation}
with certain $C,\omega>0$,  then $\omega_{\infty}>0$ and then  $V$ satisfies
(\ref{eq:Arendt-Batty-condition}).

\end{theorem}
\begin{proof}
If (\ref{eq:the-estimate-considered-here}) holds with $\omega>0$ and
$p=q=\infty$, then (\ref{eq:Arendt-Batty-condition}) follows  from the
results in \cite{C-RB-exponential_decay} mentioned above.
Therefore, we assume that we have (\ref{eq:the-estimate-considered-here})
with  $\omega>0$ and some $1\leq p,q\leq \infty$, $0<\ell\leq
N$ satisfying $\frac{s}{q}\leq \frac{\ell}{p}$ and $p\not=\infty$.  We
are going to show that (\ref{eq:the-estimate-considered-here}) implies
that the semigroup decays in $L^{\infty}(\R^{N})$ and so
(\ref{eq:Arendt-Batty-condition}) follows again from
the results in \cite{C-RB-exponential_decay}.

For this observe that  from the  order preserving properties of the
perturbed semigroup
we have
\begin{displaymath}
\|S_{\mu,V}(t)\|_{\mathcal{L}(L^\infty(\R^N))} = \|S_{\mu,V}
(t) 1\|_{L^\infty(\R^N)} .
\end{displaymath}

Then fix $x\in \R^{N}$ and denote by  $\Car_{R}$ the characteristic
function of the ball $B(x,R)$ and decompose, for $t\geq 1$,
\begin{equation}\label{eq:inequality-we-now-consider-here-below}
  (S_{\mu, V} (t) 1 ) (x)= \Big(S_{\mu, V} (1) S_{\mu,
    V} (t-1) \Car_{R}\Big) (x)   +
 \Big( S_{\mu, V} (t) (1-\Car_{R})\Big) (x) .
\end{equation}

As for the first term in
(\ref{eq:inequality-we-now-consider-here-below}), estimate
(\ref{eq:smoothing_one_potential}) gives
\begin{equation}\label{eq:from-Proposition-5.2}
\begin{split}
  \| S_{\mu, V} (1) &S_{\mu, V} (t-1) \Car_{R}\|_{L^\infty(\R^N)}
  \\
  &
  \leq \|S_{\mu, V} (1)\|_{\mathcal{L}(M^{q,s}(\R^N),L^\infty(\R^N))}
  \|S_{\mu, V} (t-1)\|_{\mathcal{L}(M^{p,\ell}(\R^N), M^{q,s}(\R^N))} \|\Car_{R}\|_{M^{p,\ell}(\R^N)}
\end{split}
\end{equation}
and $\|\Car_{R}\|_{M^{p,\ell}(\R^N)}= c R^\frac{\ell}{p}$. To see
this, observe that translations are isometries in Morrey spaces, so we can
  assume $x=0$ and then  $\Car_{R}$ is a dilation of $\Car_{1}$ and $M^{p,\ell}(\R^N)$ is a
  homogeneous space of degree $-\frac{\ell}{p}$, see \cite[Remark
  2.11(i)]{C-RB-scaling1}, so
     $ \|\Car_{R}\|_{M^{p,\ell}(\R^N)}=
     \|\Car_{1}\|_{M^{p,\ell}(\R^N)} R^\frac{\ell}{p}$.

For  the second term in
(\ref{eq:inequality-we-now-consider-here-below}), since $0 \geq
V\in L^{\infty}(\R^N)$, from Theorem
 \ref{thm:monotonicity-with-respect-to-potentials}
we have
\begin{equation}\label{eq:use-in-estimate-below}
0\leq S_{\mu,V}(t) (1-\Car_{R}) \leq S_{\mu}(t) (1-\Car_{R}).
\end{equation}

We next consider the following two cases.

\smallskip

\noindent
\fbox{Case $0<\mu<1$.} In this case using (\ref{eq:use-in-estimate-below}) and the kernel in (\ref{eq:kernel-fractional-heat-smgp})
we get as in \cite[proof of Theorem 4.1]{C-RB-exponential_decay}
\begin{displaymath}
0 \leq \big( S_{\mu, V} (t) (1-\Car_{R})\big) \leq c \frac{t}{R^{2\mu}}.
\end{displaymath}

Using  this and (\ref{eq:from-Proposition-5.2}) in
(\ref{eq:inequality-we-now-consider-here-below})
we see that
\begin{displaymath}
  \|S_{\mu,V}(t)\|_{\mathcal{L}(L^\infty(\R^N))} \leq  c \big( \|S_{\mu, V}
  (t-1)\|_{\mathcal{L}(M^{p,\ell}(\R^N), M^{q,s}(\R^N))} R^{\frac{\ell}{p}} +  \frac{t}{R^{2\mu}} \big)
\end{displaymath}
and after minimising this in $R>0$ we obtain
\begin{equation}\label{eq:estimateLinfty_from_Mpell_mu}
    \|S_{\mu,V}(t)\|_{\mathcal{L}(L^\infty(\R^N))} \leq  c  \|S_{\mu, V}
  (t-1)\|_{\mathcal{L}(M^{p,\ell}(\R^N),M^{q,s}(\R^N))}^{\frac{2p\mu}{\ell+2p\mu}} t^{\frac{\ell}{\ell + 2p\mu}}
  .
\end{equation}

Then since $\omega>0$ in  (\ref{eq:the-estimate-considered-here})  we
get  that the right hand side of (\ref{eq:estimateLinfty_from_Mpell_mu})  tends to $0$ as
$t\to\infty$, we conclude that
$\|S_{\mu,V}(t)\|_{\mathcal{L}(L^\infty(\R^N))}\to 0$ as
$t\to\infty$, and we get (\ref{eq:enough_4_decay}) and then the result.

\noindent
\fbox{Case $\mu=1$.} In this case using (\ref{eq:use-in-estimate-below}) and the kernel in (\ref{eq:kernel-heat-smgp}) we get as in  \cite[proof of Theorem 4.1]{C-RB-exponential_decay}
\begin{displaymath}
  0 \leq \big( S_{\mu, V} (t) (1-\Car_{R})\big) (x) \leq c    e^{-
  \frac{R^2}{2t}} .
\end{displaymath}
Then using this in
(\ref{eq:from-Proposition-5.2}) and
(\ref{eq:inequality-we-now-consider-here-below}) we get
\begin{equation} \label{eq:estimateLinfty_from_Mpell}
  \|S_{\mu,V}(t)\|_{\mathcal{L}(L^\infty(\R^N))} \leq  c \big( \|S_{\mu, V}
  (t-1)\|_{\mathcal{L}(M^{p,\ell}(\R^N)),M^{q,s}(\R^N))} R^{\frac{\ell}{p}} + c    e^{-
  \frac{R^2}{2t}} \big)
\end{equation}
and taking $R=t$ and
since $\omega>0$ in (\ref{eq:the-estimate-considered-here}) we get 
again $\|S_{\mu,V}(t)\|_{\mathcal{L}(L^\infty(\R^N))}\to 0$ as
$t\to\infty$, and so we obtain (\ref{eq:enough_4_decay}) and then  the result.
\end{proof}

Then we prove the following result which in particular proves that
(\ref{eq:Arendt-Batty-condition}) gives exponential decay in Morrey
spaces.

\begin{theorem}\label{thm:eponential-type-in-Morrey}

Assume  $0<\mu\leq 1$ and $0\geq V\in L^\infty(\R^N)$ and  $1\leq
p< \infty$ and $0<\ell\leq N$.

Then
\begin{enumerate}

\item
If $0<\mu <1$ 
  \begin{equation}\label{eq:the-exp-type-when-0<mu<1}
0\leq \frac{\omega_\infty }{\vartheta}\leq \omega_{p,\ell} \leq
(1+\frac{\ell}{2p\mu}) \omega_{\infty}
\end{equation}
where $\vartheta= 1+\frac{N}{2\mu}$ and we can take $\vartheta=1$ when $\ell=N$.

\item
If $\mu=1$, 
  \begin{equation}\label{eq:the-same-exp-type}
    0 \leq \omega_{p,\ell}=\omega_\infty .
  \end{equation}

\end{enumerate}

In particular, for $0<\mu \leq 1$, either all
$ \omega_{p,\ell} = 0 = \omega_{\infty}$ or they are all positive simultaneously and then  $V$ satisfies
(\ref{eq:Arendt-Batty-condition}).  In such a case there
exists a common decay rate independent of $p$ and $\ell$. 

For the semigroup in $\mathcal{M}^{\ell}(\R^{N})$, we have
$\omega_{\ell} = \omega_{1,\ell}$ for $0<\mu \leq 1$.

\end{theorem}
\begin{proof}
  As observed above $\omega_{p,\ell}\geq 0$. Now the proof is done
  through several steps. 

\noindent   {\bf Step 1.} Here we show that if
$\omega_{\infty}=0$ then $\omega_{p,\ell}=0$ for all $1\leq
 p< \infty$ and $0<\ell\leq N$.

Actually, if  for some $1\leq
 p< \infty$ and $0<\ell\leq N$ we have $\omega_{p,\ell} >0$ then
 Theorem \ref{thm:necessity-of-Arendt-Batty-condition} 
implies that $\omega_{\infty}>0$, which is a contradiction. 

\smallskip

\noindent
{\bf Step 2.} Now  we show that if $\omega_{\infty}>0$ then we have
the lower bounds for $\omega_{p,\ell}$ in terms of $\omega_{\infty}$ in
(\ref{eq:the-exp-type-when-0<mu<1}), (\ref{eq:the-same-exp-type}).

Since  $\omega_\infty>0$, we take any $0<\omega \in \Omega_\infty$ so that we have 
  \begin{equation}\label{eq:omega-in-Omegainfty}
\|S_{\mu, V} (t)\|_{\mathcal{L}(L^\infty(\R^N))} \leq
 M e^{-\omega t}, \quad t> 0 .
\end{equation}
From (\ref{eq:domination}) and the  estimates in Lebesgue spaces from
Section 6 in \cite{C-RB-scaling1}, 
  \begin{displaymath} 
    \|S_{\mu}(t)\|_{{\mathcal L}(L^{p}(\R^N), L^{q} (\R^N))} =
    \frac{c}{ t^{\frac{N}{2\mu}(\frac{1}{p}-\frac{1}q)}}, \quad t>0
  \end{displaymath}
we get that 
$\|S_{\mu ,V}(t)\|_{{\mathcal L}(L^{1}(\R^N), L^\infty(\R^N))} \leq
\frac{c}{  t^{\frac{N}{2\mu}}}$. Using this and
(\ref{eq:omega-in-Omegainfty}) we then see 
that for each $0<\nu<\omega$ there is a constant $c$ such
that \begin{displaymath}
\|S_{\mu,V}(t)\|_{{\mathcal L}(L^{1}(\R^N), L^\infty(\R^N))} \leq
\frac{c e^{-\nu t}}{  t^{\frac{N}{2\mu }}} , \quad  t>0 .
\end{displaymath}

On the other hand,  from Lemma \ref{lem:kernel-for-SmuV0-in-L1U} 
the semigroup has a nonnegative  kernel $k_{\mu,V}(t,\cdot,\cdot)$ and
\begin{displaymath}
S_{\mu, V}(t) u_0  = \int_{\R^N} k_{\mu,V} (t,\cdot,y) u_0 (y) \, dy , 
\end{displaymath}
for all $ u_0\in  L^1_U(\R^N)$ and this space contains
$M^{p,\ell}(\R^N)$ for all  $1\leq  p\leq
\infty$, $0<\ell\leq N$.

Restricting to $u_{0}\in L^{1}(\R^{N}) \subset L^{1}_{U}(\R^{N})$ and
using \cite[Theorem in \S7.3.1]{2004-Arendt}, see also \cite[proof of 
Theorem 1.3]{1994_Arendt-Bukhvalov}, we have
\begin{displaymath}
\|k_{\mu,V}(t,\cdot,\cdot)\|_{L^\infty(\R^N\times \R^N)} =
\|S_{\mu,V}(t)\|_{{\mathcal L}(L^1(\R^N), L^\infty(\R^N))} \leq
\frac{c e^{-\nu t}}{  t^{\frac{N}{2\mu }}} , \quad  t>0 .
\end{displaymath}

Also, from \cite[Section 3]{C-RB-morrey_linear_perturbation}, it is known that $S_{\mu}(t)$
can be represented by the kernel in (\ref{eq:kernel-heat-smgp}) or
(\ref{eq:kernel-fractional-heat-smgp}), for 
$u_0\in  L^1_U(\R^N) $, and hence 
\begin{displaymath}
S_{\mu}(t) u_0 = \int_{\R^N} k_{\mu} (t,\cdot,y) u_0 (y) \, dy,
\end{displaymath}

Therefore, (\ref{eq:domination}) implies 
\begin{displaymath}
0 \leq k_{\mu,V} (t,x,y)\leq k_{\mu}(t,x,y), \quad t>0 .
\end{displaymath}

Then, given any $0<\theta<1$, we have 
\begin{equation}\label{eq:theta_estimate_kernelmuV0}
0\leq k_{\mu,V} (t,x,y) = k_{\mu,V}^\theta (t,x,y)
k_{\mu,V}^{1-\theta} (t,x,y)  \leq
\|k_{\mu,V}(t,\cdot,\cdot)\|_{L^\infty(\R^N\times \R^N)}^\theta
k_{\mu}^{1-\theta} (t,x,y), \quad t>0 . 
\end{equation}

\noindent
\fbox{Case $0<\mu<1$} 
Now the fractional heat kernel satisfies
(\ref{eq:kernel-fractional-heat-smgp}), that is,  
\begin{displaymath} 
  0\leq k_{\mu}(t,x,y)\sim
\frac{1}{t^\frac{N}{2\mu}}
I_{\mu}\left(\frac{x-y}{t^{\frac{1}{2\mu}}}\right) \ \text{ for } x,y\in \R^N 
\quad 
\text{ where } \
I_{\mu}(\cdot) = \frac{1}{(1+|\cdot|^{2})^{\frac{N+2\mu}{2}}} , 
\end{displaymath}
see
\cite{2019Bogdan_fract_laplac_hardy,2017BonforteSireVazquez:_optimal-existence},\cite[Section
6]{C-RB-scaling3-4},  
so using  this in (\ref{eq:theta_estimate_kernelmuV0}) we get 
\begin{displaymath}
0\leq k_{\mu,V} (t,x,y) \leq \left(\frac{c e^{-\nu t}}{
    t^{\frac{N}{2\mu}}}\right)^\theta k_{\mu}^{1-\theta}
(t,x,y) 
= c_{\mu,\theta}\frac{ e^{-\nu \theta t}}{  t^{\frac{N}{2\mu}}} I_{\mu}^{1-\theta}\left(\frac{x-y}{t^{\frac{1}{2\mu}}}\right) , \quad \quad t>0 
\end{displaymath}
for some constant $c_{\mu,\theta}>0$.
Consequently, denoting 
\begin{displaymath}
J_{\mu,\theta,t}(\cdot)=\frac{1}{t^\frac{N}{2\mu}}
 I_{\mu}^{1-\theta}\left(\frac{\cdot}{t^{\frac{1}{2\mu}}}\right) ,
\end{displaymath} 
we have for $u_0\in M^{p,\ell}(\R^N)$, $1\leq p<\infty$, $0<\ell\leq
N$ and $x\in \R^{N}$, 
\begin{displaymath}
|S_{\mu,V}(t) u_{0}| (x) \leq  S_{\mu,V}(t)    |u_{0}|(x)
    \leq c_{\mu,\theta} e^{ -\nu \theta t} \int_{\R^N} J_{\mu,\theta,t}(y)\, |u_0 (x-y)| \, dy , \quad t>0 .
\end{displaymath}
Then, using the  generalized Minkowski's inequality, see \cite[\S18.1]{1985Nikolsky}, we see that in a ball
$B(x_0,R)\subset \R^N$ 
\begin{displaymath}
\begin{split}
  \|S_{\mu, V}(t) u_0\|_{L^p(B(x_0,R))} &\leq
  c_{\mu,\theta} e^{ - \nu \theta t} 
   \left(\int_{B(x_0,R)}\left(\int_{\R^N} J_{\mu,\theta,t}(y)  |u_0(x-y)| \, dy \right)^p dx \right)^\frac{1}{p} 
   \\
   &
   \leq    c_{\mu,\theta} e^{ -\nu \theta  t} 
   \int_{\R^N}\left(\int_{B(x_0,R)} \left(J_{\mu,\theta,t}(y)   |u_0(x-y)| \right)^p\, dx \right)^\frac{1}{p} dy
   , \quad t>0 ,
   \end{split}
\end{displaymath}
which gives
\begin{equation}\label{eq:4nd-kmuV0-estimate} 
R^\frac{\ell-N}{p}
   \|S_{\mu, V}(t) u_0\|_{L^p(B(x_0,R))} \leq c_{\mu,\theta} e^{ -\nu \theta t} 
 \|u_0 \|_{M^{p,\ell}(\R^N)}  \int_{\R^N} J_{\mu,\theta,t}(y) \, dy , \quad t>0 , \ R>0 .
 \end{equation}
If $0<\theta< \frac{1}{1+\frac{N}{2\mu}}$ then $\int_{\R^N} J_{\mu,\theta,t}(y) \, dy <\infty$, and then (\ref{eq:4nd-kmuV0-estimate}) yields
\begin{displaymath}
\|S_{\mu,V}(t)  \|_{\mathcal{L}(M^{p,\ell}(\R^N)) }\leq c_{\mu,\theta}
e^{ - \nu \theta t}, \quad t>0 .
\end{displaymath} 
Therefore, for  $0<\theta< \frac{1}{1+\frac{N}{2\mu}}$ we have 
$\nu \theta \in \Omega_{p,\ell}$, so that $\omega_{p,\ell} \geq \nu\theta$.
Since $0<\nu<\omega$ and
$0<\theta<\frac{1}{1+\frac{N}{2\mu}}$ are arbitrary, we get that $\omega_{p,\ell}\geq
\frac{1}{1+\frac{N}{2\mu}}\omega$ and therefore  $\omega_{p,\ell}\geq
\frac{1}{1+\frac{N}{2\mu}}\omega_\infty$. 

\smallskip

\noindent
\fbox{Case $\mu=1$.} 
Now the heat kernel is  given by (\ref{eq:kernel-heat-smgp}), that is,  
\begin{displaymath} 
  k_{1}(t,x,y)= \frac{e^{-\frac{|x-y|^2}{4t}}}{(4\pi
      t)^\frac{N}{2}}  , \quad t>0 , \ x,y\in \R^N ,
\end{displaymath}
and then (\ref{eq:theta_estimate_kernelmuV0}) gives 
\begin{displaymath}
\begin{split}
0\leq k_{1,V} (t,x,y) \leq \left(\frac{c e^{-\nu t}}{  t^{\frac{N}{2}}}\right)^\theta k_{1}^{1-\theta} (t,x,y)
&=c^{\theta} (4\pi)^{\frac{\theta N}{2}}{\scriptstyle(1-\theta)}^{-\frac{N}{2}} e^{ - \nu \theta t}
\frac{e^{-\frac{|x-y|^2}{4{\scriptstyle(1-\theta)^{-1}}t}}}{(4\pi {\scriptstyle(1-\theta)^{-1}}t)^\frac{N}{2}}
\\
&
=c_\theta
e^{ -\nu \theta t}
k_{1} ({\scriptstyle(1-\theta)^{-1}}t, x,y) , \quad \quad t>0
\end{split}
\end{displaymath}
where $c_\theta=c^{\theta} (4\pi)^{\frac{\theta N}{2}} (1-\theta)^{-\frac{N}{2}}$.

Then for $u_0\in M^{p,\ell}(\R^N)$ we have 
\begin{displaymath}
|S_{1,V}(t) u_{0}| \leq  S_{1,V}(t)
    |u_{0}|
    \leq
c_{\theta} e^{ -\nu \theta t} \int_{\R^N} k_{1} ({\scriptstyle(1-\theta)^{-1}}t,\cdot,y)\, |u_0 (y)| \, dy
    =
c_{\theta} e^{ - \nu \theta t} S_{1}({\scriptstyle(1-\theta)^{-1}}t) |u_{0}| , \quad t>0 ,
\end{displaymath}
and using  Lemma \ref{lem:about-contractions} we get 
\begin{displaymath}
\|S_{1,V}(t) u_{0} \|_{M^{p,\ell}(\R^N) }
\leq
c_\theta
e^{ - \nu \theta t} \| S_{1}({\scriptstyle(1-\theta)^{-1}}t) |u_{0}| \|_{M^{p,\ell}(\R^N)}
\leq c_\theta
e^{ - \nu \theta t} \| u_{0}  \|_{M^{p,\ell}(\R^N)}, \quad t>0 .
\end{displaymath}
Therefore  $\nu \theta \in \Omega_{p,\ell}$ and $\omega_{p,\ell} \geq \nu\theta$.
Since $0<\nu<\omega$ and $0<\theta<1$ are arbitrary we get that
$\omega_{p,\ell}\geq \omega$ and therefore $\omega_{p,\ell}\geq
\omega_\infty$.

 \smallskip

\noindent
{\bf Step 3.}  In this step we show that if $\omega_{p,\ell}=0$  for some $1 \leq  p < \infty$ and $0<\ell\leq N$
then $\omega_{p,\ell}=0$ for all $1 \leq  p \leq \infty$ and $0<\ell\leq N$.

For this we first observe that if $\omega_{p,\ell}=0$  for some $1 \leq  p < \infty$ and $0<\ell\leq N$
then $\omega_{\infty}=0$. Otherwise, $\omega_{\infty}>0$ and then by Step 2, $\omega_{p,\ell}>0$ for all $1\leq
 p< \infty$ and $0<\ell\leq N$, a contradiction. 

But then, if $\omega_{\infty}=0$, by Step 1,  then $\omega_{p,\ell}=0$ for all $1\leq
 p< \infty$ and $0<\ell\leq N$. 

\smallskip

\noindent
{\bf Step 4.} In this step, we prove
(\ref{eq:the-same-exp-type}) and (\ref{eq:the-exp-type-when-0<mu<1}).
From Step 2, it is enough to prove the other inequalities and from
Step 1 and 3 it is enough to consider the case
$\omega_{p,\ell}>0$. In such a case we take any $0<\omega \in \Omega_{p,\ell}$ so that we have for some $M\geq 1$
\begin{equation}\label{eq:omega-in-Omegapell} 
\|S_{\mu, V} (t)\|_{\mathcal{L}(M^{p,\ell}(\R^N))} \leq
 M e^{-\omega t}, \quad t> 0 .
\end{equation}

\noindent
\fbox{Case $0<\mu<1$}
We show that if  $\omega_{p,\ell}>0$ then $\omega_\infty \geq \frac{\omega_{p,\ell}}{1+\frac{\ell}{2p\mu}}$.

Then as in the proof of Theorem
\ref{thm:necessity-of-Arendt-Batty-condition} with $q=p$ and $s=\ell$ we get  from
(\ref{eq:estimateLinfty_from_Mpell_mu}) 
\begin{displaymath}
    \|S_{\mu,V}(t)\|_{\mathcal{L}(L^\infty(\R^N))} \leq  c  \|S_{\mu, V}
  (t-1)\|_{\mathcal{L}(M^{p,\ell}(\R^N))}^{\frac{2p\mu}{\ell+2p\mu}} t^{\frac{\ell}{\ell + 2p\mu}} , \quad t>1  .
\end{displaymath}
Then we see that there is a constant $M$ such that for all $t> 1$
\begin{equation}\label{eq:for-mu-estimate-for-tgeq1}
\|S_{\mu, V} (t)\|_{\mathcal{L}(L^\infty(\R^N))}\leq M e^{-\nu t} 
\end{equation}
 for $0<\nu< \frac{2p\mu \omega}{\ell+2p\mu}$. 

Since for $0\leq t \leq 1$ we have from (\ref{eq:domination}) and
Lemma \ref{lem:about-contractions} that
$\|S_{\mu, V} (t)\|_{\mathcal{L}(L^\infty(\R^N))} \leq 1$, increasing $M$ if necessary we see that
(\ref{eq:for-mu-estimate-for-tgeq1}) holds for all $t\geq0$, so that
$\nu\in\Omega_\infty$ and $\omega_\infty\geq \nu$.  Since
$0<\nu<\frac{2p\mu \omega}{\ell+2p\mu}$ is arbitrary we get
$\omega_\infty \geq \frac{2p\mu \omega}{\ell+2p\mu}$, and so
$\omega_\infty \geq \frac{2p\mu \omega_{p,\ell}}{\ell+2p\mu}$.

That  $\vartheta=1$ when  $\ell=N$ was obtained in
\cite{C-RB-exponential_decay} since $M^{p,N}(\R^{N}) =
L^{p}(\R^{N})$.

\noindent
\fbox{Case $\mu=1$} 
We show that if
$\omega_{p,\ell}>0$ then $\omega_\infty \geq \omega_{p,\ell}$. 

Then as in the proof of Theorem
\ref{thm:necessity-of-Arendt-Batty-condition} with $q=p$ and $s=\ell$ we get  from
(\ref{eq:estimateLinfty_from_Mpell}) 
\begin{displaymath}
  \|S_{1,V}(t)\|_{\mathcal{L}(L^\infty(\R^N))} \leq  c \big( \|S_{1, V}
  (t-1)\|_{\mathcal{L}(M^{p,\ell}(\R^N))} R^{\frac{\ell}{p}} + c e^{- \frac{R^2}{2t}} \big), \quad t>1.
\end{displaymath}
for $R>0$.

Taking $R=\sqrt{2\omega} t$ and using  (\ref{eq:omega-in-Omegapell}) we see that 
there is a constant $M$ such that for all $t> 1$
\begin{equation}\label{eq:estimate-for-tgeq1}
\|S_{1, V} (t)\|_{\mathcal{L}(L^\infty(\R^N))}\leq M e^{-\nu t} .
\end{equation}
for  $0<\nu<\omega$. 

Since for $0\leq t \leq 1$ we have from (\ref{eq:domination}) and
Lemma \ref{lem:about-contractions}  that $\|S_{1, V}
(t)\|_{\mathcal{L}(L^\infty(\R^N))} \leq 1$,
increasing $M$ if necessary we see that (\ref{eq:estimate-for-tgeq1}) holds for all $t\geq0$, so that $\nu\in\Omega_\infty$ and $\omega_\infty\geq \nu$.
Since $0<\nu<\omega$ is arbitrary we get
$\omega_\infty \geq \omega$, and so $\omega_\infty \geq
\omega_{p,\ell}$. 

The rest is immediate. 

That for measures, $\omega_{\ell} = \omega_{1,\ell}$ follows from
Corollary \ref{cor:exponential_type_4_measures}. 
\end{proof}

Observe that from Theorem \ref{thm:eponential-type-in-Morrey}  the
exponential decay of solutions in Morrey spaces is somehow determined
by the exponential decay in $L^{\infty}(\R^{N})$. Therefore our goal is now to have a more
precise information on $\omega_{0}$ in the estimate 
\begin{equation}\label{eq:decay-in-Linfty}
    \|S_{\mu,V}(t) \|_{\mathcal{L}(L^\infty(\R^N))}\leq C_0  e^{-\omega_0 t}, \quad t\geq0.
  \end{equation}

  \begin{proposition}
    \label{prop:exponential_decay_Linfty}

   Assume $0<\mu \leq 1$ and $0 \geq V\in
  L^{\infty}(\R^N)$ satisfies  (\ref{eq:Arendt-Batty-condition}).

  Therefore in (\ref{eq:decay-in-Linfty}) we can take $\omega_0$ as
\begin{displaymath}
  \omega_0=  c \|V \|_{L^\infty(\R^N)}
\end{displaymath}
with
\begin{equation}\label{eq:cmuV0}
c= -\frac{1}{\theta} \ln\left(1-(1-\theta) \inf_{\scriptstyle \R^N}
  \{\Psi_{\theta,\mu,V}  \} \right) >0
  \end{equation}
and $C_0$ as
\begin{equation}\label{eq:C0}
C_0= \frac{1}{1-(1-\theta) \inf_{\scriptstyle \R^N}
  \{\Psi_{\theta,\mu,V}  \} } > 1
  \end{equation}
where $0<\theta<1$ is arbitrary and $\Psi_{\theta,\mu,V}$ is given by
\begin{equation}
  \label{eq:PsiV0}
\Psi_{\theta,\mu,V}(x) =\int_0^{\theta/{\scriptscriptstyle \|V
    \|_{L^\infty(\R^N)}}} S_\mu(s) |V|(x) ds , \qquad x \in \R^{N}
\end{equation}
and satisfies
\begin{equation}\label{eq:about-Psitheta,mu,V0}
0< \inf_{\ \scriptstyle  \R^N} \{\Psi_{\theta,\mu,V} \} \leq
\Psi_{\theta,\mu,V}(x) \leq \theta <1, \qquad x \in \R^{N} .
\end{equation}

\end{proposition}
\begin{proof}
We use  that  $\|S_{\mu,V}(t_0)
\|_{\mathcal{L}(L^\infty(\R^N))}  =  \|S_{\mu,V}(t_0) 1
\|_{L^\infty(\R^N)}$ and so  if $\|S_{\mu,V}(t_0) 1
\|_{L^\infty(\R^N)} \leq \alpha<1$ then $\|S_{\mu,V}(t)
\|_{\mathcal{L}(L^\infty(\R^N))} \leq \frac{1}{\alpha}
e^{-\frac{\ln(1/\alpha)}{t_0} t}$ for $t\geq0$ (see Lemma
\ref{lem:about-decay} below).  Hence in (\ref{eq:decay-in-Linfty}) we
can take $\omega_0=\frac{\ln(1/\alpha)}{t_0}$ and
$C_0=\frac{1}{\alpha}$.

We proceed now  as in the proof of \cite[Theorem
4.6]{C-RB-exponential_decay}.
  Since $V=-|V|$ and $S_{\mu} (t) 1=1$, then $u(t)=
  S_{\mu,V} (t) 1 $ satisfies
  \begin{displaymath}
0\leq u(t)  = 1 -    \int_{0}^{t} S_{\mu} (t-s) |V| u(s) \, ds .
  \end{displaymath}
  Substituting the expresion above for $u(s)$ inside the integral term
  we get
  \begin{displaymath}
u(t) = 1 -   \int_{0}^{t} S_{\mu} (t-s) |V| \, ds +
\int_{0}^{t} S_{\mu} (t-s) |V|  \int_{0}^{s} S_{\mu} (s-r) |V| u(r) \, dr
\, ds .
  \end{displaymath}

  We use that $0\leq u(s) \leq S_{\mu} (s) 1 = 1 $ and
  $0\leq |V|\leq  \|V\|_{\infty} $ and then
  \begin{displaymath}
  \begin{split}
    \int_{0}^{t} S_{\mu} &(t-s) |V|   \int_{0}^{s} S_{\mu} (s-r) |V| u(r) \,
    dr \, ds
    \leq
 \|V\|_{\infty} \int_{0}^{t} s S_{\mu} (t-s) |V|
\, ds
\\&
= \|V\|_{\infty} \int_{0}^{t} (t-s) S_{\mu} (s) |V|
\, ds
\leq \|V\|_{\infty} t \int_{0}^{t} S_{\mu} (s) |V|
\, ds .
\end{split}
  \end{displaymath}

  Therefore
  \begin{displaymath}
0\leq     u(t) \leq   1  + (\|V\|_{\infty} t -1) \int_{0}^{t} S_{\mu} (s)
|V| \, ds .
  \end{displaymath}

  If we fix $\theta\in(0,1)$ and set
$t_0=\frac{\theta}{\|V\|_{\infty}}$ and with
$\Psi_{\theta,\mu,V}$ as in (\ref{eq:PsiV0})
 we get
\begin{displaymath}
0\leq u(t_{0}) \leq 1 - (1-t_0 \| V \|_{\infty}) \Psi_{\theta,\mu,V}
\end{displaymath}
and therefore
\begin{displaymath}
\qquad
    \|S_{\mu, V} (t_0) \|_{\mathcal{L}(L^\infty(\R^N))} =
    \|u (t_0) \|_{L^\infty(\R^N)}
      \leq
    \alpha \mydef   1- (1-\theta)
     \inf_{\scriptstyle \R^N} \{\Psi_{\theta,\mu,V} \} \leq
     1
\end{displaymath}
and we show below that actually $\alpha <1$.

Since the  semigroup $\{S_\mu(t)\}_{t\geq0}$ is of contractions in
$L^{\infty}(\R^{N})$, see e.g. (\ref{eq:estimates_Mpl-Mpl}) with $p=\infty$,
we get
\begin{displaymath}
\|\Psi_{\theta,\mu,V}\|_{L^\infty(\R^N)} \leq
\int_0^{\theta/{\scriptscriptstyle \|V \|_{\infty}}} \| V \|_{\infty}
ds =\theta.
\end{displaymath}

Now, from Remark \ref{rem:equivalent-2-AB-condition} there exists
$R>0$ such that
\begin{displaymath}
 C_{V} \mydef \inf_{x\in \R^{N}} \int_{B(x,R)} |V| (y) dy >0
\end{displaymath}
and using the kernels for $\{S_\mu(t)\}_{t\geq0}$ in
(\ref{eq:kernel-heat-smgp}) or (\ref{eq:kernel-fractional-heat-smgp})
we conclude that for $x\in \R^{N}$
\begin{displaymath}
  \begin{split}
    \Psi_{\theta,\mu,V}(x)& \geq \int_0^{\theta/\|V\|_{\infty}}
    \int_{B(x,R)} \frac{1}{s^\frac{N}{2\mu}}
    k_{0,\mu}\left(\frac{x-y}{s^{\frac{1}{2\mu}}}\right) |V|(y)\, dy
    ds
    \\ &
    \geq \int_0^{\theta/\|V\|_{\infty}} \inf_{|z|\leq R}
    \frac{1}{s^\frac{N}{2\mu}}
    k_{0,\mu}\left(\frac{z}{s^{\frac{1}{2\mu}}}\right) \int_{B(x,R)}
    |V| (y) dy
  \end{split}
\end{displaymath}
and hence
\begin{equation}\label{eq:inf_Psi}
\inf_{\scriptstyle x\in \R^N} \{\Psi_{\theta,\mu,V}(x) \} \geq
 C_{V} \int_0^{\theta/\|V\|_{\infty}} g_\mu(s) \, ds >0
\end{equation}
for
\begin{displaymath}
g_\mu (s) =
    \begin{cases}
     \frac{ e^{-\frac{R^2}{4s}}}{(4\pi s)^{\frac{N}2}} , &
      \mu=1 ,
      \\
     \frac{c s}{(s^\frac{1}{\mu} + R^2)^\frac{N+2\mu}{2}} , &
      0<\mu<1.
    \end{cases}
\end{displaymath}
  In particular  $\Psi_{\theta,\mu,V}$ satisfies
  (\ref{eq:about-Psitheta,mu,V0}) and $\alpha <1$.

Therefore in (\ref{eq:decay-in-Linfty}) we can take
$\omega_0=\frac{\ln(1/\alpha)}{t_0}= c \|V \|_{\infty}$,
with $c$ as in (\ref{eq:cmuV0}),  and $C_0=\frac{1}{\alpha} >1$ as in
(\ref{eq:C0}).
\end{proof}

\begin{remark}
  If the potential $V$ is a negative constant then the estimate in
  Proposition \ref{prop:exponential_decay_Linfty} recovers the optimal
  estimate
    \begin{displaymath}
   \|S_{\mu,V}(t) \|_{\mathcal{L}(L^\infty(\R^N))}\leq  e
   ^{-|V| t}, \quad t\geq0
 \end{displaymath}
that follows from the fact that
  $S_{\mu,V}(t)  = e^{-|V| t}S_{\mu}(t)$.

To see this observe that in this case $S_\mu(s)|V|= |V|
S_\mu(s)1=|V|$ and then in  (\ref{eq:PsiV0}) we get
$\Psi_{\theta,\mu,V}=\theta$ and so  in   (\ref{eq:cmuV0}) and  (\ref{eq:C0}) we have
\begin{displaymath}
  c =    -\frac{1}{\theta}\ln\left(1-(1-\theta) \theta \right),
  \quad
    C_0= \frac{1}{1-(1-\theta) \theta } .
  \end{displaymath}

So we prove that
  \begin{equation}\label{eq:-about-the-right-hand-side-in-{cmuV0}}
    \sup_{0<\theta<1} \{-\frac{1}{\theta}\ln\left(1-(1-\theta) \theta
    \right)\} =\lim_{\theta\to0^+} \frac{-\ln\left(1-(1-\theta) \theta
      \right)}{\theta}= 1
  \end{equation}
and then  $c\to 1$ and $C_0\to 1$ as $\theta\to 0^+$ and  this proves our claim.

Now we prove (\ref{eq:-about-the-right-hand-side-in-{cmuV0}})
  by showing that for every $\theta\in(0,1)$ we have
  $\ln{(1-(1-\theta)\theta)}>-\theta$. This is equivalent to proving
  $\xi(\theta)=\theta^2-\theta+1-e^{-\theta} >0$ for
  $\theta\in(0,1)$ which  follows since   $\xi''(\theta)>0$.

\end{remark}

\begin{remark}
  \begin{enumerate}
  \item
    Observe that in (\ref{eq:inf_Psi}) the term $C_{V}$
    depends on the spatial distribution of the mass of $|V|$ while
    the term $\int_0^{\theta/\|V\|_{\infty}} g_\mu(s) \, ds$
    depend only on the size of $|V|$ in $L^{\infty}(\R^{N})$.

  \item
  Notice that for $x\sim 0$ we have $\ln(1-x) \sim -x$ and therefore
  for $\theta \sim 0$, using (\ref{eq:about-Psitheta,mu,V0}), we have
  \begin{displaymath}
    \frac{-\ln\left(1-(1-\theta) \inf_{\scriptstyle
          \R^N} \{\Psi_{\theta,\mu,V} \} \right)}{\theta} \sim
    \frac{\inf_{\scriptstyle
        \R^N} \{\Psi_{\theta,\mu,V} \} }{\theta} \leq 1 .
  \end{displaymath}
  So if, as in the case of a constant potential, we have
  \begin{equation}\label{eq:addition-condition}
    \liminf_{\theta\to0^+} \frac{\inf_{\scriptstyle
        \R^N} \{\Psi_{\theta,\mu,V} \} }{\theta} =:\omega^*>0
  \end{equation}
then  taking the limit as  $\theta\to0^+$ in the constant
(\ref{eq:cmuV0})  and (\ref{eq:C0}), as actually  $C_0\to 1$ as $\theta\to 0^+$ we get
  \begin{displaymath}
    \|S_{\mu,V}(t) \|_{\mathcal{L}(L^\infty(\R^N))}\leq e ^{-\omega^* {\scriptscriptstyle \|V \|_{\infty}} t}, \quad t\geq0.
  \end{displaymath}
  Then, using that the semigroup in $L^\infty(\R^N)$ is the adjoint of
  the semigroup in $L^1(\R^N)$ (so they have equal norms) and using
  the Riesz-Thorin interpolation theorem
  we see that
for any  $1\leq p \leq \infty$ we have
  \begin{displaymath}
    \|S_{\mu,V}(t) \|_{\mathcal{L}(L^p(\R^N))}\leq e ^{-\omega^*
      {\scriptscriptstyle \|V \|_{\infty}} t}, \quad t\geq0 .
  \end{displaymath}

\item
However notice that for the lower bound in  (\ref{eq:inf_Psi}) we
have,  as
  $\theta \to  0$
  \begin{displaymath}
    \frac{1}{\theta}  \int_0^{\theta/\|V\|_{\infty}} g_\mu(s) \, ds
\to 0
  \end{displaymath}
  since $g_{\mu}(s)$ vanishes at least linearly as $s \to 0$. So this
  lower bound does not imply (\ref{eq:addition-condition}).

\end{enumerate}
\end{remark}

Finally we prove the result used in the proof of Proposition
\ref{prop:exponential_decay_Linfty}.

\begin{lemma}\label{lem:about-decay}
Assume $X$ is a Banach space, $\{S(t)\}_{t\geq0}\subset \mathcal{L}(X)$ is a semigroup of contractions and $\|S(t_0)\|_{\mathcal{L}(X)}\leq \alpha <1$ for some $t_0>0$.

Then
\begin{displaymath}
\|S(t)\|_{\mathcal{L}(X)} \leq \frac{1}{\alpha} e^{-\frac{\ln(1/\alpha)}{t_0} t}, \quad t\geq0.
\end{displaymath}
\end{lemma}
\begin{proof}
Given any $s\in[0,t_0)$ and a nonnegative integer $k$ we see that $\|S(k t_0 + s)\|_{\mathcal{L}(X)}\leq \alpha^{k}$ and letting $t=k t_0 + s$ we get
\begin{displaymath}
\|S(t)\|_{\mathcal{L}(X)}\leq  e^{\frac{t-s}{t_0} \ln{\alpha}}\leq e^{\frac{s}{t_0}\ln{(1/\alpha)}} e^{-\frac{\ln{(1/\alpha)}}{t_0}t} \leq \frac{1}{\alpha} e^{-\frac{\ln{(1/\alpha)}}{t_0}t} ,
\end{displaymath}
where in the last inequality above we used that $\frac{s}{t_0}<1$.
\end{proof}

Now for a general bounded potential we get the following estimate in
Morrey spaces which shows how the positive and negative parts of the
potential influence the behavior of solutions for large times.

\begin{corollary} \label{cor:growth_4_bounded_potential}

   Assume $0<\mu \leq 1$, $ V\in
   L^{\infty}(\R^N)$ and denote $0\leq \omega =
   \omega_{\infty}(V^{-})$ the exponential type for the negative
   part $V^{-}$ of $V$,  $a=
   \|V^+\|_{L^\infty(\R^N)}$ and $\vartheta= 1+\frac{N}{2\mu}$.

Then for
$1\leq p,q\leq \infty$, $0<s\leq \ell\leq N$ such that
$\frac{s}{q}\leq  \frac{\ell}{p}$ and  any $\eps>0$, 
we have 
\begin{displaymath}
  \| S_{\mu, V}(t)\|_{\mathcal{L}(M^{p,\ell}(\R^N), M^{q,s}(\R^N))}
  \leq
  \frac{Ce^{ (a -  \frac{\omega}{\vartheta}+\eps) t }}{t^{\frac{1}{2\mu}(\frac{\ell}{p}-\frac{s}{q})}},
  \quad t>0 . 
\end{displaymath}

\end{corollary}
\begin{proof}
For  $u_0\in M^{p,\ell}(\R^N)$, from
Theorem \ref{thm:monotonicity-with-respect-to-potentials}
and  Proposition \ref{add_exponential_to_semigroup},  we get
\begin{displaymath}
    |S_{\mu,V}(t) u_{0}|
    \leq S_{\mu,\{-V^{-}+a\}}(t) |u_{0}|
    = e^{a t} S_{\mu,-V^{-}}(t)|u_{0}| , \quad t\geq 0 . 
\end{displaymath}

From Theorem \ref{thm:eponential-type-in-Morrey}, for any $\eps>0$, $- \frac{\omega
}{\vartheta} +\eps$ with  $\vartheta= 1+\frac{N}{2\mu}$, is a common
exponential bound for $S_{\mu,-V^{-}}(t)$ in all Morrey
spaces and then $a+\eps- \frac{\omega}{\vartheta}$ is a common
exponential bound for $S_{\mu,V}(t) $ in all Morrey spaces.

Then we can use Lemma 6.2 in \cite{C-RB-morrey_linear_perturbation} to
get the result. 
\end{proof}

\subsection{An admissible  Morrey Potential}
\label{sec:dissipative-admissible-morrey}

Now we consider a dissipative admissible Morrey potential $0\geq
V\in M^{p_{1},\ell_{1}}(\R^{N})$  as in (\ref{eq:admisible_potential}) with $m=1$. 

Then for $M>0$,  we can truncate at height $-M$ and define 
\begin{displaymath}
0\geq   V_{M}(x) = \max\{V(x), -M\} \geq V(x) , \quad x\in \R^{N}
\end{displaymath}
so $V_{M}^{1} \in L^{\infty}(\R^{N}) \cap M^{p_{1},\ell_{1}}(\R^{N})$
and   from Theorem \ref{thm:monotonicity-with-respect-to-potentials}
for $u_{0} \in M^{p_{1},\ell_{1}}(\R^N)$ with $1\leq p \leq
\infty$, $0<\ell\leq \ell_{1}$ or $u_{0}\in \mathcal{M}^{\ell} (\R^{N})$,  we have
\begin{displaymath} 
  |S_{\mu, V}(t) u_{0}| \leq S_{\mu, V_{M}}(t)|u_{0}|, \quad
  t\geq 0 . 
\end{displaymath}
In particular, $\omega_{p, \ell}(V) \geq \omega_{p,
  \ell}(V_{M}) \geq 0$.

Then we have the following result.

\begin{proposition} \label{prop:decay_4_dissipative_morrey_potential}

  With the notations above,
  
  \begin{enumerate}
  \item
 If for  some $M$,  $V_{M}$ satisfies
(\ref{eq:Arendt-Batty-condition}), then  $\omega_{p, \ell}(V)
>0$. 

\item 
    Assume
    $0\geq V \in M^{p_{1},\ell_{1}}(\R^{N})\subset L^{1}_{U}(\R^{N})$
    can be approximated in $L^{1}_{U}(\R^{N})$ by bounded
    functions. This holds in particular if $V \in
    \dot{L}^{1}_{U}(\R^{N})$, see Appendix \ref{sec:expon-decay-uniform}. 

    Then if $V$ satisfies (\ref{eq:Arendt-Batty-condition}), then
    $\omega_{p, \ell}(V) >0$.
  \end{enumerate}
\end{proposition}
\begin{proof}
(i) This is immediate as  $\omega_{p, \ell}(V) \geq \omega_{p,
  \ell}(V_{M}) \geq 0$ and (\ref{eq:Arendt-Batty-condition}) and
Theorem \ref{thm:eponential-type-in-Morrey} imply $\omega_{p,
  \ell}(V_{M})>0$.  
  
\noindent (ii)
  From the assumption on $V$, from Proposition 4.7 in
\cite{C-RB-exponential_decay} we get that if  $V$ satisfies
(\ref{eq:Arendt-Batty-condition}), then for some $M$,  $V_{M}$
also satisfies (\ref{eq:Arendt-Batty-condition}), and then we get the
result.

When $V \in    \dot{L}^{1}_{U}(\R^{N})$, from Proposition 3.5 in
\cite{C-RB-exponential_decay} we get that $V$ satisfies the assumption
in (ii). 
\end{proof}

\begin{remark} \label{rem:exponential_type_in_Lebesgue}
  
  \leavevmode
  
\begin{enumerate}
\item

If
$0\geq V\in M^{p_{1},\ell_{1}}(\R^{N}) \subset
L^{p_{1}}_{U}(\R^{N})$ is a dissipative  admissible potential as in
(\ref{eq:admisible_potential}) with $m=1$ but additionally
satisfies that $p_{1}> \frac{N}{2\mu} \geq \frac{\ell_{1}}{2\mu}$
then, from \cite{C-RB-exponential_decay}, the semigroup
$S_{\mu, V}(t)$ is also well defined in the Lebesgue spaces
$L^{p}(\R^{N})$ with $1\leq p\leq \infty$. In particular, the
exponential type in this spaces, $\omega_{p}$, can be defined as
in the beginning of  Section \ref{sec:exponential-decay}. This happens in particular
if $V$ is bounded as in Section \ref{sec:bounded-potential}.

Then it was proved in \cite[Theorem 4.1]{C-RB-exponential_decay}.  For
$0<\mu <1$ and $1\leq p\leq \infty$, the exponential type
satisfies
\begin{displaymath}
\omega_{2} \geq  \omega_{p}\geq
\omega_{\infty}  \geq \frac{ \omega_{2} }{1+\frac{N}{4\mu}}
\geq 0.
\end{displaymath}
while for $\mu=1$, $\omega_{p} = \omega_{2}$ for all
$1\leq p \leq \infty$, see also \cite{simon82:_schroed} for this
latter case. Compare this with Theorem
\ref{thm:eponential-type-in-Morrey}. Also, notice that the
inequalities above for $0<\mu<1$ can be combined with the ones in
Theorem \ref{thm:eponential-type-in-Morrey}. 

In particular, either $ \omega_{p}=0$ for all $1\leq p\leq \infty$
or they are all positive simultaneously. In such a case there
exists a common decay rate for the semigroup
$\{S_{\mu, V}(t)\}_{t\geq0}$ independent of the Lebesgue space.

It was also proved in \cite[Theorem 4.6]{C-RB-exponential_decay} that
for $0 < \mu\leq 1$ 
\begin{displaymath}
\omega_{2} =  \inf\{\int_{\R^{N}} |(-\Delta)^{\frac{\mu}{2}} \phi|^2
    + \int_{\R^{N}} |V| |\phi|^2 \colon \phi\in C^\infty_c(\R^N), \
    \|\phi\|_{L^2(\R^N)}=1 \}  \geq 0.
  \end{displaymath}

\item
If $V$ is as in (i) above, assume $\omega_{\infty}>0$ and then for
some constants $C_{0}$ and $\omega_{0}>0$ 
\begin{displaymath} 
    \|S_{\mu,V}(t) \|_{\mathcal{L}(L^\infty(\R^N))}\leq C_0  e^{-\omega_0 t}, \quad t\geq0.
  \end{displaymath}
  
    Then for every $1\leq p,q\leq \infty$, $0<s\leq \ell\leq N$
    satisfying $\frac{s}{q}< \frac{\ell}{p}$ and $\eps>0$, 
    \begin{displaymath}
      \|S_{\mu, V
      }(t)\|_{\mathcal{L}(M^{p,\ell}(\R^N),M^{q,s}(\R^N))}\leq
      \frac{Ce^{(-\omega +\eps) t}}{t^{\frac{1}{2\mu}(\frac{\ell}{p}-\frac{s}{q})}},
  \quad t>0 
    \end{displaymath}
with $\omega = (1-\frac{ps}{q \ell})\omega_0 >0$.  Notice that this
estimate gets worse as $(q,s) \to (p,\ell)$. 
    
To see this, for $1\leq p<\infty$ and  $q,s,\ell$ as in the statement, we take
  $r> p$ such that $\frac{s}{q}= \frac{\ell}{r}< \frac{\ell}{p}$,
  which implies $r\geq q$ and $M^{r,\ell}(\R^N)\subset M^{q,s}(\R^N)$,
  see \cite[formula (2.3)]{C-RB-linear_morrey}.

  Using this, \cite[inequality (2.6)]{C-RB-linear_morrey} and
  (\ref{eq:smoothing_one_potential}), we get that for
  $u_0\in M^{p,\ell}(\R^N)$ and $t\geq1$
  \begin{displaymath}
    \begin{split}
      \| S_{\mu,V}(t) u_0 & \|_{M^{q,s}(\R^N)} \leq C
      \|S_{\mu,V}(t) u_0 \|_{M^{r,\ell}(\R^N)} \leq
      C\|S_{\mu,V}(t) u_0 \|_{L^{\infty}(\R^N)}^{1-\frac{p}{r}} \|
      S_{\mu,V}(t) u_0\|_{M^{p,\ell}(\R^N)}^\frac{p}{r} 
      \\
      & \leq C (\big\|S_{\mu,V}(t-1)
      \|_{\mathcal{L}(L^{\infty}(\R^N))} \|S_{\mu,V}(1)
      u_0\|_{L^{\infty}(\R^N)}\big)^{1-\frac{p}{r}} \|S_{\mu,V}(t)
      u_0\|_{M^{p,\ell}(\R^N)}^\frac{p}{r}
    \end{split}
  \end{displaymath}
  where $C=|B(0,1)|^{1-\frac{q}{r}}= |B(0,1)|^{1-\frac{s}{\ell}}$.

Again   (\ref{eq:smoothing_one_potential}) with $q=\infty$
  gives
  $\|S_{\mu,V}(t) \|_{{\mathcal L}(M^{p,\ell}(\R^N),
    L^{\infty}(\R^N))} \leq c$ for $t\geq 1$ and then from the decay
  assumption,  (\ref{eq:contracting-property-in-Mpl}),
  and the equality $\frac{1}{r}= \frac{s}{q \ell}$ we get 
  \begin{equation}\label{eq:to-define-omega}
    \| S_{\mu,V}(t) u_0 \|_{M^{q,s}(\R^N)}
    \leq C
    e^{-(1-\frac{ps}{q \ell})\omega_0 t} \|u_0\|_{M^{p,\ell}(\R^N)}, \quad t\geq1.
  \end{equation}

  On the other hand, for $0<t<1$ we have from
  (\ref{eq:smoothing_one_potential}) 
  $\| S_{\mu,V}(t)\|_{\mathcal{L}(M^{p,\ell}(\R^N),
    M^{q,s}(\R^N))} \leq
  \frac{C}{t^{\frac{1}{2\mu}(\frac{\ell}{p}-\frac{s}{q})}}$ and using
  this and (\ref{eq:to-define-omega}) we get the result.

\end{enumerate}

\end{remark}

\section{Exponential bounds  for  the  perturbed semigroups. The case of two potentials}
\label{sec:expon-bounds-2-potentials}

Our goal in this section  is to provide estimates on the exponential type of
the perturbed semigroup $\{S_{\mu, \{V^{0}, V^{1}\}}(t)\}_{t\geq 0}$ in terms of
the sizes of the positive and negative parts of the potentials that,
in particular,  will give sufficient conditions for the exponential
decay of the solutions. Our strategy for this is as follows. On the
one hand, the negative parts of the potentials are the ones responsible
for the eventual exponential decay of the solutions. On the other
hand, the positive parts are the ones responsible for the tendency to
produce exponential growth, see Theorem
\ref{thm:exponential-estimate-of-S{mu,V0,V1}}. Hence our goal is to
find a balance among them which will enable us to assert that if the
negative parts are  somehow larger than the positive ones, the
perturbed semigroup decays exponentially.

Now we analyse the situation with the perturbations using two
potentials.
We begin with the following propositions.

  \begin{proposition}
    \label{prop:the-general-prop}
    Assume $0<\mu \leq 1$ and admissible potentials as in
(\ref{eq:for-thm-with-2potentials}) with $m=1$.
Assume also $p\geq p'_{0}$, $0<\ell\leq \ell_0$ and
\begin{displaymath}
\|S_{\mu, V^{1} }(t)\|_{\mathcal{L}(M^{p,\ell}(\R^N))}\leq Ce^{a t} ,  
\end{displaymath}
for some $C>0$, $a\in \R$.

Then we have, for any $\eps>0$, 
\begin{displaymath}
  \| S_{\mu,\{V^{0}, V^{1}\}}(t)\|_{\mathcal{L}(M^{p,\ell}(\R^N))}  \leq C e^{(a
    +\eps +\tilde{a} )t}
\end{displaymath}
and  for
$1\leq p,q\leq \infty$, $0<s\leq \ell\leq \ell_{0}$ such that
$\frac{s}{q}\leq  \frac{\ell}{p}$ 
we have 
\begin{displaymath}
  \| S_{\mu,\{V^{0}, V^{1}\}}(t) \|_{\mathcal{L}(M^{p,\ell}(\R^N), M^{q,s}(\R^N))}
  \leq
  \frac{Ce^{ (a +\eps + \tilde{a}) t }}{t^{\frac{1}{2\mu}(\frac{\ell}{p}-\frac{s}{q})}},
  \quad t>0
\end{displaymath} 
where $\tilde{a}= c  \|(V^{0})^+\|_{M^{p_0,\ell_0}(\R^N)}^{\frac{1}{1-\kappa_{0}}} $
for some constant $c>0$ with $\kappa_0=\frac{\ell_0}{2\mu p_0}<1$.

  \end{proposition}

  \begin{proof}
  By Theorem \ref{thm:monotonicity-with-respect-to-potentials} for $u_{0} \in M^{p,\ell}(\R^N)$
we have $|S_{\mu,\{V^{0},V^{1}\}}(t) u_{0}| \leq S_{\mu,\{(V^{0})^+,V^{1}\}}(t)|u_{0}|$, $t\geq 0$,
whereas by Lemma \ref{lem:fixed-point-lem} $u(t) = S_{\mu,\{(V^{0})^+, V^{1}\}}(t)u_0$ with $u_0 \in M^{p,\ell}(\R^N)$ satisfies
\begin{displaymath}
  u(t)   = S_{\mu,V^{1}}(t)
u_0 +  \int_0^t S_{\mu,V^{1}}(t-\tau)  (V^{0})^+ u(\tau) \,  d\tau .
\end{displaymath}

Now, since $V^{1}$ is admissible, the semigroup $S_{\mu, V^{1} }(t)$
is well defined in the spaces $M^{z,\nu}(\R^N)$ for $1\leq z\leq
\infty$ and $0<\nu \leq \ell_{1}$. In particular, it is well defined
in $M^{p,\ell}(\R^{N})$. 

Define now  $z$,$\nu$  by
$\frac{1}{z}=\frac{1}{p}+\frac{1}{p_0}$ and
$\frac{\nu}{z}=\frac{\ell}{p}+\frac{\ell_0}{p_0}$. Then we claim that  $1\leq z\leq
\infty$ and $0<\nu \leq \ell_{1}$ and 
\begin{equation}\label{eq:two-needed-estimates}
\|S_{\mu, V^{1}}(t)\|_{\mathcal{L}(M^{z,\nu}(\R^N),M^{p,\ell}(\R^N))}\leq
\frac{Ce^{(a +\eps) t}}{t^{\frac{1}{2\mu}(\frac{\nu}{z}-\frac{\ell}{p})}}
\quad t>0 . 
\end{equation}
To see this, first since $p\geq p_{0}'$ we get $1\leq z \leq
\infty$. Second, since $\ell\leq \ell_{0}\leq \ell_{1}$ then
$\frac{\nu}{z} \leq \frac{\ell_{1}}{z}$ which implies $\nu \leq
\ell_{1}$. Finally, to prove the estimate above, observe
that from the smoothing properties, the semigroup $S_{\mu, V^{1}}(t)$
smoothes from  $M^{z,\nu}(\R^N)$  into  $M^{p,\ell}(\R^N)$ provided
$1\leq p\leq \infty$, $\ell \leq \nu$ and $\frac{\ell}{p}\leq \frac{\nu}{z}$,
see (\ref{eq:smoothing_one_potential}). Then observe that the latter
condition is clearly satisfied, while since $\ell\leq \ell_{0}$ then
$\frac{\nu}{z} \geq \frac{\ell}{z}$ which implies $\nu
\geq\ell$. Finally, (\ref{eq:two-needed-estimates}) follows then from
Lemma 6.2 in \cite{C-RB-morrey_linear_perturbation}.

Using (\ref{eq:two-needed-estimates}) and multiplication properties in Morrey spaces in \cite[Lemma
5.1]{C-RB-morrey_linear_perturbation} we get
\begin{displaymath}
\begin{split}
\|S_{\mu, V^{1} }(t-\tau) (V^{0})^+ u(\tau) \|_{M^{p,\ell}(\R^N)}
&\leq \frac{Ce^{(a +\eps) (t-\tau)}}{(t-\tau)^{\frac{1}{2\mu}(\frac{\nu}{z}-\frac{\ell}{p})}}
\|(V^{0})^+ u(\tau) \|_{M^{z,\nu}(\R^N)}
\\
&
\leq
\frac{C \|(V^{0})^+
\|_{M^{p_0,\ell_0}(\R^N)} e^{(a +\eps) (t-\tau)}}{(t-\tau)^{\frac{\ell_0}{2\mu p_0}}}
\|u(\tau) \|_{M^{p ,\ell}(\R^N)}
\end{split}
\end{displaymath}
and then 
    \begin{displaymath}
      \| u(t) \|_{M^{p,\ell}(\R^N)} \leq C  e^{a t}
      \|u_0\|_{M^{p,\ell}(\R^N)} + \int_0^t \frac{C \|(V^{0})^+
\|_{M^{p_0,\ell_0}(\R^N)} e^{(a +\eps) (t-\tau)}}{(t-\tau)^{\frac{\ell_0}{2\mu p_0}}}
\|u(\tau) \|_{M^{p ,\ell}(\R^N)} \, d\tau, \quad t>0 .
    \end{displaymath}
Hence $v(t) = e^{-(a +\eps) t} \|u(t) \|_{M^{p,\ell}(\R^N)}$
    satisfies
    \begin{displaymath}
      v(t)  \leq C   \|u_0\|_{M^{p,\ell}(\R^N)}  + \int_0^t\frac{C \|(V^{0})^+\|_{M^{p_0,\ell_0}(\R^N)} }{(t-\tau)^{\frac{\ell_0}{2\mu p_0}}}
v(\tau)  \, d\tau, \quad t>0 ,
    \end{displaymath}
so \cite[Lemma 7.1.1]{1981_Henry} gives
$v(t) \leq C   \|u_0\|_{M^{p,\ell}(\R^N)}  e^{ \tilde{a} t}$, $t\geq0$, with
$\tilde{a}$ as in the statement and the result is proved. 

The rest comes from Lemma 6.2 in
\cite{C-RB-morrey_linear_perturbation}. 
\end{proof}

In particular for a single admissible potential we get the following
result, see Corollary \ref{cor:growth_4_bounded_potential}. Also
notice that Proposition
\ref{prop:decay_4_dissipative_morrey_potential} gives conditions to
have  $a<0$ below. 

\begin{corollary} \label{cor:growth_4_morrey_potential}
      Assume $0<\mu \leq 1$ and $V\in M^{p_{1},\ell_{1}}(\R^{N})$
      admissible as
      in (\ref{eq:admisible_potential}) with $m=1$ and denote by
      $V^{+}$ and $V^{-}$ the  positive and negative  parts of $V$. 
Assume also $p\geq p'_{1}$, $0<\ell\leq \ell_1$ and
\begin{displaymath}
\|S_{\mu, V^{-} }(t)\|_{\mathcal{L}(M^{p,\ell}(\R^N))}\leq Ce^{a t} ,  
\end{displaymath}
for some $C>0$, $a\in \R$.

Then, for any $\eps>0$,  we have 
\begin{displaymath}
  \| S_{\mu, V}(t)\|_{M^{p,\ell}(\R^N)} \leq C e^{(a+\eps +\tilde{a} )t}
\end{displaymath}
and  for
$1\leq p,q\leq \infty$, $0<s\leq \ell\leq \ell_{1}$ such that
$\frac{s}{q}\leq  \frac{\ell}{p}$ 
we have 
\begin{displaymath}
  \| S_{\mu, V}(t)\|_{\mathcal{L}(M^{p,\ell}(\R^N), M^{q,s}(\R^N))}
  \leq
  \frac{Ce^{ (a+\eps +\tilde{a}) t }}{t^{\frac{1}{2\mu}(\frac{\ell}{p}-\frac{s}{q})}},
  \quad t>0
\end{displaymath} 
where $\tilde{a}= c  \|V^{+}\|_{M^{p_1,\ell_1}(\R^N)}^{\frac{1}{1-\kappa_{1}}} $
for some constant $c>0$ with $\kappa_1=\frac{\ell_1}{2\mu p_1}<1$.

In particular, if $a<0$ and $V^{+}$ is small in
$M^{p_1,\ell_1}(\R^N)$ then $S_{\mu, V}(t)$ decays exponentially in
$M^{p,\ell}(\R^N)$. 
\end{corollary}
\begin{proof}
  We take $V^{1} = V^{-}$ and $V^{0}= V^{+}$, the negative and
  positive parts of $V$ and $p_{0}=p_{1}$, $\ell_{0}=\ell_{1}$ and
  apply Proposition \ref{prop:the-general-prop} observing that
  $S_{\mu,\{V^{0}, V^{1}\}}(t) = S_{\mu, V}(t)$, see Remark
  \ref{rem:adding_potentials}. 

  The rest comes from Lemma 6.2 in
  \cite{C-RB-morrey_linear_perturbation}. The statement about the decay
  is immediate. 
\end{proof}

\begin{corollary}\label{cor:some-implication-of-Arendt-Batty-condition}

  Assume $0<\mu \leq 1$ and $V^{0}\in M^{p_0,\ell_0}(\R^N)$ is an
  admissible potential as in (\ref{eq:admisible_potential}) with $m=1$. Assume
  also $V^1 \in L^{\infty}(\R^{N})$ and the negative part
  $(V^{1})^{-}$ satisfies    (\ref{eq:Arendt-Batty-condition}). Let $\omega =
  \omega_{\infty}((V^{1})^{-}) >0$ the exponential type  of the
  negative part of $V^{1}$ and $\vartheta= 1+\frac{N}{2\mu}$. 

Then given $1\leq p,q\leq \infty$, $0<s\leq \ell\leq \ell_0$ such that
$\frac{s}{q} \leq  \frac{\ell}{p}$, $p\geq p_0'$ and $\eps>0$ 
we have
\begin{displaymath}
\|  S_{\mu,\{V^{0},V^{1}\}}(t)  \|_{\mathcal{L}(M^{p,\ell}(\R^N),M^{q,s}(\R^N))}
\leq   \frac{Ce^{ (a +\eps - \frac{\omega}{\vartheta}) t }}{t^{\frac{1}{2\mu}(\frac{\ell}{p}-\frac{s}{q})}},
  \quad t>0
\end{displaymath}
with
\begin{displaymath}
a=  \|(V^{1})^+\|_{L^\infty(\R^N)}+
c_{0}\|(V^{0})^+\|_{M^{p_0,\ell_0}(\R^N)}^{\frac{1}{1-\kappa_{0}}} ,
\end{displaymath}
$\kappa_0=\frac{\ell_{0}}{2\mu p_{0}}<1$, some $c_{0}$ independent of
$V^{0}$.

\end{corollary}
\begin{proof}
  For  $u_0\in M^{p,\ell}(\R^N)$, setting
$a_1=\|(V^{1})^{+}\|_{L^\infty(\R^N)}$, from
Theorem \ref{thm:monotonicity-with-respect-to-potentials}
and  Proposition \ref{add_exponential_to_semigroup},  we get
\begin{displaymath}
  |S_{\mu,\{V^{0},V^{1}\}}(t) u_{0}| \leq S_{\mu,\{
    -(V^{1})^{-}+a_1,(V^{0})^+\}}(t)|u_{0}|= e^{a_1 t}
  S_{\mu,\{-(V^{1})^-,(V^{0})^+\}}(t)|u_{0}| , \quad t\geq 0 . 
\end{displaymath}

From Theorem \ref{thm:eponential-type-in-Morrey}, for any $\eps>0$, $- \frac{\omega
}{\vartheta} +\eps$ with  $\vartheta= 1+\frac{N}{2\mu}$, is a common
exponential bound for $S_{\mu,-(V^{1})^{-}}(t)$ in all Morrey
spaces and then 
\begin{displaymath}
\|S_{\mu, -(V^{1})^{-} }(t)\|_{\mathcal{L}(M^{p,\ell}(\R^N))}\leq Ce^{(\eps-\frac{\omega}{\vartheta}) t} ,  
\end{displaymath}
Now applying Proposition \ref{prop:the-general-prop} we get the
result. 
\end{proof}

We close this section with the following useful remark. 

\begin{remark}
It is worth noticing that from  Lemma 6.2 in
  \cite{C-RB-morrey_linear_perturbation}, that we have used several
  times above, for $0<\mu \leq 1$ and admissible potentials as in
(\ref{eq:for-thm-with-2potentials}) with $m=1$,  if 
    \begin{displaymath}
      \| S_{\mu,\{V^{0}, V^{1}\}}(t)\|_{\mathcal{L}(M^{p,\ell}(\R^N))}
      \leq C e^{a t}
    \end{displaymath}
    then for $1\leq p,q\leq \infty$, $0<s\leq \ell\leq \ell_0$ such that
    $\frac{s}{q}\leq \frac{\ell}{p}$ and $\eps>0$ we have
    \begin{displaymath}
      \| S_{\mu,\{V^{0}, V^{1}\}}(t) \|_{\mathcal{L}(M^{p,\ell}(\R^N), M^{q,s}(\R^N))}
      \leq
      \frac{Ce^{ (a +\eps ) t }}{t^{\frac{1}{2\mu}(\frac{\ell}{p}-\frac{s}{q})}},
      \quad t>0 . 
    \end{displaymath} 
\end{remark}

\appendix

\section{Proof of Lemma \ref{lem:fixed-point-lem}.}
\label{app:technicalities}

As noted in Section \ref{sec:perturbed_equation_Morrey}, Lemma
\ref{lem:fixed-point-lem} applies to fractional powers of
uniformly  elliptic $2m$ order operator of the form
\begin{displaymath}
    A_0 = \sum_{|\zeta|= 2m} (-1)^{m} a_\zeta \partial^\zeta
    \quad
\text{ with constant real coefficients } \ a_{\zeta} .
\end{displaymath}
Hence, we  adopt below all the notations and techniques in \cite[Sections
6 and 7]{C-RB-morrey_linear_perturbation}.
We write
$X^{\gamma}=M^{p,\ell}(\R^{N})$ (or $\mathcal{M}^{\ell}(\R^N)$ if $p=1$) with
\begin{equation}\label{eq:parameters_4_Morrey}
(p,\ell) \longmapsto \gamma= \gamma (p,\ell)  = \left(\frac{1}{p},\frac{\ell}{2m\mu p}\right) \in \mathbb{J}
\end{equation}
with
\begin{displaymath}
     \mathbb{J} =  \mathbb{J}_{*}  \cup \{(0,0)\},  \qquad   \mathbb{J}_{*} =\{(\gamma_1,\gamma_2)\in (0,1]\times
    \Big(0,\frac{N}{2m\mu}\Big] \colon \ \frac{\gamma_2}{\gamma_1}
    \leq \frac{N}{2m\mu}\}
\end{displaymath}
which is a planar triangle with vertices $(0,0)$, $(0,1)$ and
$\left(1,\frac{N}{2m\mu}\right)$, so all points in
$\gamma= (\gamma_1,\gamma_2) \in \mathbb{J}_{*}$, have slopes $0\leq
\frac{\gamma_2}{\gamma_1} =\frac{\ell}{2m\mu p}\leq
  \frac{N}{2m\mu}$. Also notice that for $p=\infty$ all Morrey spaces $M^{\infty,\ell}(\R^{N})$
  are equal to $L^{\infty}(\R^{N})$ and they are mapped by
  (\ref{eq:parameters_4_Morrey}) into
  $(0,0)$. For any $(\gamma_1,\gamma_2) \in
  \mathbb{J}_{*}$ there exist a unique $1\leq p<\infty$
  and $0<\ell\leq N$
  such that $X^{\gamma}=M^{p,\ell}(\R^{N})$.

With these notations we have that the admissible potentials in
(\ref{eq:for-thm-with-2potentials}) satisfy
    \begin{displaymath}
      \text{
        $V^{i} \in X^{\gamma^i}$,  \quad  $\gamma^i
        =\left(\frac{1}{p_{i}},\frac{\ell_{i}}{2m\mu p_{i}}\right) \in
        \mathbb{J}$}.
          \end{displaymath}
and for  $\alpha \in \mathbb{J}$, the multiplication
    operator defined by $V^{i}$ in Morrey spaces is linear and bounded
    \begin{displaymath}
    V^{i}: X^{\alpha} \to
      X^{\beta_{i}}, \quad \beta_{i} \in \mathbb{J}, \quad \beta_{i} = \alpha +
      \gamma^{i}
    \end{displaymath}
    provided      $\alpha_{1}+ \gamma_{1}^{i} \leq 1$, see \cite[Lemma
    7.1]{C-RB-morrey_linear_perturbation}. Notice in particular that
    if $V^{i}$  is bounded, then $\gamma^{i}=0$ and $\beta_{i}=
    \alpha$.

Given $p,\ell$ and $u_0$ as in the statement we fix
$\gamma=\gamma(p,\ell)$ accordingly to
(\ref{eq:parameters_4_Morrey}).

From \cite[proof of part (i) in Theorem
7.5]{C-RB-morrey_linear_perturbation} we see that if
$\alpha, \beta_i\in \mathbb{J}$ are given by
  \begin{displaymath}
    \alpha=\begin{cases}\gamma & \text{ if } \gamma_1\leq \theta \\
    (\theta, \frac{\gamma_2}{\gamma_1} \theta) & \text{ if } \gamma_1>
    \theta
    \end{cases} \quad \text{ for } \ \theta:=\min\{1-\gamma_1^0,
    1-\gamma_1^1\}, \qquad \beta_i=\alpha+\gamma^i \ \text{ for } \ i=0,1
\end{displaymath}
then we have that the operators of multiplication by $V^{i}$ are
admissible in these sense of  \cite[Definition
6.3(i)]{C-RB-morrey_linear_perturbation} applied with
$S(t)=S_\mu(t)$, and  satisfy $P_i
=V^i\in \mathscr{P}_{\beta_0,R}$, for some $R>0$. So $\{P_{0},
P_{1}\}$, that abusing of the notations, we identify with $\{V^{0},
V^{1}\}$ is an admissible pair for the abstract results in
\cite[Section 6]{C-RB-morrey_linear_perturbation}, see the proof of \cite[Theorem
7.5]{C-RB-morrey_linear_perturbation}.

Also $\gamma$ belongs to the set
$\mathcal{E}_{\alpha}=\{\gamma\in\mathbb{J}\colon \alpha_2\leq
\gamma_2 < \alpha_2+1 \text{ and } \frac{\alpha_2}{\alpha_1} \leq
\frac{\gamma_2}{\gamma_1} \}$, see \cite[Step 3 in the proof of Theorem
7.5]{C-RB-morrey_linear_perturbation}. This set is the corresponding
for Morrey spaces to the set
$\mathcal{E}_\alpha$ in the abstract result  \cite[Theorem
6.5]{C-RB-morrey_linear_perturbation}.

With these settings, since $\alpha$ above corresponds by (\ref{eq:parameters_4_Morrey}) to some Morrey
space with parameters $w,\kappa$ as
in (\ref{eq:omega-and-kappa}),
then the set $\mathcal{K}_{T,K_0,\theta}^{p,\ell,w,\kappa}$ in
(\ref{eq:space-for-fixed-point}) is nothing but the set appearing in
the proof of Theorem 6.5 in   \cite{C-RB-morrey_linear_perturbation} where
existence of solutions is derived. So from this result for $T>0$,
choosing $K_0$ and $\theta$ large enough we get the existence of
$S_{\mu,\{V^{0}, V^{1}\}}(t)u_0$ on $(0,T]$ as the unique fixed point
of
\begin{displaymath}
  \varphi\mapsto  {\mathcal F}(\varphi,u_0) (t)
= S_{\mu}(t) u_0 +  \int_0^t S_{\mu}(t-\tau)  V^{0} \varphi(\tau) \,
d\tau +  \int_0^t S_{\mu}(t-\tau)  V^{1} \varphi(\tau) .
\end{displaymath}
Similar results can be obtained for three or more admissible
potentials from the abstract results in  \cite[Section
6]{C-RB-morrey_linear_perturbation}. Observe that the set
$\mathcal{E}_{\alpha}$ only depend on the space $X^{\alpha}$ which is
the same for all perturbations.

The proof of each part of Lemma \ref{lem:fixed-point-lem}
corresponds precisely to the specific choice of a third perturbation
with a suitable admissible potential combined with  some further results in
\cite{C-RB-morrey_linear_perturbation}. Among these, we will use in an
essential way \cite[Theorem 6.8]{C-RB-morrey_linear_perturbation} that
states that perturbations can be applied successively in any order and
the resulting fixed points are unchanged.

\noindent {\bf Proof of part (i).}
Now we take $V^{2}=c$ as the third potential. Then the arguments above
give that $\{S_{\mu,\{V^{0}, V^{1}, c\}}(t)\}_{t\geq0}$ is the unique fixed
point of  $\varphi\mapsto {\mathcal F}(\varphi,u_0)$ in
$\mathcal{K}_{T,K_0,\theta}^{p,\ell,w,\kappa}$ (for large enough
$\theta$, $K_0$).

Now, \cite[Corollary 6.14]{C-RB-morrey_linear_perturbation} concludes
that $S_{\mu,\{V^{0}, V^{1}, c\}}(t) u_{0} =
e^{ct} S_{\mu,\{V^{0}, V^{1}\}}(t) u_{0}$. So, part (i) is
proved.

\medskip
\noindent {\bf Proof of part (ii).}
In this case, from \cite[Theorem 6.8]{C-RB-morrey_linear_perturbation}
we obtain
 $S_{\mu,\{V^{0}, V^{1}\}}(t) u_{0} =
 \Big(S_{\mu,V^{1}}\Big)_{ V^{0}} (t) u_{0}$ which is precisely
 what part (ii) states.

\medskip
\noindent {\bf Proof of part (iii).}
If $V^{0}\in L^\infty(\R^N)$  then we take in
(\ref{eq:for-thm-with-2potentials})   $\ell_{0}=\ell_{1}$, see Remark
\ref{rem:adding_potentials}, and then
from \cite[Theorem 6.8]{C-RB-morrey_linear_perturbation}
we obtain that
 $S_{\mu,\{V^{0}, V^{1}\}}(t) u_{0} =
 \Big(S_{\mu,V^{0}}\Big)_{ V^{1}} (t) u_{0}$ which is precisely
 what part (iii) states.

\medskip
\noindent {\bf Proof of part (iv).}
If $\tilde{V}^0\in M^{p_0,\ell_0}(\R^N)$ then from the results quoted
above
\begin{displaymath}
 S_{\mu,\{V^{0}, \tilde{V}^0 , V^{1}\}}(t) u_{0} = \Big(S_{\mu,\{V^{0}, V^1 \}}\Big)_{ \tilde{V}^{0}} (t) u_{0}
\end{displaymath}
while at the same time
\begin{displaymath}
 S_{\mu,\{V^{0}, \tilde{V}^0 , V^{1}\}}(t) u_{0} =
 \Big(S_{\mu,\{V^{0}, \tilde{V}^0 \}}\Big)_{ V^{1}} (t) u_{0}
\end{displaymath}
 and also $S_{\mu,\{V^{0}, \tilde{V}^0 \}}(t) = S_{\mu, V^{0}+
   \tilde{V}^0  }(t)$ see Remark
\ref{rem:adding_potentials}, so
\begin{displaymath}
 S_{\mu,\{V^{0}, \tilde{V}^0 , V^{1}\}}(t) u_{0} =
 \Big(S_{\mu, V^{0} + \tilde{V}^0 }\Big)_{ V^{1}} (t) u_{0} .
\end{displaymath}
Equating the first and third expressions above proves part (iv).

\medskip
\noindent {\bf Proof of part (v).}
    If $\tilde{V}^1\in M^{p_1,\ell_1}(\R^N)$  then from the results quoted
above
\begin{displaymath}
 S_{\mu,\{V^{0}, V^{1}, \tilde{V}^1 \}}(t) u_{0} =
 \Big(S_{\mu,\{V^{0}, V^1 \}}\Big)_{ \tilde{V}^{1}} (t) u_{0} .
\end{displaymath}
At the same time
\begin{displaymath}
 S_{\mu,\{V^{0}, \tilde{V}^1 , V^{1}\}}(t) u_{0} =
 \Big(S_{\mu,\{V^{1}, \tilde{V}^1 \}}\Big)_{ V^{0}} (t) u_{0}
\end{displaymath}
and also $S_{\mu,\{V^{1}, \tilde{V}^1 \}}(t) = S_{\mu, V^{1}+
   \tilde{V}^1  }(t)$ see Remark
 \ref{rem:adding_potentials}. Hence
 \begin{displaymath}
 S_{\mu,\{V^{0}, \tilde{V}^1 , V^{1}\}}(t) u_{0} =
 \Big(S_{\mu, V^{1}+ \tilde{V}^1 }\Big)_{ V^{0}} (t) u_{0} =
 S_{\mu,\{V^{0}, V^{1}+ \tilde{V}^1\}} (t) u_{0}
\end{displaymath}
Equating the first and third expressions above proves part (v). $\blacksquare$

\section{A general result on the exponential type}
\label{sec:result-exponential-type}

Assume $X$ is a Banach space and     $\{S(t)\}_{t\geq0}\subset
\mathcal{L}(X)$ is a semigroup such that     $\sup_{t\in(0,\delta)}\|S(t)\|_{\mathcal{L}(X)}<\infty$ for some
$\delta>0$. The exponential type of the semigroup is
\begin{displaymath}
  \omega(X) = \sup \Omega(X)
\end{displaymath}
where 
\begin{displaymath}
\Omega(X) \mydef\{\omega \in \R \colon   \|S (t)\|_{\mathcal{L}(X)} \leq
 M e^{-\omega t}, \ \mbox{for all} \ t\geq 0,  \ \mbox{for some $M\geq 1$}\} . 
\end{displaymath}

\begin{lemma}\label{lem:same_exponential_type}
  
    Assume for $X\in\{Y,Z\}$ that  $X$ is a Banach space and     $\{S(t)\}_{t\geq0}\subset
\mathcal{L}(X)$ is a semigroup such that     $\sup_{t\in(0,\delta)}\|S(t)\|_{\mathcal{L}(X)}<\infty$ for some
$\delta>0$. 

If $Z\subset Y$ continuously, and for any $t>0$ we have $S(t)\in \mathcal{L}(Y,Z)$ then 
\begin{displaymath}
\omega(Y) = \omega(Z) . 
\end{displaymath}
\end{lemma}

\begin{proof} 
By   \cite[Proposition 2.2, p. 251]{Engel-Nagel_2000} 
\begin{equation}\label{eq:expressions-for-Omega(X)}
    \sup\Omega(X) = \sup_{t>0} \frac{-\ln{\|S(t)\|_{\mathcal{L}(X)}}}{t}
    = \lim_{t\to\infty} \frac{-\ln{\|S(t)\|_{\mathcal{L}(X)}}}{t}  \quad  \text{ for } \ X\in\{Y,Z\}.
\end{equation}

Since $Z\subset Y$, if $u_0\in Y$ and $t>1$ then, for some constant
$c>0$, 
$\|S(t) u_0\|_{Y}
\leq c \|S(t) u_0\|_{Z}
\leq
c \|S(t-1)\|_{\mathcal{L}(Z)}  \|S(1) \|_{\mathcal{L}(Y,Z)} \| u_0\|_{Y}$, and so
\begin{displaymath}
\|S(t)\|_{\mathcal{L}(Y)}
\leq
c \|S(t-1)\|_{\mathcal{L}(Z)}  \|S(1) \|_{\mathcal{L}(Y,Z)} .
\end{displaymath}
Hence 
\begin{displaymath}
\frac{-\ln{\|S(t)\|_{\mathcal{L}(Y)}}}{t}
\geq \frac{-\ln(c \|S(1) \|_{\mathcal{L}(Y,Z)})}{t} +
\frac{-\ln\|S(t-1)\|_{\mathcal{L}(Z)}}{t-1} \cdot \frac{t-1}{t}, \quad t>1,
\end{displaymath}
and then  (\ref{eq:expressions-for-Omega(X)}) gives  $\omega(Y) \geq \omega(Z)$.

If $u_0\in Z$ and $t>1$ then
$\|S(t) u_0\|_{Z}
\leq
\|S(1)\|_{\mathcal{L}(Y,Z)}
\|S(t-1)\|_{\mathcal{L}(Y)} c \|u_0\|_{Z}$, 
and so
\begin{displaymath}
\|S(t) \|_{\mathcal{L}(Z)}
\leq c
\|S(1)\|_{\mathcal{L}(Y,Z)}
\|S(t-1)\|_{\mathcal{L}(Y)}  .
\end{displaymath}
Hence
\begin{displaymath}
\begin{split}
\frac{-\ln{\|S(t) \|_{\mathcal{L}(Z)}}}{t}
\geq \frac{-\ln{(c \|S(1)\|_{\mathcal{L}(Y,Z)})}}{t} +
\frac{-\ln{\|S(t-1)\|_{\mathcal{L}(Y)}}}{t-1}
\cdot \frac{t-1}{t} , \quad t>1 ,
\end{split}
\end{displaymath}
and then  (\ref{eq:expressions-for-Omega(X)}) gives $\omega(Z) \geq \omega(Y)$.
\end{proof}

In particular, for  the fractional Schrödinger semigroup in the space of Morrey measures, by Lemma
\ref{lem:same_exponential_type} with $Y=M^\ell(\R^N)$,
$Z=M^{1,\ell}(\R^N)$, we get the following consequence.

  \begin{corollary} \label{cor:exponential_type_4_measures}
Assume  $0<\mu\leq 1$ and $V\in M^{p_{1}, \ell_{1}}(\R^N)$ an
admissible potential as in (\ref{eq:admisible_potential}). 

Then for any $0<\ell\leq \ell_{1}$ we have
\begin{displaymath} 
  \omega_{\ell} = \omega_{1,\ell} .
\end{displaymath}
\end{corollary}

\section{Schrödinger semigroup and exponential decay in uniform spaces}
\label{sec:expon-decay-uniform}

For  $1\leq p_{0}<\infty$ the  uniform space $L_U^{p}(\R^N)$,    is composed
of the functions $\phi \in L_{loc}^{p} (\R^N)$ such that there exists
$C>0$ such that for all $x_0\in\R^N$
\begin{displaymath}
\int_{B(x_0,1)} |\phi|^{p} \leq C
\end{displaymath}
endowed with the norm
\begin{displaymath}
\|\phi \|_{L_U^{p} (\R^N)}=\sup_{x_0\in\R^N}\|\phi\|_{L^{p}
  (B(x_0,1))} .
\end{displaymath}
For $p=\infty$ we define  $L_U^\infty(\R^{N})=L^\infty(\R^{N})$.

Also for $p=1$ we have the space of uniform measures
$\mathcal{M}_{U}(\R^{N})$ consisting of Radon measures $\mu$ such that
\begin{displaymath}
  |\mu|(B(x_0,1))  \leq C
\end{displaymath}
endowed with the norm
\begin{displaymath}
\|\mu \|_{\mathcal{M}_U (\R^N)}=\sup_{x_0\in\R^N}   |\mu|(B(x_0,1)) . 
\end{displaymath}

The dotted uniform  space $\dot L^{p}_{U}(\R^N)$ or  $\dot{\mathcal{M}}_{U}(\R^{N})$ is defined
as the subspace of the uniform spaces in which translations
$\tau_y\phi(\cdot)=\phi(\cdot-y)$ in $\R^N$ are continuous in the
sense that
\begin{displaymath}
\tau_y \phi -\phi \to 0 \ \text{ as } \ y\to 0, \quad \text{ in
  $L_{U}^{p}(\R^N)$ or $\mathcal{M}_{U}(\R^{N})$} . 
\end{displaymath}

Our  goal
in this section is to show, on the one hand, that for a 
bounded potential $V\in L^{\infty}(\R^{N})$, the integral
representation of the solutions of the evolution problem 
\begin{displaymath}
    \begin{cases}
    u_{t} + (-\Delta)^{\mu} u =V(x) u,  & x \in \R^{N}, \quad t>0, \\
    u(0,x)=u_0(x), & x\in \R^N,
  \end{cases}
\qquad \text{ where $0< \mu \leq 1$},
\end{displaymath}
that is, functions satisfying the  variations of constants formula
     \begin{equation} \label{eq:VCF_uniform}
       u(t) = S_{\mu}(t) u_{0} + \int_{0}^{t}  S_{\mu}(t-s) V u(s) \,
       ds, \quad  t>0
\end{equation}
define a semigroup in  the uniform spaces $L^{p}_{U}(\R^{N})$, for $1\leq p\leq \infty$,  by $u(t)= S_{\mu, V}(t)u_{0}$.

On the other hand, we prove that for a dissipative  potential $0\geq
V^{0}\in L^{\infty}(\R^{N})$, the Arendt and Batty condition
(\ref{eq:Arendt-Batty-condition}) also gives exponential decay in
the uniform spaces $L^{p}_{U}(\R^{N})$, $1\leq p \leq \infty$.

Recall that the semigroup  $S_{\mu}(t) $ was considered in uniform
spaces, in  Section 3
in \cite{C-RB-morrey_linear_perturbation} and it was proved that
\begin{displaymath}
        \|S_{\mu}(t) \|_{\mathcal{L}(L^{p}_U(\R^N), L^{q}_U(\R^N))}
      \leq C \big( 1+
      \frac{1}{t^{\frac{N}{2 \mu}(\frac{1}{p}-\frac{1}{q})}}\big),
      \quad t>0  .
\end{displaymath}
In  Lebesgue spaces, it was proved in  Section 6 in
\cite{C-RB-scaling1} that
\begin{displaymath}
        \|S_{\mu}(t) \|_{\mathcal{L}(L^{p}(\R^N), L^{q}(\R^N))}
      \leq
      \frac{C}{t^{\frac{N}{2 \mu}(\frac{1}{p}-\frac{1}{q})}},
      \quad t>0  .
\end{displaymath}
For Morrey spaces see Section \ref{sec:homogeneous-lin-eq} in this
paper, see in particular  (\ref{eq:estimates_Mpl-Mqs}).

\begin{proposition}
\label{prop:semigroup_4_boundedpotential}

  If $ V\in L^\infty(\R^N)$ and $0<\mu \leq 1$ then
  (\ref{eq:VCF_uniform}) defines   an order preserving  semigroup in $L^p_U(\R^N)$, $1\leq p\leq \infty$, given by
\begin{displaymath}
S_{\mu, V}(0)u_0=u_0 \ \text{ and } \ S_{\mu, V}(t)u_{0} =u(t) ,
\end{displaymath}
and for any $1\leq p\leq q\leq \infty$
\begin{displaymath}
(0,\infty)\times L^p_U(\R^N) \ni (t,u_0)\to S_{\mu, V}(t) \in L^q_U(\R^N)
\ \text{ is continuous,}
\end{displaymath}
\begin{equation}\label{eq:perturbed-in-LpU-2}
      \|S_{\mu, V}(t) \|_{\mathcal{L}(L^{p}_U(\R^N), L^{q}_U(\R^N))}
      \leq C e^{\omega t}\big( 1+
      \frac{1}{t^{\frac{N}{2\mu}(\frac{1}{p}-\frac{1}{q})}}\big),
      \quad t>0
\end{equation}
for some constants $C$, $\omega$ and for  $u_0\in L^p_U(\R^N)$
\begin{equation}\label{eq:perturbed-in-LpU-3}
\lim_{t\to0^+} \|S_{\mu, V}(t) u_0 - S_{\mu}(t) u_0\|_{L^{p}_U(\R^N)}=0.
\end{equation}
For $p=1$ the result holds true replacing $L^1_U(\R^N)$ by  $\mathcal{M}_U(\R^N)$.

  If $u_{0}\in L^{p}(\R^{N})$ then all these results remain true
  replacing the uniform spaces  $L^p_U(\R^N)$ by Lebesgue ones
  $L^p(\R^N)$. If $p=1$ we can also take  $u_{0}\in
  \mathcal{M}_{BTV}(\R^{N})$.

If  $u_{0}\in M^{p,\ell}(\R^{N})$  or  $u_{0}\in
\mathcal{M}^{\ell}(\R^{N})$ then the semigroup above coincides
with the one in Section \ref{sec:perturbed_equation_Morrey}.

\end{proposition}
\begin{proof}
The proof is rather standard so we give the main arguments. Denote by $X$ either one of the spaces $L^{p}_{U}(\R^{N})$,
$\mathcal{M}_U(\R^N)$,  $L^{p}(\R^{N})$, $\mathcal{M}_{BTV}(\R^N)$  or
$M^{p,\ell}(\R^{N})$,  $\mathcal{M}^{\ell}(\R^N)$. Observe then  that for
  $u_{0}\in X$ and $T_{0} \| V\|_{L^{\infty}} <1$ the
  mapping
  \begin{displaymath}
    \mathcal{F}(u,u_{0})(t) =  S_{\mu}(t) u_{0} + \int_{0}^{t}  S_{\mu}(t-s) V u(s) \,
    ds
  \end{displaymath}
  defines a contraction in $L^{\infty}((0,T_{0} ),X)$ as
  \begin{displaymath}
  \| \mathcal{F}
    (u, u_{0})(t)\|_{X} \leq \|u_{0}\|_{X} + \int_{0}^{t} \|
    V\|_{L^{\infty}} \|u(s)\|_{X} \, ds
  \end{displaymath}
  so $\mathcal{F} (u) \in L^{\infty}((0,T_{0} ), X)$
  and if   $u_{0},v_{0} \in X$
  \begin{displaymath}
    \begin{split}
      \| \mathcal{F} (u, u_{0})(t) - \mathcal{F} (v, v_{0})(t)\|_{X}
&      \leq c\|u_{0}-v_{0}\|_{X} + \int_{0}^{t} \| V\|_{L^{\infty}}
\|u(s)- v(s)\|_{X} \, ds
\\ &
\leq c\|u_{0}-v_{0}\|_{X} + T_{0} \|
      V\|_{L^{\infty}} \sup_{0\leq t \leq T_{0} }\|u(s) - v(s)\|_{X} .
    \end{split}
  \end{displaymath}

Since the contraction time is independent of the initial data, it is
easy, by solving the fixed point in $[T_{0} ,2T_{0}]$ with initial
time $u(T_{0})$ and so on, to construct a fixed point of $\mathcal{F}$
    in $L^{\infty}((0,kT_{0} ), X)$ for any $k\in \N$.
    The  uniqueness of the fixed point for $\mathcal{F}$ in
  $L^{\infty}((0,T ), X
  )$ for $T>0$, follows from
  Gronwall's lemma since for two such fixed points and $t>0$
  \begin{displaymath}
    \| u(t) - v(t)\|_{X} \leq       \int_{0}^{t}  \| V\|_{L^{\infty}} \|u(s)- v(s)\|_{X} \,
    ds .
  \end{displaymath}

  With all these, the semigroup property follows from
  uniqueness. Also, estimate (\ref{eq:perturbed-in-LpU-3}) follows
  since the solution is in $L^{\infty}((0,T ), X)$.  The estimate
  (\ref{eq:perturbed-in-LpU-2}) follows from the one for $S_{\mu}(t)$
  in $X$ and $V\in L^{\infty}(\R^{N})$. Also, if   $u_{0}\in M^{p,\ell}(\R^{N})$   then the fixed point above is the
same as the one in Lemma \ref{lem:fixed-point-lem}.

Finally, the  order preserving property is like in Proposition
\ref{prop:order_preserving_bounded_perturbation}, since  $S_{\mu}(t)$
is order preserving in $X$.
\end{proof}

Now we prove that if $V$ is dissipative,   the semigroup $S_{\mu,
  V}(t)$ has an associated integral kernel for  initial data in $L^{1}_{U}(\R^{N})$.

\begin{lemma}
  \label{lem:kernel-for-SmuV0-in-L1U}
Assume  $0<\mu\leq 1$ and $0\geq V\in L^\infty(\R^N)$.

Then there is a nonnegative kernel,  $0\leq k_{\mu,V}(t,\cdot,\cdot)$, such that for any $u_0\in  L^1_U(\R^N)$
\begin{equation}\label{eq:integral-formula-foir-mu-to-prove}
S_{\mu, V}(t) u_0 = \int_{\R^N} k_{\mu,V} (t,\cdot,y) u_0 (y) \, dy .
\end{equation}
  
\end{lemma}
\begin{proof}
We proceed in four steps.

\smallskip

\noindent{\bf Step 1.} 
In this step we show that if $0\leq u_0 \in L^1_U(\R^N)$ then there is
a sequence $\{\phi_{n}\}\subset L^1(\R^N)$ satisfying 
\begin{equation}\label{eq:-2-to-use-below}
\phi_{n} \leq \phi_{n+1} \leq u_0, \quad n\in\N 
\end{equation}
and having pointwise in $\R^N$ limit
\begin{equation}\label{eq:-1-to-use-below}
\lim_{n\to \infty}\phi_n = u_0 .
\end{equation}

For this it suffices to take $\phi_{n}=u_0 \chi_{B(0,n)}$, $n\in\N$,
where $\chi_{B(0,n)}$ is the characteristic function of the ball
$B(0,n)\subset \R^N$ of radius $n$ around zero. 

\smallskip

\noindent{\bf Step 2.}  In this step given any $0\leq u_0\in
L^1_U(\R^N)$ and taking $\{\phi_n\}$ as in Step 1 above we show for
each $t>0$ that  $S_{\mu,V}(t) \phi_n\to S_{\mu,V}(t) u_0$ as $n\to \infty$ pointwise in $\R^N$. 

Due to monotonicity of $\{\phi_n\}$ and order preserving properties of
the perturbed semigroup, see Proposition  \ref{prop:semigroup_4_boundedpotential}, 
we have that the functions $v_{n}(t) \mydef S_{\mu,V}(t)\phi_{n}$
satisfy for $t\geq 0$, 
\begin{displaymath}
v_{n}(t)  \leq v_{n+1}(t)\leq S_{\mu,V}(t) u_0, \quad n\in\N .
\end{displaymath}
Hence for any $t>0$, the pointwise  limit 
\begin{equation}\label{eq:0-to-use-below}
v(t)\mydef \lim_{n\to\infty} v_{n}(t), \quad t>0 
\end{equation}
exists in $\R^{N}$. 

We recall from Proposition \ref{prop:semigroup_4_boundedpotential} 
that $S_{\mu,V}(t)\phi_{n}$ satisfies variation of constants formula
(\ref{eq:VCF_uniform}), so that we have  
\begin{equation}\label{eq:0.15-to-use-below}
    v_{n}(t) = S_{\mu}(t)\phi_{n}  + \int_{0}^{t} S_{\mu}(t-s) V v_{n}(s) \, ds  , \quad t>0 .
  \end{equation}

We also recall that $S_{\mu}(t)$ has for any $t>0$ kernel $k_{\mu}(t,\cdot,\cdot)$ such that  
\begin{displaymath}
S_\mu(t) \phi = \int_{\R^N} k_\mu(t,\cdot,y) \phi (y) \, dy, \quad \phi\in L^1_U(\R^N)  
\end{displaymath}
see \cite[Section 3]{C-RB-morrey_linear_perturbation}, see also 
\cite{2019Bogdan_fract_laplac_hardy,2017BonforteSireVazquez:_optimal-existence},
\cite[Section 6]{C-RB-scaling3-4}.

Using (\ref{eq:-2-to-use-below}), (\ref{eq:-1-to-use-below}) and Lebesgue's monotone convergence theorem,
we thus have the  pointwise in $\R^N$ limit
\begin{displaymath}
\lim_{n\to\infty}  S_\mu(t) \phi_n = \lim_{n\to\infty} \int_{\R^N} k_\mu(t,\cdot,y) \phi_n (y) \, dy 
= \int_{\R^N} k_\mu(t,\cdot,y) u_0 (y) \, dy  = S_\mu(t) u_0
\end{displaymath}
From this, (\ref{eq:0-to-use-below}) and (\ref{eq:0.15-to-use-below})
we have the  pointwise in $\R^N$ limit
\begin{equation}\label{eq:1st-to-use-below}
\lim_{n\to \infty}  \int_{0}^{t} S_{\mu}(t-s) V v_{n}(s) \, ds  =
v(t) - S_{\mu}(t)u_{0}  , \quad t>0. 
\end{equation}

On the other hand, using the  kernel $k_{\mu}(t,\cdot,\cdot)$, we get in $\R^{N}$ 
\begin{displaymath}
   \int_{0}^{t} S_{\mu}(t-s) V v_{n}(s) \, ds =  \int_{0}^{t}
   \int_{\R^{N}} k_{\mu} (t-s,\cdot,y ) V(y) v_{n}(s)(y) \, dy \, ds , \quad t>0 .
 \end{displaymath}
Now from  (\ref{eq:0-to-use-below}) we have for arbitrarily fixed $0<s<t$ and $x,y\in \R^N$ 
 \begin{displaymath}
   \lim_{n\to\infty} k_{\mu} (t-s,x,y ) V(y) v_{n}(s) (y) = k_{\mu} (t-s,x,y ) V(y) v(s) (y) ,
 \end{displaymath}
then  Lebesgue's monotone convergence theorem gives the  pointwise
limit in $\R^N$
\begin{equation}\label{eq:3rd-to-use-below}
 \lim_{n\to\infty} \int_{0}^{t}
   \int_{\R^{N}} k_{\mu} (t-s,\cdot,y ) V(y) v_{n}(s)(y) \, dy \, ds 
  = \int_{0}^{t}    \int_{\R^{N}} k_{\mu} (t-s,\cdot,y ) V(y) v(s)
  (y) \, dy \, ds . 
 \end{equation}

Since the right hand side above equals 
\begin{displaymath}
\int_{0}^{t} S_{\mu}(t-s) V v(s) \, ds   , 
\end{displaymath}
 from  (\ref{eq:1st-to-use-below}) and (\ref{eq:3rd-to-use-below}) we get in $\R^N$
\begin{displaymath}
  \int_{0}^{t} S_{\mu}(t-s) V v(s) \, ds  =v(t) -  S_{\mu}(t)u_{0} , \quad t>0 
\end{displaymath}
that is, $v$ satisfies the variation of constants formula. 

Now we prove that $v(t)$ equals 
$u(t)=S_{\mu,V}(t) u_0$ from Proposition
\ref{prop:semigroup_4_boundedpotential}. For this, from the  variation of constants formula
we obtain  
\begin{displaymath}
\|u(t)-v(t)\|_{L^1_U(\R^N)} \leq \int_0^t  \|V\|_{L^\infty(\R^N)} \|u(s)-v(s)\|_{L^1_U(\R^N)} , \ t>0 .
\end{displaymath}
and,  as in the proof of Proposition
\ref{prop:semigroup_4_boundedpotential}, 
we conclude via Gronwall's lemma that $u=v$.
 
\smallskip

\noindent{\bf Step 3.} In this step we get (\ref{eq:integral-formula-foir-mu-to-prove}) for $0\leq u_0\in L^1_U(\R^N)$.

First, from  (\ref{eq:perturbed-in-LpU-2})
using the embedding $L^1(\R^N)\subset L^1_U(\R^N)$, we see that $S_{\mu, V}(t)\in \mathcal{L}(L^1(\R^N),L^\infty(\R^N))$.
Then from  \cite[Theorem 1.3]{1994_Arendt-Bukhvalov}  for each
$t>0$ there is a kernel, which we denote by $k_{\mu,V}(t,\cdot,\cdot)$,  such that 
\begin{displaymath}
S_{\mu, V}(t) u_0  = \int_{\R^N} k_{\mu,V} (t,\cdot,y) u_0 (y) \, dy, \quad u_0\in   L^{1}(\R^N) .
\end{displaymath}
Now we prove that this expression holds also for $0\leq u_0\in
L^1_U(\R^N)$.  

For this observe that since the semigroup in Proposition 
\ref{prop:semigroup_4_boundedpotential} is order preserving then 
$k_{\mu,V}(t,\cdot,\cdot)\geq 0$. 
Therefore, given any $0\leq u_0\in L^1_U(\R^N)$ and taking the
sequence $\{\phi_n\}$ as in Step 1 above, we have
\begin{displaymath}
S_{\mu, V}(t) \phi_n  = \int_{\R^N} k_{\mu,V} (t,\cdot,y) \phi_n (y) \, dy , \quad n \in \N .
\end{displaymath}
Then, from Step 2, passing to the limit as $n\to\infty$ pointwise in
$\R^N$ and using monotone convergence in the integral, we get 
\begin{displaymath}
S_{\mu, V}(t) u_0  = \int_{\R^N} k_{\mu,V} (t,\cdot,y) u_0
(y) \, dy . 
\end{displaymath}

\smallskip

\noindent{\bf Step 4.} In this step we show that (\ref{eq:integral-formula-foir-mu-to-prove}) holds for any $u_0\in L^1_U(\R^N)$.

Actually, if  $u_0\in L^1_U(\R^N)$ we write  $u_{0}= u_0^{+}
-u_0^{-}$, the  positive and negative parts. Since by Proposition
\ref{prop:semigroup_4_boundedpotential}, 
 $S_{\mu, V}(t)$ is linear and Step 3 above,  we get
\begin{displaymath}
S_{\mu, V}(t) u_0 = 
S_{\mu, V}(t)u_0^{+} - S_{\mu, V}(t) u_0^{-} = \int_{\R^N} k_{\mu,V} (t,\cdot,y) u_0^{+} (y) \, dy -
\int_{\R^N} k_{\mu,V} (t,\cdot,y) u_0^{-} (y) \, dy 
\end{displaymath}
and this equals $\int_{\R^N} k_{\mu,V} (t,\cdot,y) u_0 (y) \, dy$.
\end{proof}

According to Appendix \ref{sec:result-exponential-type},  we define the exponential type of the semigroup
$\{S_{\mu, V} (t)\}_{t\geq 0} $ in $X= L^p_{U}(\R^N)$, $1\leq p\leq
\infty$, or even $X=\mathcal{M}_U(\R^N)$,  as
\begin{displaymath}
 \omega(X) \mydef   \sup\{a \in \R, \ \|S_{\mu, V} (t)\|_{\mathcal{L}(X)} \leq
 M e^{-at}, \ \mbox{for all} \ t\geq 0,  \ \mbox{for some $M\geq 1$} \}  ,
\end{displaymath}

\begin{theorem}

For $V\in L^{\infty}(\R^{N})$ we have
\begin{displaymath}
  \omega(\mathcal{M}_U) = \omega(L^p_U)  , \quad 1\leq p \leq \infty.
\end{displaymath}

In particular, if  $0\geq V\in L^\infty(\R^N)$ satisfies  the Arendt and Batty condition
(\ref{eq:Arendt-Batty-condition}), then  the semigroup $\{S_{\mu,
  V}(t)\}_{t\geq0}$ decays asymptotically in $L^p_U(\R^N)$ and in 
$\mathcal{M}_U(\R^N)$. 

\end{theorem}
\begin{proof}
  Since for $1\leq p < q\leq \infty$ we have $L^{q}_{U}(\R^{N})
  \subset L^{p}_{U}(\R^{N}) \subset \mathcal{M}_U(\R^N)$ and the
  estimates (\ref{eq:perturbed-in-LpU-2}), the result follows from
  Lemma \ref{lem:same_exponential_type}. 

If (\ref{eq:Arendt-Batty-condition}) holds, the exponential decay is
from the one in $L^{\infty}(\R^{N}) =
    L^{\infty}_U(\R^N)$. 
\end{proof}

\bibliographystyle{plain}

\end{document}